\documentclass[12pt,twoside,english]{extarticle}
\usepackage{lmodern}
\usepackage{helvet}
\usepackage[T1]{fontenc}
\usepackage[latin9]{inputenc}
\usepackage{geometry}
\usepackage{color}
\usepackage{babel}
\usepackage{float}
\usepackage{mathrsfs}
\usepackage{amsmath}
\usepackage{amsthm}
\usepackage{amssymb}
\usepackage{graphicx}
\usepackage{setspace}
\usepackage[unicode=true,pdfusetitle,
 bookmarks=true,bookmarksnumbered=false,bookmarksopen=false,
 breaklinks=true,pdfborder={0 0 0},pdfborderstyle={},backref=false,colorlinks=true]
 {hyperref}

\makeatletter

\providecommand{\tabularnewline}{\\}

\numberwithin{equation}{section}
\numberwithin{table}{section}
\numberwithin{figure}{section}
\usepackage[natbibapa]{apacite}
\theoremstyle{plain}
\newtheorem{assumption}{\protect\assumptionname}
\theoremstyle{plain}
\newtheorem{prop}{\protect\propositionname}[section]
\theoremstyle{definition}
\newtheorem{condition}{\protect\conditionname}
\theoremstyle{plain}
\newtheorem{thm}{\protect\theoremname}[section]
\theoremstyle{plain}
\newtheorem{cor}{\protect\corollaryname}[section]
\theoremstyle{definition}
\newtheorem{defn}{\protect\definitionname}[section]
\theoremstyle{plain}
\newtheorem{lyxalgorithm}{\protect\algorithmname}
\theoremstyle{plain}
\newtheorem{lem}{\protect\lemmaname}[section]

\usepackage{graphicx}
\usepackage{amsmath}
\usepackage{dcolumn}
\usepackage{listings}
\usepackage{color}
\usepackage{setspace}
\usepackage[normalem]{ulem}  % For using \uline
\definecolor{hellgelb}{rgb}{1,1,0.8}
\definecolor{colKeys}{rgb}{0,0,1}
\definecolor{colIdentifier}{rgb}{0,0,0}
\definecolor{colComments}{rgb}{1,0,0}
\definecolor{colString}{rgb}{0,0.5,0}

\usepackage{lmodern}
\usepackage[T1]{fontenc}
\usepackage{geometry}
\usepackage{fancyhdr}
\usepackage{amsthm}
\usepackage{thmtools}
\usepackage{amsmath}
\usepackage{amssymb}
\usepackage{esint}
\usepackage{multirow}
\usepackage{mathrsfs} 
\usepackage{textgreek}
\usepackage{amsfonts}
\usepackage{amssymb}
\usepackage{mathrsfs}

\usepackage[USenglish]{isodate}
\usepackage{chngcntr}

\numberwithin{equation}{section}
\numberwithin{table}{section}
\numberwithin{assumption}{section}
\counterwithout{figure}{section}
\counterwithout{table}{section}

\makeatother

  \providecommand{\algorithmname}{Algorithm}
  \providecommand{\assumptionname}{Assumption}

  \providecommand{\corollaryname}{Corollary}
  
  \providecommand{\definitionname}{Definition}

  \providecommand{\lemmaname}{Lemma}

  \providecommand{\propositionname}{Proposition}

  \providecommand{\theoremname}{Theorem}
 \providecommand{\corollaryname}{Corollary}
 \providecommand{\theoremname}{Theorem}
 
\usepackage{thmtools}

\newtheoremstyle{MyTheoremstyle}
  {\topsep} % Space above
  {\topsep} % Space below
  {} % Body font
  {} % Indent amount
  {\bfseries} % Theorem head font
  {.} % Punctuation after theorem head
  {.90em} % Space after theorem head
  {} % Theorem head spec (can be left empty, meaning `normal')
\theoremstyle{MyTheoremstyle} 
\theoremstyle{MyTheoremstyle} 
\theoremstyle{MyTheoremstyle} 
\theoremstyle{MyTheoremstyle} 
\theoremstyle{MyTheoremstyle}

\declaretheoremstyle[
    headfont=\bfseries,
    notefont=\normalfont,
    bodyfont=\itshape,
    headpunct=\newline,
    headformat={%
        \makebox{\NAME\ \NUMBER\ }{\NOTE}%
    },
]{theorem}

\newlength{\spacelength}
\declaretheoremstyle[
    headfont=\bfseries,
    notefont=\normalfont,
    bodyfont=\itshape,
    headpunct=\newline,
    headformat={%
        \makebox[0pt][l]{\NAME\ \NUMBER\ }\hskip-\spacelength{\NOTE}%
    },
]{theore}

 \usepackage{fancyhdr}
\newcommand{\simpleheading}{
\newgeometry{verbose,tmargin=2cm,bmargin=2cm,lmargin=2cm,rmargin=2cm,headheight=1cm,headsep=1cm,footskip=1cm}
\setlength{\headheight}{25pt}
\fancyhead{{\normalsize \textsc{}}}}

\usepackage{booktabs}
\usepackage{float}
\usepackage{graphicx}
\usepackage{epstopdf}
\usepackage{morefloats}
\usepackage[referable]{threeparttablex}
\usepackage{footnote}

\usepackage{setspace} 
\usepackage{graphicx,color}

\usepackage{listings}
\RequirePackage[T1]{fontenc}
\RequirePackage{ae,fancyvrb}
\DefineVerbatimEnvironment{Sinput}{Verbatim}{fontshape=sl}
\DefineVerbatimEnvironment{Soutput}{Verbatim}{}
\DefineVerbatimEnvironment{Scode}{Verbatim}{fontshape=sl}

\title{\bf Continuous Record Asymptotic Framework for Inference in Strucutral Change Models}
\author{%\hfill
\textsc{\textcolor{MyBlue}{Alessandro Casini}}\thanks{Department of Economics and Finance, University of Rome Tor Vergata, Via Columbia 2, Rome, 00133, IT. %Tel: + . Fax: +. 
Email: %\email{acasini@bu.edu} 
\texttt{\textcolor{MyBlue}{{alessandro.casini@uniroma2.it}}}.} 
\\
\small{\text{University of Rome Tor Vergata}}
\and
\textsc{\textcolor{MyBlue}{Pierre Perron}}\thanks{Department of Economics, Boston University, 270 Bay State Road, Boston, MA 02215, US. %Tel: + . Fax: +. 
Email: %\email{acasini@bu.edu} 
\texttt{\textcolor{MyBlue}{\mbox{perron@bu.edu}}}.} 
\\
\small{\text{Boston University}}
}

\date{\small{\today} \\} %\footnotesize{First Version: \printdate{28.10.2015}}}

 \makeatletter

\numberwithin{equation}{section}
\makeatother

\usepackage{babel}
\addto\captionsenglish{%
}

\usepackage{color}
\usepackage{xcolor}

\definecolor{MyRed}{rgb}{0.8,0,0}
\definecolor{MyBlue}{rgb}{0,0,0.7}
\definecolor{Green}{rgb}{0,0.5,0}
\definecolor{hellgelb}{rgb}{1,1,0.8}
\definecolor{colKeys}{rgb}{0,0,1}
\definecolor{colIdentifier}{rgb}{0,0,0}
\definecolor{colComments}{rgb}{1,0,0}
\definecolor{colString}{rgb}{0,0.5,0}

\definecolor{MyLightRed}{rgb}{2.2,0.2,0.4} % changed this
\definecolor{MyLightRed2}{rgb}{0.6,0.2,0.3} %amaranto
\definecolor{MyLightRed2temp}{rgb}{0.6,0.2,0.3}
\definecolor{MyLightRed3}{rgb}{0.8,0.1,0.1} %red
\definecolor{MyRed}{rgb}{0.7,0.0,0}
\definecolor{MyLigthBlue13}{rgb}{0,0.2,0.7}
 \definecolor{MyLigthBlack}{rgb}{0.2,0.25,0.3} % Light Black for font

\hypersetup{%
bookmarks=true
  pdftitle = {Title Here},
  pdfsubject = {},
  pdfkeywords = {Keywords},
  pdfauthor = {Alessandro Casini},}
\hypersetup{ 
  colorlinks = {true}, % false if surround link by color fram. True if color the text
  citecolor = {MyBlue},
  linkcolor = {MyRed},
  filecolor= {Green},	
  urlcolor = {MyRed},
  hyperindex = {true},
  linktocpage = {true},
  linkbordercolor ={1 0 0},
 citebordercolor ={0 1 1},
 urlbordercolor ={0 1 1}
}

\makeatother

\providecommand{\algorithmname}{Algorithm}
\providecommand{\assumptionname}{Assumption}
\providecommand{\conditionname}{Condition}
\providecommand{\corollaryname}{Corollary}
\providecommand{\definitionname}{Definition}
\providecommand{\lemmaname}{Lemma}
\providecommand{\propositionname}{Proposition}
\providecommand{\theoremname}{Theorem}

\begin{document}
\setcounter{page}{0}
\title{\textbf{Generalized Laplace Inference in Multiple Change-Points Models}\thanks{This paper is based on the fourth chapter of the first author's doctoral
dissertation at Boston University. We thank the Editor and a Co-Editor
for guiding the review, and three anonymous referees for constructive
comments. We also thank Zhongjun Qu for useful comments. }\textbf{ }}
\maketitle
\begin{abstract}
{\footnotesize{}Under the classical long-span asymptotic framework
we develop a class of Generalized Laplace (GL) inference methods for
the change-point dates in a linear time series regression model with
multiple structural changes analyzed in, e.g., \citet{bai/perron:98}.
The GL estimator is defined by an integration rather than optimization-based
method and relies on the least-squares criterion function. It is interpreted
as a classical (non-Bayesian) estimator and the inference methods
proposed retain a frequentist interpretation. This approach provides
a better approximation about the uncertainty in the data of the change-points
relative to existing methods. On the theoretical side, depending on
some input (smoothing) parameter, the class of GL estimators exhibits
a dual limiting distribution; namely, the classical shrinkage asymptotic
distribution, or a Bayes-type asymptotic distribution. We propose
an inference method based on Highest Density Regions using the latter
distribution. We show that it has attractive theoretical properties
not shared by the other popular alternatives, i.e., it is bet-proof.
Simulations confirm that these theoretical properties translate to
good finite-sample performance.}{\footnotesize\par}
\end{abstract}
\noindent {\bf{Keywords}}: Asymptotic Distribution, Bet-Proof, Break Date, Change-point,   Generalized Laplace Inference, Highest Density Region, Quasi-Bayes.  %%%%%%%%%%%%%% \onehalfspacing 
% \doublespacing %%%%%%%%%%%%%% \newpage{} 

\onehalfspacing
\thispagestyle{empty}
\allowdisplaybreaks

\vfill{}
\pagebreak{}

\section{\label{Section Introduction Lap_BP}Introduction}

In the context of the multiple change-points model analyzed in \citet{bai/perron:98},
we develop inference methods for the change-point dates for a class
of Generalized Laplace (GL) estimators using a classical long-span
asymptotic framework. They are defined by an integration rather than
an optimization-based method, the latter typically characterizing
classical extremum estimators. The idea traces back to \citet{laplace:74},
who first suggested to interpret transformations of a least-squares
criterion function as a statistical belief over a parameter of interest.
Hence, a Laplace estimator is defined similarly to a Bayesian estimator
although the former relies on a statistical criterion function rather
than a parametric likelihood function. As a consequence, the GL estimator
is interpreted as a classical (non-Bayesian) estimator and the inference
methods proposed retain a frequentist interpretation such that the
GL estimators are constructed as a function of integral transformations
of the least-squares criterion. In a first step, we use the approach
of \citet{bai/perron:98} to evaluate the least-squares criterion
function at all candidate break dates. We then apply a transformation
to obtain a proper distribution over the parameters of interest,
referred to as the Quasi-posterior. For a given choice of a loss function
and (possibly) a prior density, the estimator is then defined either
explicitly as, for example, the mean or median of the (weighted) Quasi-posterior
or implicitly as the minimizer of a smooth convex optimization problem. 

The underlying asymptotic framework considered  is the long-span
shrinkage asymptotics of \citet{bai:97RES}, \citet{bai/perron:98}
and also \citet{perron/qu:06} who considerably relaxed some conditions,
where the magnitude of the parameter shift is sample-size dependent
and approaches zero as the sample size increases. Early contributions
to this approach are \citet{hinkley:71}, \citet{bhattacharya:87},
and \citet{yao:87} for estimating break points. For testing for structural
breaks, see \citet{hawkins:77}, \citet{picard:85}, \citet{kim/siegmund:89},
\citet{andrews:93}, \citet{horvath:93} and \citet{andrews/ploberger:94}.
See also the reviews of \citet{csorgo/horvath:97}, \citet{perron:06},
\citet{casini/perron_Oxford_Survey} and references therein. 

One of our goals is to develop GL estimates with better small-sample
properties compared to least-squares estimates, namely lower Mean
Absolute and Root-Mean Squared Errors, and confidence sets with accurate
coverage probabilities and relatively short lengths for a wide range
of break sizes, whether small or large; existing methods work well
for either small or large breaks, but not for both. A second goal
is to establish theoretical results that support the reported finite-sample
properties about inference.

The asymptotic distribution of the GL estimator is derived via a local
parameter related to a normalized deviation from the true fractional
break date. The normalization factor corresponds to the rate of convergence
of the original (extremum) least-squares estimator as established
by \citet{bai/perron:98}.  The asymptotic distribution of the GL
estimator then depends on a sample-size dependent smoothing parameter
sequence applied to the least-squares criterion function. We derive
two distinct limiting distributions corresponding to different smoothing
sequences of the criterion function {[}cf. \citet{jun/pinkse/wan:15}
for a related application in the context of the cube-root asymptotics
of \citet{kim/pollard:90}{]}. In one case, the estimator displays
the same limit law as the asymptotic distribution of the least-squares
estimator derived in \citet{bai/perron:98} {[}see also \citet{hinkley:71},
\citet{picard:85} and \citet{yao:87}{]}. In a second case, the
limiting distribution is characterized by a ratio of integrals over
functions of Gaussian processes and resembles the limiting distribution
of Bayesian change-point estimators. The latter is exploited for
the purpose of constructing confidence sets for the break dates.
We use the concept of highest density regions (HDR) introduced by
\citet{casini/perron_CR_Single_Break} for structural change problems,
which best summarizes the properties of the probability distribution
of interest. The HDR are common in Bayesian analysis where they are
applied to a posterior distribution {[}see, e.g., \citeauthor{box/tiao:1973}{]}.
\citeauthor{kendall/stuart:1961} discussed the difference between
frequentist confidence intervals and Bayesian approaches in relation
to the existence of a sufficient statistic. Our procedure yields confidence
sets for the break date which, in finite samples, better account for
the uncertainty over the parameter space in finite-samples because
it effectively incorporates a statistical measure of the uncertainty
in the least-squares criterion function. As noted in the literature
on likelihood-based inference in some classes of nongranular problems
{[}see e.g., \citet{chernozhukov/hong:03}, \citet{ghosal/ghosh/samanta:95},
\citet{hirano/porter:03} and \citet{ibragimov/has:81}{]}, the Maximum
Likelihood Estimator (MLE) is generally not an asymptotically sufficient
statistic in these models and so the likelihood contains more information
asymptotically than the MLE. Hence, likelihood-based procedures are
generally not functions of the MLE even asymptotically. This incompleteness
property motivated the study of the entire likelihood rather than
just the MLE. Likewise, our method exploits the entire behavior of
the objective function.

Laplace's seminal insight has been applied successfully in many disciplines.
In econometrics, \citet{chernozhukov/hong:03} introduced Laplace-type
estimators as an alternative to classical (regular) extremum estimators
in several problems such as censored median regression and  nonlinear
instrumental variable; see also \citet{forneron/ng:17} for a review
and comparisons. Their main motivation was to solve the curse of dimensionality
inherent to the computation of such estimators. In contrast, the class
of GL estimators in structural change models serves distinct multiple
purposes. First,  inference about the break dates  presents several
challenges, in particular to provide methods with a satisfactory performance
uniformly over different data-generating mechanisms and break magnitudes.
The GL inference proves to be reliable and accurate in finite-samples.
Second, it leads to inference methods that have both frequentist and
credibility properties which is not shared by the other popular methods. 

Turning to the problem of constructing confidence sets for a single
break date, the standard asymptotic method for the linear regression
model was proposed in \citet{bai:97RES}, while \citet{elliott/mueller:07}
proposed to invert the locally best invariant test of \citet{nyblom:89},
and \citet{eo/morley:15} suggested to invert the likelihood-ratio
statistic of \citet{qu/perron:07}. The latter were mainly motivated
by finite-sample results indicating that the exact coverage rates
of the confidence intervals obtained from Bai's (1997) method are
often below the nominal level when the magnitude of the break is small.
It has been shown that the method of \citet{elliott/mueller:07}  delivers
the most accurate coverage rates but the average length of the confidence
sets is significantly larger than with other methods. The confidence
sets for the break dates constructed from the GL inference that we
develop result in exact coverage rates close to the nominal level
and short length of the confidence sets. This holds true whether the
magnitude of the break is small or large. In fact, we show that
GL inference is bet-proof, a measure of ``reasonableness'' of frequentist
inference in non-regular problems {[}see, e.g., \citet{buehler:59}{]}.

The GL inference developed in this paper has been applied by \citet{casini/perron_Lap_CR_Single_Inf}
to achieve finite-sample improvements under the continuous record
asymptotic framework of \citet{casini/perron_CR_Single_Break}. The
latter proposed an alternative asymptotic framework to explain the
non-standard features of the finite-sample distribution of the least-squares
estimator. 

The paper is organized as follows. We first focus on the single change-point
case. Section \ref{Section, Model and its Assumptions, Bai Lap} presents
the statistical setting. We develop the asymptotic theory in Section
\ref{Section Asymptotic Results LapBai97} and the inference methods
in Section \ref{Section Inference Methods}. Results for multiple
change-points models are given in Section \ref{Section Models with Multiple Breaks}
while Section \ref{Section Theoretical-Properties-of GL Inference}
discusses some theoretical properties of GL inference. Section \ref{Section Monte-Carlo-Simulation}
presents simulation results about the finite-sample performance. Section
\ref{Section Conclusions} concludes. All proofs are included in an
online supplement {[}\citet{casini/perron_SC_BP_Lap_Supp}{]}.

\section{\label{Section, Model and its Assumptions, Bai Lap}The Model and
the Assumptions}

This section introduces the structural change model with a single
break, reviews the least-squares estimation method for the break date,
and presents the relevant assumptions. We start with introducing the
formal setup for our analysis. The following notation is used throughout.
We denote the transpose of a matrix $A$ by $A'$.  We use $\left\Vert \cdot\right\Vert $
to denote the Euclidean norm of a linear space, i.e., $\left\Vert x\right\Vert =\left(\sum_{i=1}^{p}x_{i}^{2}\right)^{1/2}$
for $x\in\mathbb{R}^{p}.$ For a matrix $A$, we use the vector-induced
norm, i.e., $\left\Vert A\right\Vert =\sup_{x\neq0}\left\Vert Ax\right\Vert /\left\Vert x\right\Vert .$
All vectors are column vectors. For two vectors $a$ and $b$, we
write $a\leq b$ if the inequality holds component-wise. We use $\left\lfloor \cdot\right\rfloor $
to denote the largest smaller integer function. We use $\overset{\mathbb{P}}{\rightarrow}$
and $\overset{d}{\rightarrow}$ to denote convergence in probability
and convergence in distribution, respectively. $\mathbb{C}_{b}\left(\mathbf{E}\right)$
{[}$\mathbb{D}_{b}\left(\mathbf{E}\right)${]} is the collection of
bounded continuous {[}càdlàg{]} functions from some specified set
$\mathbf{E}$ to $\mathbb{R}$. Weak convergence on either $\mathbb{C}_{b}\left(\mathbf{E}\right)$
or $\mathbb{D}_{b}\left(\mathbf{E}\right)$ is denoted by $\Rightarrow$.
The symbol ``$\triangleq$'' stands for definitional equivalence.

We consider a sample of observations $\left\{ \left(y_{t},\,w_{t},\,z_{t}\right):\,t=1,\ldots,\,T\right\} ,$
defined on a filtered probability space $\left(\Omega,\,\mathscr{F},\,\mathbb{P}\right)$,
on which all of the random elements introduced in what follows are
defined. The model is 
\begin{align}
\begin{split}y_{t} & =w_{t}'\phi^{0}+z_{t}'\delta_{1}^{0}+e_{t},\quad\left(t=1,\ldots,\,T_{b}^{0}\right)\qquad y_{t}=w_{t}'\phi^{0}+z_{t}'\delta_{2}^{0}+e_{t},\quad\left(t=T_{b}^{0}+1,\ldots,\,T\right)\end{split}
\label{Eq. the Single Break Model}
\end{align}
 where $y_{t}$ is a scalar dependent variable, $w_{t}$ and $z_{t}$
are regressors of dimensions, $p$ and $q,$ respectively, and $e_{t}$
is an unobserved error term. The true parameter vectors $\phi^{0},\,\delta_{1}^{0}$
and $\delta_{2}^{0}$ are unknown and we define $\delta^{0}\triangleq\delta_{2}^{0}-\delta_{1}^{0},$
with $\delta^{0}\neq0$ so that a structural change occurs at date
$T_{b}^{0}$. It is useful to re-parametrize the model. Letting $x_{t}\triangleq\left(w'_{t},\,z'_{t}\right)'$
and $\beta^{0}\triangleq\left(\left(\phi^{0}\right)',\,\left(\delta_{1}^{0}\right)'\right)',$
we have
\begin{align}
\begin{split}y_{t} & =x_{t}'\beta^{0}+e_{t},\quad\left(t=1,\ldots,\,T_{b}^{0}\right)\qquad y_{t}=x_{t}'\beta^{0}+z_{t}'\delta^{0}+e_{t},\quad\left(t=T_{b}^{0}+1,\ldots,\,T\right).\end{split}
\label{Eq. Reparameterize SC model}
\end{align}
 More generally, we can define $z_{t}\triangleq D'x_{t}$, where $D$
is a $\left(p+q\right)\times q$ matrix with full column rank. A pure
structural change model in which all regression parameters are subject
to change corresponds to $D=I_{\left(p+q\right)\times\left(p+q\right)}$,
whereas a partial structural change model arises when $D=\left(0_{q\times p},\,I_{q\times q}\right)'.$
In order to facilitate the derivations, we reformulate model \eqref{Eq. Reparameterize SC model}
in matrix format. Let $Y=\left(y_{1},\,\ldots,\,y_{T}\right)',\,X=\left(x_{1},\,\ldots,\,x_{T}\right)'$,
$e=\left(e_{1},\,\ldots,\,e_{T}\right)',$ $X_{1}=\left(x_{1},\,\ldots,\,x_{T_{b}},\,0,\,\ldots,\,0\right)'$,
$X_{2}=(0,\,\ldots,\,0,\,$ $x_{T_{b}+1},\ldots,\,x_{T})'$ and $X_{0}=(0,\,\ldots,\,0,\,x_{T_{b}^{0}+1},\ldots,\,x_{T})'$.
Further, define $Z_{1},\,Z_{2}$ and $Z_{0}$ in a similar way: $Z_{1}=X_{1}D,\,Z_{2}=X_{2}D$
and $Z_{0}=X_{0}D$. We omit the dependence of the matrices $X_{i}$
and $Z_{i}$ ($i=1,\,2$) on $T_{b}$. Then, \eqref{Eq. Reparameterize SC model}
is equivalent to
\begin{align}
Y & =X\beta+Z_{0}\delta+e.\label{Eq Matrix Format of the model}
\end{align}

Let $\theta^{0}\triangleq\left(\left(\phi^{0}\right)',\,\left(\delta_{1}^{0}\right)',\,\left(\delta^{0}\right)'\right)'$
denote the true value of the parameter vector $\theta\triangleq\left(\phi,\,\delta_{1},\,\delta\right).$
The break date least-squares (LS) estimator $\widehat{T}_{b}^{\mathrm{LS}}$
is the minimizer of the sum of squared residuals {[}denoted $S_{T}\left(\theta,\,T_{b}\right)${]}
from \eqref{Eq Matrix Format of the model}. The parameter $\theta$
can be concentrated out resulting in a criterion function depending
only on $T_{b}=T\lambda_{b}$, i.e.,~$\widehat{T}_{b}^{\mathrm{LS}}=\arg\min_{1\leq T_{b}\leq T}S_{T}(\widehat{\theta}^{\mathrm{LS}}(T_{b}),\,T_{b})$
where $\widehat{\theta}^{\mathrm{LS}}(T_{b})=\arg\min_{\theta}S_{T}(\theta,\,T_{b})$
with $S_{T}(\theta,\,T_{b})=\sum_{t=1}^{T_{b}}\left(y_{t}-\phi'w_{t}-\delta'_{1}z_{t}\right)^{2}+\sum_{t=T_{b}+1}^{T}\left(y_{t}-\phi'w_{t}-\delta'z_{t}\right)^{2}$.
Also,
\begin{align}
\arg\min_{1\leq T_{b}\leq T}S_{T}(\widehat{\theta}^{\mathrm{LS}}\left(T_{b}\right),\,T_{b}) & =\arg\max_{T_{b}}\widehat{\delta}^{\mathrm{LS\prime}}(T_{b})(Z_{2}'M_{X}Z_{2})\widehat{\delta}^{\mathrm{LS}}(T_{b})\label{Eq. Relationship Sup Wald and SSR}\\
 & \triangleq\arg\max_{\lambda_{b}}Q_{T}(\widehat{\delta}^{\mathrm{LS}}(\lambda_{b}),\,\lambda_{b}),\nonumber 
\end{align}
 where $M_{X}\triangleq I-X\left(X'X\right)^{-1}X'$, $\widehat{\delta}^{\mathrm{LS}}\left(\lambda_{b}\right)$
is the least-squares estimator of $\delta^{0}$ obtained by regressing
$Y$ on $X$ and $Z_{2}$ and the statistic $Q_{T}\left(\widehat{\delta}^{\mathrm{LS}}\left(\lambda_{b}\right),\,\lambda_{b}\right)$
is the numerator of the sup-Wald statistic. The Laplace-type inference
builds on the least-squares criterion function $Q_{T}\left(\delta\left(\lambda_{b}\right),\,\lambda_{b}\right),$
where $\delta\left(\lambda_{b}\right)$ stands for $\widehat{\delta}^{\mathrm{LS}}\left(\lambda_{b}\right)$
to minimize notational burden.
\begin{assumption}
\label{Assumption A Bai97}$T_{b}^{0}=\left\lfloor T\lambda_{b}^{0}\right\rfloor ,$
where $\lambda_{b}^{0}\in\varGamma^{0}\subset\left(0,\,1\right).$
\end{assumption}

\begin{assumption}
\label{Assumption A5 Bai97 and A4 BP98}With $\left\{ \mathscr{F}_{t},\,t=1,\,2,\ldots\right\} $
a sequence of increasing $\sigma$-fields, $\left\{ z_{t}e_{t},\,\mathscr{F}_{t}\right\} $
forms an $L^{r}$-mixingale sequence with $r=2+\nu$ for some $\nu>0$.
That is, there exist nonnegative constants $\left\{ \varrho_{1,t}\right\} _{t\geq1}$
and $\left\{ \varrho_{2,j}\right\} _{j\geq0}$ such that $\varrho_{2,j}\rightarrow0$
as $j\rightarrow\infty$, and for all $t\geq1$, $j\geq0$ and $r\geq1$,
(i) $\left\Vert \mathbb{E}\left(z_{t}e_{t}|\,\mathscr{F}_{t-j}\right)\right\Vert _{r}\leq\varrho_{1,t}\varrho_{2,j}$,
(ii) $\left\Vert z_{t}e_{t}-\mathbb{E}\left(z_{t}e_{t}|\,\mathscr{F}_{t+j}\right)\right\Vert _{r}\leq\varrho_{1,t}\varrho_{2,j+1}$.
In addition, (iii) $\max_{t}\varrho_{1,t}<C_{1}<\infty$ and (iv)
$\sum_{j=0}^{\infty}j^{1+\nu}\varrho_{2,j}<\infty$ for some $\nu>0$,
(v) $\left\Vert z_{t}\right\Vert _{2r}<C_{2}<\infty$ and $\left\Vert e_{t}\right\Vert _{2r}<C_{3}<\infty$
for some $C_{1},\,C_{2},\,C_{3}>0$. 
\end{assumption}

\begin{assumption}
\label{Assumption A3 Bai97}There exists an $l_{0}>0$ such that for
all $l>l_{0},$ the minimum eigenvalues of $H_{l}^{*}=\left(1/l\right)\sum_{T_{b}^{0}-l+1}^{T_{b}^{0}}x_{t}x'_{t}$
and $H_{l}^{**}=\left(1/l\right)\sum_{T_{b}^{0}+1}^{T_{b}^{0}+l}x_{t}x'_{t}$
are bounded away from zero. These matrices are invertible when $l\geq p+q$
and have stochastically bounded norms uniformly in $l$.
\end{assumption}

\begin{assumption}
\label{Assumption A4 Bai97}$T^{-1}X'X\overset{\mathbb{P}}{\rightarrow}\Sigma_{XX},$
where $\Sigma_{XX}$, a positive definite matrix.
\end{assumption}
These assumptions are standard and similar to those in \citet{perron/qu:06}.
It is well-known that only the fractional break date $\lambda_{b}^{0}$
(not $T_{b}^{0}$) can be consistently estimated, with $\widehat{\lambda}_{b}^{\mathrm{LS}}$
having a $T$-rate of convergence. The corresponding result for
the break date estimator $\widehat{T}_{b}^{\mathrm{LS}}$ states
that, as $T$ increases, $\widehat{T}_{b}^{\mathrm{LS}}$ remains
within a bounded distance from $T_{b}^{0}$. However, this does not
affect the estimation problem of the regression coefficients $\theta^{0}$,
for which $\widehat{\theta}^{\mathrm{LS}}$ is a regular estimator;
i.e., $\sqrt{T}$-consistent and asymptotically normally distributed,
since the estimation of the regression parameters is asymptotically
independent from the estimation of the change-point. Hence, the regression
parameters are essentially estimated as if the change-point was known.
More complex is the derivation of the asymptotic distribution of $\widehat{\lambda}_{b}^{\mathrm{LS}}$;
e.g., \citet{hinkley:71} for an i.i.d. Gaussian process with a mean
change. Therefore, to make progress it is necessary to consider a
shrinkage asymptotic setting in which the size of the shift converges
to zero as $T\rightarrow\infty$; see \citet{picard:85} and \citet{yao:87}
and extended by \citet{bai:97RES} to general linear models. 

\section{\label{Section Asymptotic Results LapBai97}Generalized Laplace Estimation}

We define the GL estimator in Section \ref{subsec:The-Setting-and}
and discuss its usefulness in Section \ref{subsection Discussion-on-GL}.
Section \ref{subsec:Normalized-Version-of} describes the asymptotic
framework under which we derive the limiting distribution with the
results presented in Section \ref{Sub section Theorems Laplace in SC}.

\subsection{\label{subsec:The-Setting-and}The Class of Laplace Estimators}

The class of GL estimators relies on the original least-squares criterion
function $Q_{T}\left(\delta\left(\lambda_{b}\right),\,\lambda_{b}\right)$,
with the parameter of interest being $\lambda_{b}^{0}=T_{b}^{0}/T$.
The Quasi-posterior $p_{T}\left(\lambda_{b}\right)$ is defined by
the exponential transformation,
\begin{align}
p_{T}\left(\lambda_{b}\right) & \triangleq\frac{\exp\left(Q_{T}\left(\delta\left(\lambda_{b}\right),\,\lambda_{b}\right)\right)\pi\left(\lambda_{b}\right)}{\int_{\varGamma^{0}}\exp\left(Q_{T}\left(\delta\left(\lambda_{b}\right),\,\lambda_{b}\right)\right)\pi\left(\lambda_{b}\right)d\lambda_{b}},\label{Eq. Quasi-posterior-1}
\end{align}
where $\pi\left(\cdot\right)$ is a density function. Note that $p_{T}\left(\lambda_{b}\right)$
defines a proper distribution over the parameter space $\varGamma^{0}$.
The $\mathscr{\mathscr{L}}\left(\theta,\,T_{b}\right)$-class of estimators
are the solutions of smooth convex optimization problems for a given
loss function, restricting attention to convex loss functions $l_{T}\left(\cdot\right)$.
Examples include (a) $l_{T}\left(r\right)=a_{T}^{m}\left|r\right|^{m},$
the polynomial loss function (the squared loss function is obtained
when $m=2$ and the absolute deviation loss function when $m=1$);
(b) $l_{T}\left(r\right)=a_{T}\left(\tau-\mathbf{1}\left(r\leq0\right)\right)r,$
the check loss function; where $a_{T}$ is a divergent sequence. We
define the Expected Risk function, under the density $p_{T}\left(\cdot\right)$
and the loss $l_{T}\left(\cdot\right)$ as $\mathcal{R}_{l,T}\left(s\right)\triangleq\mathbb{E}_{p_{T}}\left[l_{T}\left(s-\widetilde{\lambda}_{b}\right)\right],$
where $\widetilde{\lambda}_{b}$ is a random variable with distribution
$p_{T}$ and $\mathbb{E}_{p_{T}}$ denotes expectation taken under
$p_{T}.$ Using \eqref{Eq. Quasi-posterior-1} we have,
\begin{align}
\mathcal{R}_{l,T}\left(s\right) & \triangleq\int_{\varGamma^{0}}l_{T}\left(s-\lambda_{b}\right)p_{T}\left(\lambda_{b}\right)d\lambda_{b}.\label{Eq. Expected Risk function-1}
\end{align}
 The Laplace-type estimator $\widehat{\lambda}_{b}^{\mathrm{GL}}$
shall be interpreted as a decision rule that, given the information
contained in the Quasi-posterior $p_{T}$, is least unfavorable according
to the loss function $l_{T}$ and the prior density $\pi$. Then $\widehat{\lambda}_{b}^{\mathrm{GL}}$
is the minimizer of the expected risk function \eqref{Eq. Expected Risk function-1},
i.e., $\widehat{\lambda}_{b}^{\mathrm{GL}}\triangleq\arg\min_{s\in\varGamma^{0}}\left[\mathcal{R}_{l,T}\left(s\right)\right].$
Observe that the GL estimator $\widehat{\lambda}_{b}^{\mathrm{GL}}$
results in the mean (median) of the Quasi-posterior upon choosing
 the squared (absolute deviation) loss function. The choice of the
loss and of the prior density functions hinges on the statistical
problem addressed. In the structural change problem, a natural choice
for the Quasi-prior $\pi$ is the density of the asymptotic distribution
of $\widehat{\lambda}_{b}^{\mathrm{LS}}$. This requires to replace
the population quantities appearing in that distribution by consistent
plug-in estimates\textemdash cf. \citet{bai/perron:98}\textemdash and
derive its density via simulations as in \citet{casini/perron_CR_Single_Break}.
The attractiveness of the Quasi-posterior \eqref{Eq. Quasi-posterior-1}
is that it provides additional information about the parameter of
interest $\lambda_{b}^{0}$ beyond what is already included in the
point estimate $\widehat{\lambda}_{b}^{\mathrm{LS}}$ and its distribution
(see Section \ref{subsection Discussion-on-GL}). This approach will
result in more accurate inference in finite-samples even in cases
with high uncertainty in the data as we shall document in Section
\ref{Section Monte-Carlo-Simulation}. This is supported in Section
\ref{Section Theoretical-Properties-of GL Inference} showing that
the GL inference is bet-proof which is a desirable theoretical property
in non-regular problems.
\begin{assumption}
\label{Assumption The-loss-function LapBai97}Let $l_{T}\left(r\right)\triangleq l\left(a_{T}r\right)$,
with $a_{T}$ a positive divergent sequence. $\boldsymbol{L}$ denotes
the set of functions $l:\,\mathbb{R}\rightarrow\mathbb{R}_{+}$ that
satisfy (i) $l\left(r\right)$ is defined on $\mathbb{R}$, with $l\left(r\right)\geq0$
and $l\left(r\right)=0$ if and only if $r=0$; (ii) $l\left(r\right)$
is continuous at $r=0$; (iii) $l\left(\cdot\right)$ is convex and
$l\left(r\right)\leq1+\left|r\right|^{m}$ for some $m>0$. 
\end{assumption}

\begin{assumption}
\label{Assumption Prior LapBai97}$\pi:\,\mathbb{R}\rightarrow\mathbb{R}_{+}$
is a continuous, uniformly positive density function satisfying $\pi^{0}\triangleq\pi\left(\lambda_{b}^{0}\right)>0,$
and for some finite $C_{\pi}<\infty,$ $\pi^{0}<C_{\pi}$. Also, $\pi\left(\lambda_{b}\right)=0$
for all $\lambda_{b}\notin\varGamma^{0}$, and $\pi$ is twice continuously
differentiable with respect to $\lambda_{b}$ at $\lambda_{b}^{0}$. 
\end{assumption}

Assumption \ref{Assumption The-loss-function LapBai97} is similar
to those in \citet{bickel/yahav:69}, \citet{ibragimov/has:81} and
\citet{chernozhukov/hong:03}. The convexity assumption on $l_{T}\left(\cdot\right)$
is guided by practical considerations. The dominant restriction in
part (iii) is conventional and implicitly assumes that the loss function
has been scaled by some constant. What is important is that the growth
of the function $l_{T}\left(r\right)$ as $\left|r\right|\rightarrow\infty$
is slower than $\exp\left(\epsilon\left|r\right|\right)$ for any
$\epsilon>0$. Assumption \ref{Assumption Prior LapBai97} on the
prior is satisfied for any reasonable choice. For priors that have
a peak at $\lambda_{b}^{0}$ one can apply some basic smoothing techniques
to make it differentiable locally {[}e.g., mean smoothing, Gaussian
smoothing and Savitzky-Golay filter{]}. We did not find any particular
difference in the empirical results and so we used the mean smoothing.
The assumption on the differentiability of the kernel can be relaxed
at the expense of one more step in the proof. \citet{chernozhukov/hong:03}
assumed differentiability of the prior; we also keep the same assumption
and applied the smoothing. The large-sample properties of the $\mathscr{\mathscr{L}}\left(\theta,\,T_{b}\right)$-class
are studied under the shrinkage asymptotic setting of \citet{bai:97RES}
and \citet{bai/perron:98}. Thus, we need the following assumption.
\begin{assumption}
\label{Assumption Small Shift BP}Let $\delta_{T}\triangleq\delta_{T}^{0}\triangleq v_{T}\delta^{0}$
where $v_{T}>0$ is a scalar satisfying $v_{T}\rightarrow0$ as $T\rightarrow\infty$
and $T^{1/2-\vartheta}v_{T}\rightarrow\infty$ for some $\vartheta\in\left(0,\,1/4\right)$. 
\end{assumption}
We omit the superscript 0 from $\delta_{T}^{0}$ for notational convenience
since it should not cause any confusion. Assumption \ref{Assumption Small Shift BP}
requires the magnitude of the break to shrink to zero at any slower
rate than $T^{-1/2}$. The specific rates allowed differ from those
in \citet{bai:97RES} and \citet{bai/perron:98}, since they require
$\vartheta\in\left(0,\,1/2\right)$. The reason is merely technical;
the asymptotics of the Laplace-type estimator involve smoothing the
criterion function, and thus one needs to guarantee that $\widehat{\lambda}_{b}$
approaches $\lambda_{b}^{0}$ at a sufficiently fast rate. Under the
shrinkage asymptotics, Proposition 1 and Corollary 1 in \citet{bai:97RES}
state that $T\left\Vert \delta_{T}\right\Vert ^{2}\left(\widehat{\lambda}_{b}^{\mathrm{LS}}-\lambda_{b}^{0}\right)=O_{\mathbb{P}}\left(1\right)$
and $\widehat{\delta}_{T}^{\mathrm{LS}}-\delta_{T}=o_{\mathbb{P}}\left(1\right)$. 

\subsection{\label{subsection Discussion-on-GL}Discussion about the GL Approach}

We use Figure \ref{Fig1}-\ref{Fig2} to illustrate the main idea
behind the usefulness of the GL method. They present plots of the
density of the distribution of $\widehat{T}_{b}^{\mathrm{LS}}$ and
$\widehat{T}_{b}^{\mathrm{GL}}$ for the simple model $y_{t}=\phi^{0}+z_{t}\left(\delta_{1}^{0}+\delta^{0}\mathbf{1}\left\{ t>T_{b}^{0}\right\} \right)+e_{t}$
where $\left\{ z_{t}\right\} $ follows an ARMA(1,1) process and $e_{t}\sim i.i.d.\,\mathscr{N}\left(0,\,1\right)$.
The distributions presented are the exact finite-sample distributions
of the LS and GL estimators, \citeauthor{bai:97RES}\textquoteright s
(1997) classical large-$N$ limit distribution, and the asymptotic
distribution of the GL estimator. Noteworthy are the non-standard
features of the finite-sample distribution of the LS estimator when
the break magnitude is small, which include multi-modality, fat tails
and asymmetry. The central mode is near $\widehat{T}_{b}^{\mathrm{LS}}$
while the other two modes are in the tails near the start and end
of the sample period; when the break magnitude is small $\widehat{T}_{b}^{\mathrm{LS}}$
tends to locate the break in the tails since the evidence of a break
is weak. It is evident that the classical large-$N$ asymptotic distribution
provides a poor approximation especially for small break sizes. Some
of these features have been found in other works {[}see, e.g., \citet{perron/zhu:05},
\citet{deng/perron:06}, Jiang, et al. \citeyearpar{jiang/wang/yu:16,jiang/wang/yu:17},
and \citet{casini/perron_CR_Single_Break}{]}. Turning to the densities
of the GL estimators, some of the nonstandard features appear also
for the GL estimator although to a much lesser extent. In particular,
the densities of the GL estimators are less spread out than the corresponding
densities for the LS estimator. For small breaks, the finite-sample
distributions of the LS and GL estimators are quite different, which
suggests that standard measures of accuracy (e.g., MAE and RMSE) can
be expected to differ substantially. For $\lambda_{0}=0.5$ the GL
estimator exhibits much less variability and more precision. The figures
also show that the asymptotic distribution of the GL estimator provides
an accurate approximation for large breaks while for small breaks
the approximation is less accurate. However, it captures the fat-tails
of the finite-sample distribution which suggests that it does not
underestimate uncertainty about the break location unlike \citeauthor{bai:97RES}\textquoteright s
(1997) distribution. 

The GL method is useful because it weights the information from the
least-squares criterion function with the information from the prior
density\textemdash which, here, is the density of the asymptotic distribution
of $\widehat{T}_{b}^{\mathrm{LS}}$. Note that the least-squares objective
function is quite flat when the magnitude of the break is small and
so $\widehat{T}_{b}^{\mathrm{LS}}$ is imprecise. The resulting Quasi-posterior,
or, e.g., its  median, is likely to lead to better estimates in finite-samples,
because it takes into account the overall shape of the objective function
which weighted by the prior becomes more informative about the uncertainty
of the break date. 

\subsection{\label{subsec:Normalized-Version-of}Normalized Version of $\mathcal{R}_{l,T}\left(s\right)$}

In order to develop the asymptotic results, we introduce a smoothing
sequence $\left\{ \gamma_{T}\right\} $ whose properties are specified
below and work with a normalized version of $\mathcal{R}_{l,T}\left(s\right)$
in order to be able to derive the relevant limit results. We assume
that $\lambda_{b}^{0}\in\varGamma^{0}\subset\left(0,\,1\right)$ is
the unknown extremum of $\widetilde{Q}\left(\theta^{0},\,\lambda_{b}\right)=\mathbb{E}\left[Q_{T}\left(\theta^{0},\,\lambda_{b}\right)\right]$
and that $\theta^{0}\triangleq\left(\left(\phi^{0}\right)',\,\left(\delta_{1}^{0}\right)',\,\left(\delta^{0}\right)'\right)'\in\mathbf{S}\subset\mathbb{R}^{p}\times\mathbb{R}^{q}\times\mathbb{R}^{q}$.
Our analysis is within a vanishing neighborhood of $\theta^{0}$.
For any $\theta\in\mathbf{S}$, let $\lambda_{b}^{0}\left(\theta\right)$
be an arbitrary element of $\varGamma^{0}\left(\theta\right)\triangleq\left\{ \lambda_{b}\in\varGamma^{0}:\,\widetilde{Q}\left(\theta,\,\lambda_{b}\right)=\sup_{\widetilde{\lambda}_{b}\in\mathcal{\varGamma}^{0}}\widetilde{Q}\left(\theta,\,\widetilde{\lambda}_{b}\right)\right\} $.
Provided a uniqueness condition is assumed (see Assumption \ref{Assumption Gaussian Process for Lap LapBai97}),
$\varGamma^{0}\left(\theta\right)$ contains a single element, $\lambda_{b}^{0}$.
Further, let $\overline{Q}_{T}\left(\theta,\,\lambda_{b}\right)\triangleq Q_{T}\left(\theta,\,\lambda_{b}\right)-Q_{T}\left(\theta,\,\lambda_{b}^{0}\right),$
$Q_{T}^{0}\left(\theta,\,\lambda_{b}\right)\triangleq\mathbb{E}\left[Q_{T}\left(\theta,\,\lambda_{b}\right)-Q_{T}\left(\theta,\,\lambda_{b}^{0}\right)|\,X\right],$
and $G_{T}\left(\theta,\,\lambda_{b}\right)\triangleq\overline{Q}_{T}\left(\theta,\,\lambda_{b}\right)-Q_{T}^{0}\left(\theta,\,\lambda_{b}\right).$
These expressions are given by $G_{T}\left(\theta,\,\lambda_{b}\right)=g_{e}\left(\theta,\,\lambda_{b}\right)$,
$Q_{T}^{0}=g_{d}\left(\theta,\,\lambda_{b}\right)$ and $\overline{Q}_{T}=g_{d}\left(\theta,\,\lambda_{b}\right)+g_{e}\left(\theta,\,\lambda_{b}\right)$,
where

\begin{align}
g_{d}\left(\theta,\,\lambda_{b}\right) & =\delta'_{T}\left\{ \left(Z'_{0}MZ_{2}\right)\left(Z'_{2}MZ_{2}\right)^{-1}\left(Z'_{2}MZ_{0}\right)-Z_{0}'MZ_{0}\right\} \delta_{T},\label{eq. A.2.1-1}
\end{align}
and
\begin{align}
g_{e} & \left(\theta,\,\lambda_{b}\right)=2\delta_{T}'\left(Z'_{0}MZ_{2}\right)\left(Z'_{2}MZ_{2}\right)^{-1}Z_{2}Me-2\delta_{T}'\left(Z'_{0}Me\right)\label{eq. A.2.3-1}\\
 & \quad+e'MZ_{2}\left(Z'_{2}MZ_{2}\right)^{-1}Z_{2}Me-e'MZ_{0}\left(Z'_{0}MZ_{0}\right)^{-1}Z'_{0}Me.\nonumber 
\end{align}
They are derived in Section \ref{subsection: Preliminary Lemmas}.
For the purpose of developing the asymptotic theory, the GL estimator
$\widehat{\lambda}_{b}^{\mathrm{GL}}\left(\theta\right)$ is defined
as the minimizer of a normalized version of $\mathcal{R}_{l,T}\left(s\right)$:
\vfill{}

\begin{align}
\Psi_{l,T}\left(s;\,\theta\right) & =\int_{\varGamma^{0}}l\left(s-\lambda_{b}\right)\frac{\exp\left(\left(\gamma_{T}/\left(T\left\Vert \delta_{T}\right\Vert ^{2}\right)\right)\overline{Q}_{T}\left(\theta,\,\lambda_{b}\right)\right)\pi\left(\lambda_{b}\right)}{\int_{\varGamma^{0}}\exp\left(\left(\gamma_{T}/\left(T\left\Vert \delta_{T}\right\Vert ^{2}\right)\right)\overline{Q}_{T}\left(\theta,\,\lambda_{b}\right)\right)\pi\left(\lambda_{b}\right)d\lambda_{b}}d\lambda_{b}\label{eq. Definition GLE, Psi(s, theta) eq. 5}\\
 & =\int_{\varGamma^{0}}l\left(s-\lambda_{b}\right)\frac{\exp\left(\left(\gamma_{T}/\left(T\left\Vert \delta_{T}\right\Vert ^{2}\right)\right)\left(G_{T}\left(\theta,\,\lambda_{b}\right)+Q_{T}^{0}\left(\theta,\,\lambda_{b}\right)\right)\right)\pi\left(\lambda_{b}\right)}{\int_{\varGamma^{0}}\exp\left(\left(\gamma_{T}/\left(T\left\Vert \delta_{T}\right\Vert ^{2}\right)\right)\left(G_{T}\left(\theta,\,\lambda_{b}\right)+Q_{T}^{0}\left(\theta,\,\lambda_{b}\right)\right)\right)\pi\left(\lambda_{b}\right)d\lambda_{b}}d\lambda_{b}.\nonumber 
\end{align}
Note that, under Condition \ref{Condition 1 LapBai97} below, this
is equivalent to the minimizer of $\mathcal{R}_{l,T}\left(s\right)$
since $\overline{Q}_{T}\left(\theta,\,\lambda_{b}\right)$ can always
be normalized without affecting its maximization. Different choices
of $\left\{ \gamma_{T}\right\} $ give rise to GL estimators with
different limiting distributions. Using $\delta_{T}$ or any consistent
estimate (e.g., $\widehat{\delta}_{T}^{\mathrm{LS}}$) in the factor
$\gamma_{T}/\left(T\left\Vert \delta_{T}\right\Vert ^{2}\right)$
is irrelevant because they are asymptotically equivalent. Our analysis
is local in nature and thus we write $\widehat{\lambda}_{b}^{\mathrm{GL}}(\widehat{\theta})\triangleq\widehat{\lambda}_{b}^{\mathrm{GL,}*}(r_{T}(\widehat{\theta}-\theta^{0}),\,r_{T}(\widehat{\theta}-\theta^{0})),$
where $r_{T}$ is the convergence rate of $\widehat{\theta}-\theta^{0}$.
Note that $G_{T}\left(\cdot,\,\cdot\right)$ and $Q_{T}^{0}\left(\cdot,\,\cdot\right)$
constitute the stochastic and the deterministic part of the objective
function, respectively. Both depend on $r_{T}(\widehat{\theta}-\theta^{0})$
and our proof proceeds in conditioning first on the effect of $r_{T}(\widehat{\theta}-\theta^{0})$
on the deterministic part to obtain weak convergence of the stochastic
part to a limit process that does not depend on this conditioning.
See below for more details. Hence, it is required to introduce two
indices $\widetilde{v}$ and $v$, such that we define $\widehat{\lambda}_{b}^{\mathrm{GL}}(\widehat{\theta})=\widehat{\lambda}_{b}^{\mathrm{GL,*}}(\widetilde{v},\,v)$
as the minimizer of
\begin{align}
\Psi_{l,T}(s;\,\widetilde{v},\,v) & \triangleq\int_{\varGamma^{0}}l(s-\lambda_{b})\times\label{Eq. (14) - Definition of Psi(s,v,v)}\\
 & \quad\frac{\exp\left(\left(\gamma_{T}/\left(T\left\Vert \delta_{T}\right\Vert ^{2}\right)\right)\left(G_{T}\left(\theta^{0}+\widetilde{v}/r_{T},\,\lambda_{b}\right)+Q_{T}^{0}\left(\theta^{0}+v/r_{T},\,\lambda_{b}\right)\right)\right)\pi\left(\lambda_{b}\right)}{\int_{\varGamma^{0}}\exp\left(\left(\gamma_{T}/\left(T\left\Vert \delta_{T}\right\Vert ^{2}\right)\right)\left(G_{T}\left(\theta^{0}+\widetilde{v}/r_{T},\,\lambda_{b}\right)+Q_{T}^{0}\left(\theta^{0}+v/r_{T},\,\lambda_{b}\right)\right)\right)\pi\left(\lambda_{b}\right)d\lambda_{b}}d\lambda_{b}.\nonumber 
\end{align}
 For each $v,$ we show weak convergence as a function of $\widetilde{v}$
to a limit process that does not depend on $v.$ In a second step,
we use the monotonicity in $v$ of $Q_{T}^{0}$ which, relying on
the argument in \citet{jureckova:77}, allows us to achieve weak convergence
uniformly in $v$. We first show the consistency and rate of convergence
of $\widehat{\lambda}_{b}^{\mathrm{GL}}$. These results imply that
$\theta^{0}$ is estimated as if $T_{b}^{0}$ were known. Thus, $\widehat{\theta}$
is $\sqrt{T}$-consistent and asymptotically normal so that we set
$r_{T}=\sqrt{T}$ hereafter. We first show, for each pair $\left(v,\,\widetilde{v}\right)$
with $v,\,\widetilde{v}\in\mathbf{V}$, the convergence of the marginal
distributions of the sample function $\Psi_{l,T}\left(s;\,v,\,\widetilde{v}\right)$
to the marginal distributions of the random function
\begin{align*}
\Psi_{l}^{0}\left(s\right) & =\int_{\mathbb{R}}l\left(s-u\right)\left(\mathscr{V}\left(u\right)/\int_{\mathbb{R}}\mathscr{V}\left(v\right)dv\right)du,
\end{align*}
 where 
\begin{align}
\mathscr{V}\left(s\right)\triangleq\mathscr{W}\left(s\right)-\varLambda^{0}\left(s\right) & \triangleq\begin{cases}
2\left(\left(\delta^{0}\right)'\Sigma_{1}\delta^{0}\right)^{1/2}W_{1}\left(-s\right)-\left|s\right|\left(\delta^{0}\right)'V_{1}\delta^{0}, & \textrm{if }s\leq0\\
2\left(\left(\delta^{0}\right)'\Sigma_{2}\delta^{0}\right)^{1/2}W_{2}\left(s\right)-s\left(\delta^{0}\right)'V_{2}\delta^{0}, & \textrm{if }s>0,
\end{cases}\label{eq. V(s) Limit Process Theorem Post Mean LapBai97}
\end{align}
 and $W_{1},\,W_{2}$ are independent standard Wiener processes defined
on $[0,\,\infty)$. The limit process $\Psi_{l}^{0}\left(s\right)$
does not depend on $v$ nor $\widetilde{v}$. Next, we show that
the family of probability measures in $\mathbb{C}_{b}\left(\mathbf{K}\right)$,
with $\mathbf{K}\triangleq\left\{ s\in\mathbb{R}:\,\left|s\right|\leq K\textrm{ and }K<\infty\right\} $,
generated by the contractions of $\Psi_{l,T}\left(s;\,\widetilde{v},\,v\right)$
on $\mathbf{K}$ is dense uniformly in $\left(v,\,\widetilde{v}\right)$.
Finally, we examine the oscillations of the minimizers of the sample
criterion $\Psi_{l,T}\left(s;\,v,\,\widetilde{v}\right)$. 

It is important to note that the results derived in this section are
more general than what is required for the structural change model.
The reason is that the change-point model is recovered as a special
case corresponding to $\Psi_{l,T}\left(s\right)=\Psi_{l,T}\left(s;\,0,\,0\right)$.
That is, defining the GL estimator in a $1/r_{T}$-neighborhood of
the slope parameter vector $\theta^{0}$ is not strictly necessary
and one can essentially develop the same analysis with $\theta$ fixed
at its true value $\theta^{0}$. This relies on the properties of
(orthogonal) least-squares projections and would not apply, for example,
to the least absolute deviation (LAD) estimator of the break date
{[}cf. \citet{baiLAD:95}{]} for which $\Psi_{l,T}\left(s;\,\widetilde{v},\,v\right)$
should instead be considered. The same issue is present when estimating
structural changes in the quantile regression model {[}cf. \citet{oka/qu:10}{]}
and in using instrumental variables models {[}cf. \citet{hall/han/boldea:10}
and Perron and Yamamoto \citeyearpar{perron/yamamoto:14,perron/yanamoto:15}{]}.
We establish theoretical results under this more general setting
since they may be useful for future work.

Let $\lambda_{b,T}^{0}\left(v\right)=\lambda_{b,T}^{0}\left(\theta^{0}+v/r_{T}\right)$.
 Introduce the local parameter $u=\psi_{T}\left(\lambda_{b}-\lambda_{b,T}^{0}\left(v\right)\right)$
and let $\pi_{T,v}\left(u\right)\triangleq\pi\left(\lambda_{b,T}^{0}\left(v\right)+u/\psi_{T}\right)$,
$Q_{T,v}\left(u\right)\triangleq Q_{T}^{0}(\theta^{0}+v/r_{T},\,\lambda_{b,T}^{0}$
$\left(v\right)+u/\psi_{T}$), and $\widetilde{G}_{T,v}\left(u,\,\widetilde{v}\right)\triangleq G_{T}\left(\theta^{0}+\widetilde{v}/r_{T},\,\lambda_{b,T}^{0}\left(v\right)+u/\psi_{T}\right)$,
where the sequence $\left\{ \psi_{T}\right\} $ depends on the results
on consistency and rate of convergence of $\widehat{\lambda}_{b}^{\mathrm{GL}}$
in Proposition \ref{Proposition: Consistency and Rate of Convergence}.
Apply a simple substitution in \eqref{Eq. (14) - Definition of Psi(s,v,v)}
to yield,
\begin{align}
\Psi_{l,T}\left(s;\,\widetilde{v},\,v\right) & =\int_{\Gamma_{T}}l\left(s-u\right)\frac{\exp\left(\left(\gamma_{T}/T\left\Vert \delta_{T}\right\Vert ^{2}\right)\left(\widetilde{G}_{T,v}\left(u,\,\widetilde{v}\right)+Q_{T,v}\left(u\right)\right)\right)\pi_{T,v}\left(u\right)du}{\int_{\Gamma_{T}}\exp\left(\left(\gamma_{T}/T\left\Vert \delta_{T}\right\Vert ^{2}\right)\left(\widetilde{G}_{T,v}\left(w,\,\widetilde{v}\right)+Q_{T,v}\left(w\right)\right)\right)\pi_{T,v}\left(w\right)dw},\label{eq. (17)-1-1}
\end{align}
 where $\Gamma_{T}\triangleq\left\{ u\in\mathbb{R}:\,\lambda_{b}^{0}+u/\psi_{T}\in\varGamma^{0}\right\} $.
  
\begin{assumption}
\label{A.9a Bai 97}$\left\{ \left(z_{t},\,e_{t}\right)\right\} $
is second-order stationary within each regime such that $\mathbb{E}\left(z_{t}z'_{t}\right)=V_{1}$
and $\mathbb{E}\left(e_{t}^{2}\right)=\sigma_{1}^{2}$ for $t\leq T_{b}^{0}$
and $\mathbb{E}\left(z_{t}z'_{t}\right)=V_{2}$ and $\mathbb{E}\left(e_{t}^{2}\right)=\sigma_{2}^{2}$
for $t>T_{b}^{0}$. 
\end{assumption}

\begin{assumption}
\label{Assumption A.9b Bai 97, LapBai97}For $r\in\left[0,\,1\right],$
$\left(T_{b}^{0}\right)^{-1/2}\sum_{t=1}^{\left\lfloor rT_{b}^{0}\right\rfloor }z_{t}e_{t}\Rightarrow\mathscr{G}_{1}\left(r\right)$
and $\left(T-T_{b}^{0}\right)^{-1/2}\sum_{t=T_{b}^{0}+1}^{T_{b}^{0}+\left\lfloor r\left(T-T_{b}^{0}\right)\right\rfloor }$
$z_{t}e_{t}\Rightarrow\mathscr{G}_{2}\left(r\right)$, where $\mathscr{G}_{i}\left(\cdot\right)$
is a multivariate Gaussian process on $\left[0,\,1\right]$ with zero
mean and covariance $\mathbb{E}\left[\mathscr{G}_{i}\left(u\right),\,\mathscr{G}_{i}\left(s\right)\right]=\min\left\{ u,\,s\right\} \Sigma_{i}$
$\left(i=1,\,2\right)$, and $\Sigma_{1}\triangleq\lim_{T\rightarrow\infty}\mathbb{E}\left[\left(T_{b}^{0}\right)^{-1/2}\sum_{t=1}^{T_{b}^{0}}z_{t}e_{t}\right]^{2}$,
$\Sigma_{2}\triangleq\lim_{T\rightarrow\infty}\mathbb{E}\left[\left(T-T_{b}^{0}\right)^{-1/2}\sum_{t=T_{b}^{0}+1}^{T}z_{t}e_{t}\right]^{2}$.
Furthermore, for any $0<r_{0}<1$ with $r_{0}<\lambda_{0}$, $T^{-1}\sum_{t=\left\lfloor r_{0}T\right\rfloor +1}^{\left\lfloor \lambda_{0}T\right\rfloor }z_{t}z'_{t}\overset{\mathbb{P}}{\rightarrow}\left(\lambda_{0}-r_{0}\right)V_{1},$
and with $\lambda_{0}<r_{0}$ $T^{-1}\sum_{t=\left\lfloor \lambda_{0}T\right\rfloor +1}^{\left\lfloor r_{0}T\right\rfloor }$
$z_{t}z'_{t}\overset{\mathbb{P}}{\rightarrow}\left(r_{0}-\lambda_{0}\right)V_{2}$
so that $\lambda_{-}$ and $\lambda_{+}$ (the minimum and maximum
of the eigenvalues of the last two matrices) satisfy $0<\lambda_{-}\leq\lambda_{+}<\infty.$ 
\end{assumption}
Assumptions \ref{A.9a Bai 97}-\ref{Assumption A.9b Bai 97, LapBai97}
are equivalent to A9 in \citet{bai:97RES} and A7 in \citet{bai/perron:98}.
More specifically, Assumption \ref{Assumption A.9b Bai 97, LapBai97}
requires that, within each regime, an Invariance Principle holds for
$\left\{ z_{t}e_{t}\right\} .$ Let $\zeta_{t}\triangleq z_{t}e_{t}$.
For $u\leq0$ let $g\left(\zeta_{t};\,u\right)\triangleq\left(\delta^{0}\right)'\sum_{t=T_{b}^{0}+\left\lfloor u/v_{T}^{2}\right\rfloor }^{T_{b}^{0}}\zeta_{t}$
and $\widetilde{g}\left(\zeta_{t};\,u,\,\widetilde{v},\,v;\,\psi_{T},\,r_{T}\right)\triangleq\sqrt{\psi_{T}}\left(\delta^{0}+\widetilde{v}/r_{T}\right)'\sum_{t=T\lambda_{b}^{0}\left(\theta^{0}+v/r_{T}\right)+\left\lfloor u/\psi_{T}\right\rfloor }^{T\lambda_{b}^{0}\left(\theta^{0}+v/r_{T}\right)}\zeta_{t}.$
Define analogously $g\left(\zeta_{t};\,u\right)$ and $\widetilde{g}\left(\zeta_{t};\,u,\,\widetilde{v},\,v;\,\psi_{T},\,r_{T}\right)$
for the case $T_{b}>T_{b}^{0}$. We now present some technical assumptions
that are necessary for the derivation of the asymptotic results for
the GL estimate. 
\begin{assumption}
\label{Assumption Gaussian Process for Lap LapBai97}For some neighborhood
 $\Theta^{0}\subset\mathbf{S}$ of $\theta^{0},$ (i) for all $\lambda_{b}\neq\lambda_{b}^{0},\,\widetilde{Q}\left(\theta^{0},\,\lambda_{b}\right)<\widetilde{Q}\left(\theta^{0},\,\lambda_{b}^{0}\right)$;
(ii) for any $v,\,\widetilde{v}_{1},\,\widetilde{v}_{2}\in\mathbf{V}$
and $u,\,s\in\mathbb{R},$ 
\[
\varSigma\left(u,\,s\right)\triangleq\lim_{T\rightarrow\infty}\mathbb{E}\left[\widetilde{g}\left(\zeta_{t};\,u,\,\widetilde{v}_{1},\,v;\,\psi_{T},\,r_{T}\right)\widetilde{g}\left(\zeta_{t};\,s,\,\widetilde{v}_{2},\,v;\,\psi_{T},\,r_{T}\right)'\right],
\]
does not depend on $v,\,\widetilde{v}_{1},\,\widetilde{v}_{2}\in\mathbf{V}$.
\end{assumption}
Part (i) of Assumption \ref{Assumption Gaussian Process for Lap LapBai97}
is an identification condition. Assumption \ref{Assumption Gaussian Process for Lap LapBai97}-(ii)
holds whenever $\widehat{\lambda}_{b}$ is consistent. With Assumption
\ref{Assumption Gaussian Process for Lap LapBai97}-(ii) we fully
characterize the Gaussian component of the limit process $\mathscr{V}\left(\cdot\right);$
it implies that $\varSigma\left(\cdot,\,\cdot\right)$ is strictly
positive and that 
\begin{align}
\forall u,\,s\in\mathbb{R}: & \begin{cases}
\forall c>0:\,\varSigma\left(cu,\,cs\right)=c\varSigma\left(u,\,s\right),\\
\varSigma\left(u,\,u\right)+\varSigma\left(s,\,s\right)-2\varSigma\left(u,\,s\right)=\varSigma\left(u-s,\,u-s\right),
\end{cases}\label{Eq. (18) Covariance of Limit Gaussina process}
\end{align}
 where the second implication requires some simple but tedious manipulations.
Finally, the following assumption is automatically satisfied if $l\left(\cdot\right)$
is a convex function with a unique minimum. 
\begin{assumption}
\label{Assumption Uniquess loss Fucntion LapBai97}$\xi_{l}^{0}\triangleq\xi\left(\lambda_{b}^{0}\right)$
is uniquely defined by 
\begin{align*}
\Psi_{l}\left(\xi_{l}^{0}\right)\triangleq\inf_{s}\Psi_{l}\left(s\right)=\inf_{s}\int_{\mathbb{R}}l\left(s-u\right)\left(\exp\left(\mathscr{V}\left(u\right)\right)/\left(\int_{\mathbb{R}}\exp\left(\mathscr{V}\left(w\right)\right)dw\right)\right)du & .
\end{align*}
 
\end{assumption}

\subsection{\label{Sub section Theorems Laplace in SC}Asymptotic Results for
the GL Estimate}

We first show the consistency and rate of convergence of the GL estimator.
The latter allows us to characterize the rate of $\psi_{T}$ and proceed
with the asymptotic analysis in a neighborhood of $\lambda_{b}^{0}$.
In practice, the squared loss function is often employed. Hence, it
is useful to first present in Theorem \ref{Theorem Posterior Mean Lap Estimation in Bai 97}
the theoretical results for this case for which the GL estimator is
$\widehat{\lambda}_{b}^{\mathrm{GL}}=\int_{\varGamma^{0}}\lambda_{b}p_{T}\left(\lambda_{b}\right)d\lambda_{b},$
i.e., the Quasi-posterior mean. This allows us to keep the theoretical
results tractable and provide the main intuition without the need
of  complex notation. This case is also instructive since we can
compare our results with corresponding ones for the least-squares
and Bayesian change-point estimators. Corresponding results for general
loss functions  are given in Theorem \ref{Theorem Geneal Laplace Estimator LapBai97}.

\subsubsection{Consistency and Rate of Convergence}

The rate of convergence is similar to that of the LS estimator; the
difference being that $\psi_{T}=T^{1-2\vartheta}$ with $\vartheta\in\left(0,\,1/2\right)$
for the LS estimator and $\vartheta\in\left(0,\,1/4\right)$ for the
GL estimator. 
\begin{prop}
\label{Proposition: Consistency and Rate of Convergence}Under Assumptions
\ref{Assumption A Bai97}-\ref{Assumption A4 Bai97}, \ref{Assumption The-loss-function LapBai97}-\ref{Assumption Small Shift BP}
and \ref{Assumption Gaussian Process for Lap LapBai97}-(i): (i) $\widehat{\lambda}_{b}^{\mathrm{GL}}=\lambda_{b}^{0}+o_{\mathbb{P}}\left(1\right)$;
(ii) $\widehat{\lambda}_{b}^{\mathrm{GL}}=\lambda_{b}^{0}+O_{\mathbb{P}}\left(\left(T\left\Vert \delta_{T}\right\Vert ^{2}\right)^{-1}\right)$.
\end{prop}

\subsubsection{The Asymptotic Distribution of the Quasi-posterior Mean}

For the squared loss function $\widehat{\lambda}_{b}^{\mathrm{GL}}\left(\widehat{\theta}\right)\triangleq\widehat{\lambda}_{b}^{\mathrm{GL,}*}\left(\widetilde{v},\,v\right),$
where 
\begin{align}
\widehat{\lambda}_{b}^{\mathrm{GL},*}\left(\widetilde{v},\,v\right) & \triangleq\frac{\int_{\varGamma^{0}}\lambda_{b}\exp\left(\left(\gamma_{T}/\left(T\left\Vert \delta_{T}\right\Vert ^{2}\right)\right)\left(G_{T}\left(\theta^{0}+\widetilde{v}/r_{T},\,\lambda_{b}\right)+Q_{T}^{0}\left(\theta^{0}+v/r_{T},\,\lambda_{b}\right)\right)\right)\pi\left(\lambda_{b}\right)d\lambda_{b}}{\int_{\varGamma^{0}}\exp\left(\left(\gamma_{T}/\left(T\left\Vert \delta_{T}\right\Vert ^{2}\right)\right)\left(G_{T}\left(\theta^{0}+\widetilde{v}/r_{T},\,\lambda_{b}\right)+Q_{T}^{0}\left(\theta^{0}+v/r_{T},\,\lambda_{b}\right)\right)\right)\pi\left(\lambda_{b}\right)d\lambda_{b}},\label{Eq. (14) - Definition of lambda*(v,v)}
\end{align}
 and $v,$ $\widetilde{v}$ each belong to some compact set $\mathbf{V}\subset\mathbb{R}^{p+2q}$.
For each $v\in\mathbf{V},$ we consider $\widehat{\lambda}_{b}^{\mathrm{GL,*}}\left(\cdot,\,v\right)$
as a random process with paths in $\mathbb{D}_{b}\left(\mathbf{V}\right)$.
We focus on the weak convergence of $\widehat{\lambda}_{b}^{\mathrm{GL},*}\left(\cdot,\,v\right)$
for fixed $v$ since the limit process is independent of $v$ and
constant as a function of $\widetilde{v}$; we then exploit monotonicity
in $v.$  More precisely, we will show that for $\lambda_{b,T}^{0}\left(v\right)=\lambda_{b,T}^{0}\left(\theta^{0}+v/r_{T}\right)$
and diverging sequences $\left\{ \gamma_{T}\right\} $ and $\left\{ r_{T}\right\} $,
the sequence $a_{T}\left(\widehat{\lambda}_{b}^{\mathrm{GL},*}\left(\widetilde{v},\,v\right)-\lambda_{b,T}^{0}\left(v\right)\right)$
converges in distribution in $\mathbb{D}_{b}\left(\mathbf{V}\right)$
for each $v$ to a limit process not depending on $v$ nor $\widetilde{v}$.
Since it is monotonic in $v$, we do not need to show uniform convergence
directly.  Introduce the local parameter $u=\psi_{T}\left(\lambda_{b}-\lambda_{b,T}^{0}\left(v\right)\right)$;
a simple substitution in \eqref{Eq. (14) - Definition of lambda*(v,v)}
yields, 
\begin{align}
\psi_{T}\left(\widehat{\lambda}_{b}^{\mathrm{GL},*}\left(\widetilde{v},\,v\right)-\lambda_{b,T}^{0}\left(v\right)\right) & =\frac{\int_{\mathbb{R}}u\exp\left(\left(\gamma_{T}/\left(T\left\Vert \delta_{T}\right\Vert ^{2}\right)\right)\left(\widetilde{G}_{T,v}\left(u,\,\widetilde{v}\right)+Q_{T,v}\left(u\right)\right)\right)\pi_{T,v}\left(u\right)du}{\int_{\mathbb{R}}\exp\left(\left(\gamma_{T}/\left(T\left\Vert \delta_{T}\right\Vert ^{2}\right)\right)\left(\widetilde{G}_{T,v}\left(u,\,\widetilde{v}\right)+Q_{T,v}\left(u\right)\right)\right)\pi_{T,v}\left(u\right)du},\label{eq. (17)-1}
\end{align}
where again we have used the notation $\pi_{T,v}\left(u\right)=\pi\left(\lambda_{b,T}^{0}\left(v\right)+u/\psi_{T}\right),$
$Q_{T,v}\left(u\right)=Q_{T}^{0}\left(\theta^{0}+v/r_{T}\right.$
$\left.\lambda_{b,T}^{0}\left(v\right)+u/\psi_{T}\right)$ and $\widetilde{G}_{T,v}\left(u,\,\widetilde{v}\right)=G_{T}\left(\theta^{0}+\widetilde{v}/r_{T},\,\lambda_{b,T}^{0}\left(v\right)+u/\psi_{T}\right)$.
The limit of the GL estimator depends on the limit of the process
$\left(\gamma_{T}/\left(T\left\Vert \delta_{T}\right\Vert ^{2}\right)\right)\left(\widetilde{G}_{T,v}\left(u,\,\widetilde{v}\right)+Q_{T,v}\left(u\right)\right)$.
As part of the proof of Theorem \ref{Theorem Posterior Mean Lap Estimation in Bai 97},
we show that the sequence of processes $\left\{ \widetilde{G}_{T,v}\left(u,\,\widetilde{v}\right),\,T\geq1\right\} $
converges weakly in $\mathbb{D}_{b}\left(\mathbb{R}\times\mathbf{V}\right)$
to a Gaussian process $\mathscr{W}$ not varying with $v$, whereas
$Q_{T,v}\left(\cdot\right)$ is approximated by a (deterministic)
drift process taking negative values, and is monotonic in $v$ and
flat in $\widetilde{v}$.  We show that this implies that $\widehat{\lambda}_{b}^{\mathrm{GL},*}\left(\widetilde{v},\,v\right)-\lambda_{b,T}^{0}\left(v\right)$
is monotonic in $v$ which then leads to uniform convergence in $v$
following the argument of \citet{jureckova:77}.

In anticipation of the results, we make a few comments about the notation
for the weak convergence of processes on the space of bounded càdlàg
functions $\mathbb{D}_{b}$. Let $\mathbf{V}\subset\mathbb{R}^{p+2q}$
be a compact set. Let $W_{T}\left(u,\,\widetilde{v},\,v\right)$ denote
an arbitrary sample process with bounded càdlàg paths evaluated at
the local parameters $u\in\mathbb{R},$ and $v,\,\widetilde{v}\in\mathbf{V}$.
For each fixed $v\in\mathbf{V}$, we shall write $W_{T}\left(u,\,\widetilde{v},\,v\right)\Rightarrow\mathscr{W}\left(u,\,\widetilde{v},\,v\right)$
in $\mathbb{D}_{b}\left(\mathbb{R}\times\mathbf{V}\right)$ whenever
the process $W_{T}\left(\cdot,\,\cdot,\,v\right)$ converges weakly
to $\mathscr{W}\left(\cdot,\,\cdot,\,v\right)$, where $\mathscr{W}\left(\cdot,\,\cdot,\,v\right)$
also belongs to $\mathbb{D}_{b}\left(\mathbb{R}\times\mathbf{V}\right)$.
As a shorthand, we shall omit the argument $u\,\left(\widetilde{v}\right)$
if the limit process does not depend on $u\,\left(\widetilde{v}\right)$.
The same notational conventions are used for the case when $W_{T}$
is only a function of $\left(\widetilde{v},\,v\right).$ In Theorem
\ref{Theorem Posterior Mean Lap Estimation in Bai 97} the convergence
holds for every $v\in\mathbf{V}$, stated as convergence in $\mathbb{D}_{b}$. 
\begin{condition}
\label{Condition 1 LapBai97}As $T\rightarrow\infty$ there exist
a positive finite number $\kappa_{\gamma}$ such that $\gamma_{T}/T\left\Vert \delta_{T}\right\Vert ^{2}\rightarrow\kappa_{\gamma}$.
\end{condition}
\begin{thm}
\label{Theorem Posterior Mean Lap Estimation in Bai 97}Assume $l\left(\cdot\right)$
is the squared loss function. Under Assumptions \ref{Assumption A Bai97}-\ref{Assumption A4 Bai97}
and \ref{Assumption The-loss-function LapBai97}-\ref{Assumption Uniquess loss Fucntion LapBai97},
and Condition \ref{Condition 1 LapBai97}, then in $\mathbb{D}_{b}$,
\begin{align}
T\left\Vert \delta_{T}\right\Vert ^{2}\left(\widehat{\lambda}_{b}^{\mathrm{GL}}-\lambda_{b}^{0}\right) & \Rightarrow\frac{\int u\exp\left(\mathscr{W}\left(u\right)-\varLambda^{0}\left(u\right)\right)du}{\int\exp\left(\mathscr{W}\left(u\right)-\varLambda^{0}\left(u\right)\right)du}\triangleq\int up_{0}^{*}\left(u\right)du,\label{eq. Asy Dist Laplace Bai97 part (ii)}
\end{align}
where $\mathscr{W}\left(\cdot\right)$ and $\varLambda^{0}\left(\cdot\right)$
are defined in \eqref{eq. V(s) Limit Process Theorem Post Mean LapBai97}.
\end{thm}
Theorem \ref{Theorem Posterior Mean Lap Estimation in Bai 97} states
that the asymptotic distribution of the GL estimate is a ratio of
integrals of functions of tight Gaussian processes. We shall compare
this result with the limiting distribution of the Bayesian change-point
estimator of \citet{ibragimov/has:81}. They considered a simple diffusion
process with a change-point in the deterministic drift {[}see their
eq. (2.17) on pp. 338{]}. The limiting distribution of the GL estimate
from Theorem \ref{Theorem Posterior Mean Lap Estimation in Bai 97}
for the case of a break in the mean for model \eqref{Eq. the Single Break Model}
is essentially the same (and exactly so in the i.i.d. case with stationary
regimes) as theirs. Hence, while the GL estimator has a classical
(frequentist) interpretation, it is first-order equivalent in law
to a corresponding Bayes-type estimator. 

We now present a result about the dual nature of the limiting distribution
of the GL estimator. The following proposition shows that, under different
conditions on the smoothing sequence parameter $\left\{ \gamma_{T}\right\} $,
the GL estimator achieves different limiting distributions.
\begin{condition}
\label{Condition 2 LapBai97}As $T\rightarrow\infty$, $T\left\Vert \delta_{T}\right\Vert ^{2}/\gamma_{T}=o\left(1\right)$.
\end{condition}
\begin{prop}
\label{Proposition}Assume $l\left(\cdot\right)$ is the squared loss
function. Under Assumptions \ref{Assumption A Bai97}-\ref{Assumption A4 Bai97}
and \ref{Assumption The-loss-function LapBai97}-\ref{Assumption Uniquess loss Fucntion LapBai97},
and Condition \ref{Condition 2 LapBai97}, $T\left\Vert \delta_{T}\right\Vert ^{2}\left(\widehat{\lambda}_{b}^{\mathrm{GL}}-\lambda_{b}^{0}\right)\Rightarrow\arg\max_{s\in\mathbb{R}}\mathscr{V}\left(s\right)$
in $\mathbb{D}_{b}$, with $\mathscr{V}\left(\cdot\right)$ defined
in \eqref{eq. V(s) Limit Process Theorem Post Mean LapBai97}.
\end{prop}
\begin{cor}
\label{Corollary part (i) of Theorem LapBai97}Define $\Xi_{e}\triangleq\left(\delta^{0}\right)'\Sigma_{2}\delta^{0}/\left(\delta^{0}\right)'\Sigma_{1}\delta^{0}$
and $\Xi_{Z}\triangleq\left(\delta^{0}\right)'V_{2}\delta^{0}/\left(\delta^{0}\right)'V_{1}\delta^{0}$.
Under Assumptions \ref{Assumption A Bai97}-\ref{Assumption A4 Bai97}
and \ref{Assumption The-loss-function LapBai97}-\ref{Assumption Uniquess loss Fucntion LapBai97},
and Condition \ref{Condition 2 LapBai97}, $\left(\left(\delta'_{T}V_{1}\delta{}_{T}\right)^{2}/\delta'_{T}\Sigma_{1}\delta_{T}\right)\left(\widehat{T}_{b}^{\mathrm{GL}}-T_{b,T}^{0}\right)\overset{d}{\rightarrow}\arg\max_{s\in\mathbb{R}}\mathscr{V}^{*}\left(s\right)$
in $\mathbb{D}_{b}$ where 
\begin{align*}
\mathscr{V}^{*}\left(s\right)=W_{1}\left(-s\right)-\left|s\right|/2\,\mathrm{\,if\,}\,s\leq0 & ;\,\,\mathscr{V}^{*}\left(s\right)=\Xi_{e}^{1/2}W_{2}\left(s\right)-\Xi_{Z}s/2\,\,\mathrm{if}\,\,s>0.
\end{align*}
\end{cor}
Corollary \ref{Corollary part (i) of Theorem LapBai97} and Proposition
\ref{Proposition} show that with enough smoothing applied, the GL
estimator is (first-order) asymptotically equivalent to the least-squares
or MLE {[}cf. \citet{bai:97RES} and \citet{yao:87}, respectively{]}.
The intuition is that when the criterion function is sufficiently
smoothed, the Quasi-posterior probability density converges to the
generalized dirac probability measure concentrated at the argmax of
the limit criterion function. This is analogous to a well-known result
{[}cf. Corollary 5.11 in \citet{robert/casella:04}{]}, stating that
in a parametric statistical experiment indexed by a parameter $\theta\in\Theta,$
the MLE $\widehat{\theta}_{T}^{\mathrm{ML}}$ is the limit of a Bayes
estimator as the smoothing parameter $\gamma\rightarrow\infty$, i.e.,
using obvious notation: 
\begin{align*}
\widehat{\theta}_{T}^{\mathrm{ML}} & =\arg\max_{\theta\in\Theta}L_{T}\left(\theta\right)=\lim_{\gamma\rightarrow\infty}\frac{\int_{\Theta}\theta\exp\left(\gamma L_{T}\left(\theta\right)\right)\pi\left(\theta\right)d\theta}{\int_{\Theta}\exp\left(\gamma L_{T}\left(\theta\right)\right)\pi\left(\theta\right)d\theta}.
\end{align*}

\subsubsection{The Asymptotic Distribution for General Loss Functions}

For general loss functions satisfying Assumption \ref{Assumption The-loss-function LapBai97},
Theorem \ref{Theorem Geneal Laplace Estimator LapBai97} shows that
$T\left\Vert \delta_{T}\right\Vert ^{2}\left(\widehat{\lambda}_{b}^{\mathrm{GL}}-\lambda_{b}^{0}\right)$
is (first-order) asymptotically equivalent to $\xi_{l}^{0}$ defined
by 
\begin{align}
\Psi_{l}\left(\xi_{l}^{0}\right) & \triangleq\inf_{r}\Psi_{l}\left(r\right)=\inf_{r\in\mathbb{R}}\left\{ \int_{\mathbb{R}}l\left(r-u\right)p_{0}^{*}\left(u\right)du\right\} .\label{Eq. Distribution of Bayes Estimator-1}
\end{align}

\begin{thm}
\label{Theorem Geneal Laplace Estimator LapBai97}Under Assumptions
\ref{Assumption A Bai97}-\ref{Assumption A4 Bai97} and \ref{Assumption The-loss-function LapBai97}-\ref{Assumption Uniquess loss Fucntion LapBai97},
and Condition \ref{Condition 1 LapBai97}, for $l\in\boldsymbol{L},$
$T\left\Vert \delta_{T}\right\Vert ^{2}(\widehat{\lambda}_{b}^{\mathrm{GL}}-\lambda_{b}^{0})$
$\Rightarrow\xi_{l}^{0}$ as defined by \eqref{Eq. Distribution of Bayes Estimator-1}.
\end{thm}
The existence and uniqueness of $\xi_{l}^{0}$ follow from Assumption
\ref{Assumption Uniquess loss Fucntion LapBai97}. If one interprets
$p_{0}^{*}\left(u\right)$ as a true posterior density function, then
$\xi_{l}^{0}$ would naturally be viewed as a Bayesian estimator for
the loss function $l_{T}\left(\cdot\right).$ In particular, in analogy
to the above comparison with the Bayesian estimator of \citet{ibragimov/has:81},
one can interpret the GL estimator as a Quasi-Bayesian estimator.
While this is by itself a theoretically interesting result, we actually
exploit it to construct more reliable inference methods about the
date of a structural change. Under the least-absolute deviation loss,
the GL estimator converges in distribution to the median of $p_{0}^{*}\left(u\right)$.
We shall use the results in Theorem \ref{Theorem Posterior Mean Lap Estimation in Bai 97}-\ref{Theorem Geneal Laplace Estimator LapBai97}
but not Proposition \ref{Proposition} since the latter implies the
same confidence intervals as in \citet{bai:97RES} and \citet{bai/perron:98}.
GL inference based on the Bayes-type limiting distribution provides
a more accurate description of the uncertainty over the parameter
space than the inference based on the density of $\arg\max_{s\in\mathbb{R}}\mathscr{V}\left(s\right)$
which underestimates uncertainty as shown by confidence intervals
with empirical coverage rates below the nominal level particularly
when the magnitude of the break is small (see Section \ref{Section Monte-Carlo-Simulation}).
After some investigation, we found that both estimation and inference
under the least-absolute loss works well and this is what will be
used in our simulation study. 

\section{\label{Section Inference Methods}Confidence Sets Based on the GL
Estimator}

In this section, we discuss inference procedures for the break date
based on the large-sample results of the previous section. Inference
under general loss functions based on Theorem \ref{Theorem Geneal Laplace Estimator LapBai97}
is what we recommend to use in practice, in particular with an absolute
loss function.

Since the limiting distribution from Theorem \ref{Theorem Geneal Laplace Estimator LapBai97}
involves certain unknown quantities, we begin by assuming that they
can be replaced by consistent estimates. They are easy to construct
{[}cf. \citet{bai:97RES} and \citet{bai/perron:98}; see also Section
\ref{Section Monte-Carlo-Simulation}{]}. 
\begin{assumption}
\label{Assumption Laplace Inference LapBai97}There exist sequences
of estimators $\widehat{\lambda}_{b,T},\,\widehat{\delta}_{T},\,\widehat{\Xi}_{Z,T},$
and $\widehat{\Xi}_{e,T}$ such that $\widehat{\lambda}_{b,T}=\lambda_{0}+o_{\mathbb{P}}\left(1\right)$,
$\widehat{\delta}_{T}=\delta_{T}+o_{\mathbb{P}}\left(1\right)$, $\widehat{\Xi}_{Z,T}=\Xi_{Z}+o_{\mathbb{P}}\left(1\right)$
and $\widehat{\Xi}_{e,T}=\Xi_{e}+o_{\mathbb{P}}\left(1\right)$.
Furthermore, for all $u,\,s\in\mathbb{R}$ and any $c>0$, there exist
covariation processes $\widehat{\varSigma}_{i,T}\left(\cdot\right)$
$\left(i=1,\,2\right)$ that satisfy (i) $\widehat{\varSigma}_{i,T}\left(u,\,s\right)=\varSigma_{i}^{0}\left(u,\,s\right)+o_{\mathbb{P}}\left(1\right)$,
(ii) $\widehat{\varSigma}_{i,T}\left(u-s,\,u-s\right)=\widehat{\varSigma}_{i,T}\left(u,\,u\right)+\widehat{\varSigma}_{i,T}\left(s,\,s\right)-2\widehat{\varSigma}_{i,T}\left(s,\,u\right)$,
(iii) $\widehat{\varSigma}_{i,T}\left(cu,\,cu\right)=c\widehat{\varSigma}_{i,T}\left(u,\,u\right)$,
(iv) $\mathbb{E}\left\{ \sup_{\left\Vert u\right\Vert =1}\widehat{\varSigma}_{i,T}^{2}\left(u,\,u\right)\right\} =O\left(1\right)$.
\end{assumption}
The first part and (i) of the second part follow from consistency
of $\widehat{\lambda}_{b,T}$ and from an Invariance Principle {[}cf.
Assumption \ref{Assumption A.9b Bai 97, LapBai97}{]}. Part (ii)-(iii)
are implied by Assumption \ref{Assumption Gaussian Process for Lap LapBai97}-(ii)
and consistency of $\widehat{\lambda}_{b,T}$. Let $\left\{ \widehat{\mathscr{W}}_{T}\right\} $
be a (sample-size dependent) sequence of two-sided zero-mean Gaussian
processes with covariance $\widehat{\varSigma}_{T}.$ Construct the
process $\widehat{\mathscr{V}}_{T}$ by replacing the population quantities
in $\mathscr{V}$ by their corresponding estimates from the first
part of Assumption \ref{Assumption Laplace Inference LapBai97} and
further, replace $\mathscr{W}$ by $\widehat{\mathscr{W}}_{T}$. Assumption
\ref{Assumption Laplace Inference LapBai97}-(i) basically implies
that the finite-dimensional limit law of $\left\{ \widehat{\mathscr{W}}_{T}\right\} $
is the same as the finite-dimensional laws of $\mathscr{W}$ while
parts (ii)-(iii) are needed for the integrability of the transform
$\exp\left(\widehat{\mathscr{V}}_{T}\left(\cdot\right)\right)$. Part
(iv) is needed for the proof of the asymptotic stochastic equicontinuity
of $\left\{ \widehat{\mathscr{W}}_{T}\right\} $. Let $\widehat{\xi}_{T}$
be defined as the sample analogue of $\xi_{l}^{0}$ that uses $\widehat{\mathscr{V}}_{T}\left(v\right)$
in place of $\mathscr{V}_{T}\left(v\right)$ in \eqref{Eq. Distribution of Bayes Estimator-1}.
The distribution of $\widehat{\xi}_{T}$ can be evaluated numerically.
\begin{prop}
\label{Proposition Inference Limi Distr LapBai97}Let $l\in\boldsymbol{L}$
be continuous. Under Assumption \ref{Assumption Laplace Inference LapBai97},
$\widehat{\xi}_{T}$ converges in distribution to the limiting distribution
in Theorem \ref{Theorem Geneal Laplace Estimator LapBai97}. 
\end{prop}
The asymptotic distribution theory of the GL estimator may be exploited
in several ways to conduct inference about the break date. The finite-sample
distribution of the LS nd GL estimate of the break date displays significant
non-standard features (cf. Figure \ref{Fig1}). Hence, a conventional
two-sided confidence interval may not result in a confidence set with
reliable properties across all break magnitudes and break locations.
Thus, as in \citet{casini/perron_CR_Single_Break}, we use the concept
of Highest Quasi-posterior Density (HQPD) regions, defined analogously
to the Highest Density Region (HDR); cf. \citet{hyndman:96}. See
also \citet{samworth/wand:10} and \citeauthor{mason/polonik:09}
(2008, 2009)\nocite{mason/polonik:08} for more recent developments.
For an illustrative example on the properties of the HDR see the discussion
of Figure 11 in \citet{casini/perron_CR_Single_Break}. 
\begin{defn}
\label{def:Highest-Density-Region}Highest Density Region: Let the
density function $f_{Y}\left(y\right)$ of a random variable $Y$
defined on a probability space $\left(\Omega_{Y},\,\mathscr{F}_{Y},\,\mathbb{P}_{Y}\right)$
and taking values on the measurable space $\left(\mathcal{Y},\,\mathscr{Y}\right)$
be continuous and bounded. The $\left(1-\alpha\right)100\%$ Highest
Density Region is a subset $\mathbf{S}\left(\kappa_{\alpha}\right)$
of $\mathcal{Y}$ defined as $\mathbf{S}\left(\kappa_{\alpha}\right)=\left\{ y:\,f_{Y}\left(y\right)\geq\kappa_{\alpha}\right\} $
where $\kappa_{\alpha}$ is the largest constant that satisfies $\mathbb{P}_{Y}\left(Y\in\mathbf{S}\left(\kappa_{\alpha}\right)\right)\geq1-\alpha$. 
\end{defn}
For $s=T\left\Vert \delta_{T}\right\Vert ^{2}\left(\widehat{\lambda}_{b}^{\mathrm{LS}}-\lambda_{b}^{0}\right)$,
the asymptotic distribution theory of \citet{bai:97RES} suggests
a belief $\pi\left(s\right)$ over $s\in\mathbb{R}.$ This belief
function can be used as a Quasi-prior for $\lambda_{b}$ in the definition
of the Quasi-posterior $p_{T}\left(\lambda_{b}\right).$ Let $\mu\left(\lambda_{b}\right)$
denote some density function defined by the Radon-Nikodym equation
$\mu\left(\lambda_{b}\right)=dp_{T}\left(\lambda_{b}\right)/d\lambda_{\mathrm{L}},$
where $\lambda_{\mathrm{L}}$ denotes the Lebesgue measure. The following
algorithm describes how to construct a confidence set for $T_{b}^{0}$.
\begin{lyxalgorithm}
\label{Alg: HQPD}GL HQDR-based Confidence Sets for $T_{b}^{0}$:\\
(1) Estimate by least-squares the break date and the regression coefficients
from model \eqref{Eq Matrix Format of the model};\\
(2) Set the Quasi-prior $\pi\left(\lambda_{b}\right)$ equal to the
probability density of the limiting distribution from Corollary \ref{Corollary part (i) of Theorem LapBai97};
\\
(3) Construct the Quasi-posterior given in \eqref{Eq. Quasi-posterior-1};
\\
(4) Obtain numerically the density $\mu\left(\lambda_{b}\right)$
as explained above and label it by $\widehat{\mu}\left(\lambda_{b}\right)$;\\
(5) Compute the Highest Quasi-Posterior Density (HQPD) region of the
probability distribution $\widehat{p}_{T}\left(\lambda_{b}\right)$
and include the point $T_{b}$ in the level $\left(1-\alpha\right)\%$
confidence set $C_{\mathrm{HQPD}}\left(\mathrm{cv}_{\alpha}\right)$
if $T_{b}$ satisfies Definition \ref{def:Highest-Density-Region}.

If a general Quasi-prior $\pi\left(\lambda_{b}\right)$ is used, one
begins directly with step 3.
\end{lyxalgorithm}
In principle, any Quasi-prior $\pi\left(\lambda_{b}\right)$ satisfying
Assumption \ref{Assumption Prior LapBai97} can be used. Note that
$C_{\mathrm{HQPD}}\left(\mathrm{cv}_{\alpha}\right)$ retains a frequentist
interpretation, since no parametric likelihood function of the data
is required.

\section{\label{Section Models with Multiple Breaks}Models with Multiple
Change-Points}

Following \citet{bai/perron:98}, the multiple linear regression model
with $m$ change-points is
\begin{align*}
y_{t} & =w_{t}'\phi^{0}+z'_{t}\delta_{j}^{0}+e_{t},\qquad\qquad\qquad\left(t=T_{j-1}^{0}+1,\ldots,\,T_{j}^{0}\right)
\end{align*}
for $j=1,\ldots,\,m+1$, where by convention $T_{0}^{0}=0$ and $T_{m+1}^{0}=T$.
There are $m$ unknown break points $\left(T_{1}^{0},\ldots,\,T_{m}^{0}\right)$
and consequently $m+1$ regimes each corresponding to a distinct parameter
value $\delta_{j}^{0}$. The purpose is to estimate the unknown regression
coefficients together with the break points when $T$ observations
on $\left(y_{t},\,w_{t},\,z_{t}\right)$ are available. Many of the
theoretical results follow directly from the single break case; the
break points are asymptotically distinct and thus, given the mixing
conditions, our results for the single break date extend readily to
multiple breaks. More complicated is the computation of the estimates
of the break dates which has been addressed by \citet{bai/perron:03}
who proposed an efficient algorithm based on the principle of dynamic
programming; see also \citet{hawkins:76}.

Let $T_{i}\triangleq\left\lfloor T\lambda_{i}\right\rfloor $ and
$\theta\triangleq\left(\phi',\,\delta'_{1},\,\Delta_{1}^{\prime}\ldots,\,\Delta_{m}^{\prime}\right)'$
where $\Delta_{i}=\delta_{i+1}-\delta_{i}$, $i=1,\ldots,\,m$. The
class $\mathscr{\mathscr{L}}\left(\theta,\,T_{i};\,1\leq i\leq m\right)$
of GL estimators in multiple change-points models relies on the least-squares
criterion function $Q_{T}\left(\delta\left(\boldsymbol{\lambda}_{b}\right),\,\boldsymbol{\lambda}_{b}\right)=\sum_{i=1}^{m+1}\sum_{t=T_{i}-1}^{T_{i}}\left(y_{t}-w'_{t}\phi-z'_{t}\delta_{i}\right)^{2}$,
with $\boldsymbol{\lambda}_{b}\triangleq\left(\lambda_{i};\,1\leq i\leq m\right)$.
In order to state the large-sample properties, we need to introduce
the shrinkage theoretical framework of \citet{bai/perron:98}. 
\begin{assumption}
\label{Assumption A1-A6 BP98}(i) Let $x_{t}=\left(w'_{t},\,z'_{t}\right)'$,
$X=\left(x_{1},\ldots x_{T}\right)'$ and $\overline{X}_{0}=\textrm{diag}\left(X_{1}^{0},\ldots,\,X_{m+1}^{0}\right)$
be the diagonal partition of $X$ at $\left(T_{1}^{0},\ldots,T_{m}^{0}\right)$.
For each $i=1,\ldots,m+1$ $\left(X_{1}^{0}\right)'X_{1}^{0}/\left(T_{i}^{0}-T_{i-1}^{0}\right)$
converges to a non-random positive definite matrix not necessarily
the same for all $i$. (ii) Assumption \ref{Assumption A3 Bai97}
holds. (iii) The matrix $\sum_{t=k}^{l}z_{t}z'_{t}$ is invertible
for $l-k\geq q$. (iv) $T_{i}^{0}=\left\lfloor T\lambda_{i}^{0}\right\rfloor $,
where $0<\lambda_{1}^{0}<\cdots<\lambda_{m}^{0}<1$. (v) Let $\Delta_{T,i}=v_{T}\Delta_{i}^{0}$
where $v_{T}>0$ is a scalar satisfying $v_{T}\rightarrow0$ and $T^{1/2-\vartheta}v_{T}\rightarrow\infty$
for some $\vartheta\in\left(0,\,1/4\right)$, and $\mathbb{E}\left\Vert z_{t}\right\Vert ^{2}<C,\,$
$\mathbb{E}\left\Vert e_{t}\right\Vert ^{2/\vartheta}<C$ for some
$C<\infty$ and all $t$.
\end{assumption}

\begin{assumption}
\label{Ass A7 BP98}Let $\Delta T_{i}^{0}=T_{i}^{0}-T_{i-1}^{0}$.
For $i=1,\ldots,\,m+1$, uniformly in $s\in\left[0,\,1\right]$, (a)
$\left(\Delta T_{i}^{0}\right)^{-1}\sum_{t=T_{i-1}^{0}+1}^{T_{i-1}^{0}+\left\lfloor s\Delta T_{i}^{0}\right\rfloor }z_{t}z'_{t}\overset{\mathbb{P}}{\rightarrow}sV_{i}$,
$\left(\Delta T_{i}^{0}\right)^{-1}\sum_{t=T_{i-1}^{0}+1}^{T_{i-1}^{0}+\left\lfloor s\Delta T_{i}^{0}\right\rfloor }e_{t}^{2}\overset{\mathbb{P}}{\rightarrow}s\sigma_{i}^{2}$,
and 
\begin{align*}
\left(\Delta T_{i}^{0}\right)^{-1}\sum_{t=T_{i-1}^{0}+1}^{T_{i-1}^{0}+\left\lfloor s\Delta T_{i}^{0}\right\rfloor }\sum_{r=T_{i-1}^{0}+1}^{T_{i-1}^{0}+\left\lfloor s\Delta T_{i}^{0}\right\rfloor }\mathbb{E}\left(z_{t}z'_{r}u_{t}u_{r}\right) & \overset{\mathbb{P}}{\rightarrow}s\Sigma_{i};
\end{align*}

(b) $\left(\Delta T_{i}^{0}\right)^{-1/2}\sum_{t=T_{i-1}^{0}+1}^{T_{i-1}^{0}+\left\lfloor s\Delta T_{i}^{0}\right\rfloor }z_{t}u_{t}\overset{\mathbb{P}}{\rightarrow}\mathscr{G}_{i}\left(s\right)$
where $\mathscr{G}_{i}\left(s\right)$ is a multivariate Gaussian
process on $\left[0,\,1\right]$ with mean zero and covariance $\mathbb{E}\left[\mathscr{G}_{i}\left(s\right)\mathscr{G}_{i}\left(u\right)\right]=\min\left\{ s,\,u\right\} \Sigma_{i}$. 
\end{assumption}
Next, for $i=1,\ldots,\,m$, define $\Xi_{Z,i}=\left(\Delta_{i}^{0}\right)'V_{i+1}\Delta_{i}^{0}/\left(\Delta_{i}^{0}\right)'V_{i}\Delta_{i}^{0},$
$\Xi_{e,i}^{2}=\left(\Delta_{i}^{0}\right)'\Sigma_{i+1}\Delta_{i}^{0}/\left(\Delta_{i}^{0}\right)'\Sigma_{i}\Delta_{i}^{0}$,
and let $W_{1}^{\left(i\right)}\left(s\right)$ and $W_{2}^{\left(i\right)}\left(s\right)$
be independent Wiener processes defined on $[0,\,\infty)$, starting
at 0 when $s=0$; $W_{1}^{\left(i\right)}\left(s\right)$ and $W_{2}^{\left(i\right)}\left(s\right)$
are also independent over $i.$ Finally, define
\begin{align}
\mathscr{V}^{\left(i\right)}\left(s\right)\triangleq\mathscr{W}^{\left(i\right)}\left(s\right)-\varLambda_{i}^{0}\left(s\right) & \triangleq\begin{cases}
2\left(\left(\Delta_{i}^{0}\right)'\Sigma_{i}\Delta_{i}\right)^{1/2}W_{1}^{\left(i\right)}\left(-s\right)-\left|s\right|\left(\Delta_{i}^{0}\right)'V_{i}\Delta_{i}, & \textrm{if }s\leq0\\
2\left(\left(\Delta_{i}^{0}\right)'\Sigma_{i+1}\delta^{0}\right)^{1/2}W_{2}^{\left(i\right)}\left(s\right)-s\left(\Delta_{i}^{0}\right)'V_{i+1}\Delta_{i}, & \textrm{if }s>0.
\end{cases}\label{eq. V(s) Limit Process Theorem Post Mean BP98}
\end{align}
 We now extend the notation of Section \ref{Section Asymptotic Results LapBai97}
to the present context. By redefining the Quasi-posterior $p\left(\boldsymbol{\lambda}_{b}\right)$
in terms of $\boldsymbol{\lambda}_{b},$ the GL estimator as the minimizer
of the associated risk function {[}recall \eqref{Eq. Expected Risk function-1}{]},
$\widehat{\boldsymbol{\lambda}}_{b}^{\mathrm{GL}}=\arg\min_{s\in\varGamma^{0}}\left[\mathcal{R}_{l,T}\left(s\right)\right],$
where now $\varGamma^{0}=\mathbf{B}_{1}\times\ldots\times\mathbf{B}_{m}$,
with $\mathbf{B}_{i}$ a compact subset of $\left(0,\,1\right)$.
The sets $\mathbf{B}_{i}$ are disjoint and satisfy $\sup_{\lambda\in\mathbf{B}_{i}}<\inf_{\lambda\in\mathbf{B}_{i+1}}$
for all $i.$ 
\begin{assumption}
\label{Ass Gaussian BP98}Assumptions \ref{Assumption The-loss-function LapBai97}-\ref{Assumption Prior LapBai97}
hold with obvious modifications to allow for the multidimensional
parameter $\boldsymbol{\lambda}_{b}\in\varGamma^{0}.$ Assumption
\ref{Assumption Gaussian Process for Lap LapBai97} holds where now
in part (i) $\boldsymbol{\lambda}_{b}$ replaces $\lambda_{b}$, and
in part (ii) $\varSigma^{\left(i\right)}\left(\cdot,\,\cdot\right)$
($1\leq i\leq m+1$) replaces $\varSigma\left(\cdot,\,\cdot\right)$
and is defined analogously for each regime.
\end{assumption}
Assumption \ref{Assumption Uniquess loss Fucntion LapBai97} implies
that $\xi_{l,i}^{0}\triangleq\xi\left(\lambda_{i}^{0}\right)$ is
uniquely defined by $\Psi_{l}\left(\xi_{l,i}^{0}\right)\triangleq\inf_{s}\Psi_{l,i}\left(s\right)=\inf_{s}\int_{\mathbb{R}}l\left(s-u\right)\left(\exp\left(\mathscr{V}^{\left(i\right)}\left(u\right)\right)/\left(\int_{\mathbb{R}}\exp\left(\mathscr{V}^{\left(i\right)}\left(w\right)\right)dw\right)\right)du$.
 The GL estimator is defined as the minimizer of
\begin{align*}
\mathcal{R}_{l,T} & \triangleq\int_{\varGamma^{0}}l\left(s-\boldsymbol{\lambda}_{b}\right)\frac{\exp\left(-Q_{T}\left(\delta\left(\boldsymbol{\lambda}_{b}\right),\,\boldsymbol{\lambda}_{b}\right)\right)\pi\left(\boldsymbol{\lambda}_{b}\right)}{\int_{\varGamma^{0}}\exp\left(-Q_{T}\left(\delta\left(\boldsymbol{\lambda}_{b}\right),\,\boldsymbol{\lambda}_{b}\right)\right)\pi\left(\boldsymbol{\lambda}_{b}\right)d\boldsymbol{\lambda}_{b}}d\boldsymbol{\lambda}_{b}.
\end{align*}
 The analysis is now in terms of the $m\times1$ local parameter $u$
with components $u_{i}=T\left\Vert \Delta_{T,i}\right\Vert ^{2}(\lambda_{i}-$
$\lambda_{i,T}^{0}\left(v\right))$, with $\lambda_{i,T}^{0}\left(v\right)=\lambda_{i,T}^{0}\left(\theta^{0}+v/r_{T}\right)$. 

Theorem \ref{Theorem Posterior Mean Lap Estimation in BP98}-\ref{Theorem Geneal Laplace Estimator LapBP98}
extend corresponding results from Theorem \ref{Theorem Posterior Mean Lap Estimation in Bai 97}-\ref{Theorem Geneal Laplace Estimator LapBai97},
respectively, to multiple change-points. The fast rate of convergence
implies that asymptotically the behavior of the GL estimator only
matters in a small neighborhood of each $T_{i}^{0}$. Since each such
neighborhood increases at rate $1/v_{T}$ while $T\rightarrow\infty$
at a faster rate, given the mixing conditions, these are asymptotically
distinct and the limiting distribution is then similar to that in
the single break case. This is the same argument underlying the analysis
of \citet{bai/perron:98} and of \citet{ibragimov/has:81}. The same
comments as those in Section \ref{Section Asymptotic Results LapBai97}
apply.
\begin{condition}
\label{Condition 1 LapBP98}For $1\leq i\leq m$ there exist positive
finite numbers $\kappa_{\gamma,i}$ such that $\gamma_{T}/T\left\Vert \Delta_{T,i}\right\Vert ^{2}\rightarrow\kappa_{\gamma,i}$.
\end{condition}
\begin{thm}
\label{Theorem Posterior Mean Lap Estimation in BP98}Assume $l\left(\cdot\right)$
is the squared loss function. Under Assumption \ref{Assumption A1-A6 BP98}-\ref{Ass Gaussian BP98}
and Condition \ref{Condition 1 LapBP98}, we have in $\mathbb{D}_{b}$,
\begin{align}
T\left\Vert \Delta_{T,i}\right\Vert ^{2}\left(\widehat{\lambda}_{i}^{\mathrm{GL}}-\lambda_{i}^{0}\right) & \Rightarrow\frac{\int u\exp\left(\mathscr{W}^{\left(i\right)}\left(u\right)-\varLambda_{i}^{0}\left(u\right)\right)du}{\int\exp\left(\mathscr{W}^{\left(i\right)}\left(u\right)-\varLambda_{i}^{0}\left(u\right)\right)du}.\label{eq. Asy Dist Laplace Bai97 part (ii)-1}
\end{align}
\end{thm}
Turning to the general case of loss functions satisfying Assumption
\ref{Assumption The-loss-function LapBai97}, Theorem \ref{Theorem Geneal Laplace Estimator LapBP98}
shows that the random quantity $T\left\Vert \delta_{T}\right\Vert ^{2}\left(\widehat{\lambda}_{i}^{\mathrm{GL}}-\lambda_{i}^{0}\right)$
is (first-order) asymptotically equivalent to the random variable
$\xi_{l,i}^{0}$ determined by 
\begin{align}
\Psi_{l}\left(\xi_{l,i}^{0}\right) & \triangleq\inf_{r}\Psi_{l,i}\left(r\right)=\inf_{r\in\mathbb{R}}\left\{ \int_{\mathbb{R}}l\left(r-u\right)\frac{\exp\left(\mathscr{W}^{\left(i\right)}\left(u\right)-\varLambda_{i}^{0}\left(u\right)\right)}{\int\exp\left(\mathscr{W}^{\left(i\right)}\left(u\right)-\varLambda_{i}^{0}\left(u\right)\right)du}du\right\} .\label{Eq. Distribution of Bayes Estimator-1-1}
\end{align}

\begin{thm}
\label{Theorem Geneal Laplace Estimator LapBP98}Under Assumptions
\ref{Assumption A1-A6 BP98}-\ref{Ass Gaussian BP98} and Condition
\ref{Condition 1 LapBP98}, for $l\in\boldsymbol{L},$ $T\left\Vert \Delta_{T,i}\right\Vert ^{2}\left(\widehat{\lambda}_{i}^{\mathrm{GL}}-\lambda_{i}^{0}\right)\Rightarrow\xi_{l,i}^{0},$
as defined by \eqref{Eq. Distribution of Bayes Estimator-1-1}.
\end{thm}
A direct consequence of the results of this section is that statistical
inference for the break dates $T_{i}^{0}$ $\left(i=1,\ldots,\,m\right)$
can be carried out using the same methods for the single break case.

\section{\label{Section Theoretical-Properties-of GL Inference}Theoretical
Properties of GL Inference}

This section shows that the GL-HPDR confidence sets are bet-proof.
The betting framework and the notion of bet-proofness are useful to
study the properties of frequentist inference in non-regular problems.
The literature concluded that frequentist confidence sets may exhibit
undesirable properties in non-regular problems {[}e.g., \citet{buehler:59},
\citet{cornfield:69}, \citet{cox:58}, \citet{mueller/norest:16}
\citet{pierce:73}, \citet{robinson:77} and \citet{wallace:59}{]}.
For example, the confidence sets can be too short or empty with positive
probability. This arises because frequentist procedures often have
the property that, conditional on a sample point lying in some subset
of the sample space, the conditional confidence level  is less than
the unconditional confidence level uniformly in the parameters.

We use the same betting framework as in \citet{buehler:59}. Let $\mathbb{P}\left(\cdot|\,\lambda_{b}\right)$
denote the likelihood of the data $Y\in\mathcal{Y}$ conditional on
$\lambda_{b}\in\varGamma^{0}$. Assume $\mathbb{P}\left(\cdot|\,\lambda_{b}\right)$
has density $p\left(\cdot|\,\lambda_{b}\right)$ with respect to a
finite measure $\zeta.$ We define a $1-\alpha$ confidence set by
a rejection probability rule $\varphi:\,\varGamma^{0}\times\mathcal{Y}\mapsto\left[0,\,1\right]$
satisfying $\int\left[1-\varphi\left(\lambda_{b},\,y\right)\right]p\left(y|\,\lambda_{b}\right)d\zeta\left(y\right)\geq1-\alpha$,
with $\varphi\left(\lambda_{b},\,y\right)$ the probability that $\lambda_{b}$
is not included in the set when $y$ is observed. For any realization
of the data $Y=y$, an inspector can choose to object to the confidence
set $\varphi$. The inspector\textquoteright s objection $\widetilde{b}:\mathcal{Y}\mapsto\left[0,\,1\right]$
takes value 1 if there is an objection. Denote by $\mathbf{B}$ the
set of all measurable strategies $\widetilde{b}$. When $\widetilde{b}=1$
the inspector receives 1 if $\varphi$ does not contain $\lambda_{b}$,
and she loses $\alpha/\left(1-\alpha\right)$ otherwise. For a given
parameter $\lambda_{b}$ and betting strategy $\widetilde{b}$, the
inspector's expected loss is, 
\begin{align*}
L_{\alpha}\left(\varphi,\,\widetilde{b},\,\lambda_{b}\right) & =\frac{1}{1-\alpha}\int\left[\alpha-\varphi\left(\lambda_{b},\,y\right)\right]\widetilde{b}\left(y\right)p\left(y|\lambda_{b}\right)d\zeta\left(y\right).
\end{align*}
 A confidence set $\varphi$ is said to be bet-proof at level $1-\alpha$
if for each $\widetilde{b}\in\mathbf{B}$, $L_{\alpha}\left(\varphi,\,\widetilde{b},\,\lambda_{b}\right)\geq0$
for some $\lambda_{b}\in\varGamma^{0}$. If there exists a strategy
$\widetilde{b}$ such that $L_{\alpha}\left(\varphi,\,\widetilde{b},\,\lambda_{b}\right)<0$
for all $\lambda_{b}\in\varGamma^{0}$, then the inspector would be
right on average and would make positive expected profits. Hence,
such $\varphi$ would be an ``unreasonable'' confidence set. 
Without loss of substance, we restrict our attention to a change in
the mean of a sequence of i.i.d. Gaussian variables. Let $y_{t}=\delta_{T}\mathbf{1}\left\{ t>T_{b}^{0}\right\} +e_{t},$
where $e_{t}\sim i.i.d.\,\mathrm{\mathscr{N}\left(0,\,1\right)}.$
The result below can also be shown to hold for fixed shifts $\delta_{T}=\delta^{0}$.
For ease of exposition, we assume $\delta^{0}$ known. The general
case leads to similar results, with more lengthy derivations without
any gain in intuition.

Recall that $\varphi$ is such that the Quasi-posterior probability
$p_{T}\left(\lambda_{b}|\,y\right)=p_{T}\left(\lambda_{b}\right)$
of excluding $\lambda_{b}$ is less than or equal to $\alpha,$ 
\begin{align}
\int\varphi\left(\lambda_{b},\,y\right)p_{T}\left(\lambda_{b}|\,y\right)d\lambda_{b} & \leq\alpha\quad\mathrm{for\,all\,}y\in\mathcal{Y}.\label{eq (def Quasi-credible set)}
\end{align}

\begin{prop}
\label{Proposition Bet Proof}Under Assumptions \ref{Assumption A Bai97}-\ref{Assumption A4 Bai97}
and \ref{Assumption The-loss-function LapBai97}-\ref{Assumption Uniquess loss Fucntion LapBai97},
and Condition \ref{Condition 1 LapBai97}, for $l\in\boldsymbol{L}:$
(i) $\varphi$ is bet-proof at level $1-\alpha$; (ii) If \eqref{eq (def Quasi-credible set)}
holds with equality, then $\varphi$ is the shortest confidence set
in the class of level $1-\alpha$ confidence sets, i.e., there cannot
exist a level $1-\alpha$ confidence set $\varphi'$ with the property
that, for all $y\in\mathcal{Y}$ $\int\varphi'\left(\lambda_{b},\,y\right)d\lambda_{b}\geq\int\varphi\left(\lambda_{b},\,y\right)d\lambda_{b}$,
and for all $y\in\mathcal{Y}_{0}$ with $\zeta\left(\mathcal{Y}_{0}\right)>0,$
$\int\varphi'\left(\lambda_{b},\,y\right)d\lambda_{b}>\int\varphi\left(\lambda_{b},\,y\right)d\lambda_{b}$.
\end{prop}
Part of the proof shows that the Quasi-posterior is asymptotically
equivalent (in total variation distance) to the Bayesian posterior.
Given the conservativeness allowed by Definition \ref{def:Highest-Density-Region},
the GL confidence interval is asymptotically a superset of a Bayesian
credible interval. Bet-proofness is a useful criterion in change-point
models where popular inference methods face some difficulties, as
shown in the next section. Proposition \ref{Proposition Bet Proof}
suggests that GL inference should not suffer from these issues; the
simulations in the next section will confirm that this is indeed the
case.

\section{\label{Section Monte-Carlo-Simulation}Finite-Sample Evaluations}

The purpose of this section is twofold. Section \ref{subsec:Precision-of-the}
assesses the accuracy of the GL estimate of the change-point while
Section \ref{subsec:Properties-of-the} evaluates the small-sample
properties of the proposed method to construct confidence sets. We
consider DGPs that take the form:
\begin{align}
y_{t}=D{}_{t}\alpha^{0}+Z{}_{t}\beta^{0}+Z{}_{t}\delta^{0}\boldsymbol{1}\left\{ t>T_{b}^{0}\right\} +e_{t}, & \qquad\qquad t=1,\ldots,\,T,\label{Eq. DGP Simulation Study}
\end{align}
with a sample size $T=100.$ Three versions of \eqref{Eq. DGP Simulation Study}
are investigated: M1 involves a break in mean: $Z_{t}=1$, $D_{t}$
absent, and $e_{t}\sim i.i.d.\,\mathscr{N}\left(0,\,1\right)$; M2
is similar to M1 but with $e_{t}=0.3e_{t-1}+u_{t}$, $u_{t}\sim\mathscr{N}\left(0,\,1\right)$;
M3 is a dynamic model with $D_{t}=y_{t-1}$, $Z_{t}=1$, $e_{t}\sim i.i.d.\,\mathscr{N}\left(0,\,0.5\right)$
and $\alpha^{0}=0.6$. We set $\beta^{0}=1$ in M1-M2 and $\beta^{0}=0$
in M3. We consider $\lambda_{0}=0.3$ and 0.5, and break magnitudes
$\delta^{0}=0.3,\,0.4,\,0.6$ and $1$. Additional simulations are
presented in the supplement. 

\subsection{\label{subsec:Precision-of-the}Precision of the Change-point Estimate}

We consider the following estimators of $T_{b}^{0}$: the least-squares
estimator (OLS), the GL estimator under a least-absolute loss function
 (GL-LN); the GL estimator under a least-absolute loss function
with a uniform prior (GL-Uni). We compare the mean absolute error
(MAE), standard deviation (Std), root-mean-squared error (RMSE), and
the 25\% and 75\% quantiles. We set the trimming parameter $\epsilon$
equal to 0.05. As explained in \citet{casini/perron_Lap_CR_Single_Inf},
the trimming $\epsilon$ should not be chosen too high because otherwise
the estimate might tend to overestimate (resp. underestimate) the
break date if it is in the first (resp. second) half of the sample.
They found that $\epsilon=0.05$ performs well for different locations
of the break date and this is also confirmed in the simulations in
this section. See Section 1 and 5 in \citet{casini/perron_Lap_CR_Single_Inf}
for more discussion. 

Tables \ref{Table M0 Bias}-\ref{Table M5 Bias} present the results.
When the magnitude of the break is small, the OLS estimator displays
quite large MAE, which increases as the change-point point moves
toward the tails. In contrast, the GL estimator shows substantially
lower MAE uniformly over break magnitudes and break locations. 
In addition, the GL estimator has smaller variance as well as lower
RMSE compared to the OLS estimator. Notably, the distribution of
GL-LN concentrates a higher fraction of the mass around the mid-sample
relative to the finite-sample distribution of the OLS estimate. This
is mainly due to the fact that the Quasi-posterior essentially does
not share the marked trimodality of the finite-sample distribution
{[}cf. \citet{casini/perron_CR_Single_Break}{]}. When the break magnitude
is small, the objective function is quite flat with a small peak at
the OLS estimate. The Quasi-posterior has higher mass close to the
OLS estimate\textemdash which corresponds to the middle mode\textemdash and
accordingly lower mass in the tails. The GL estimator that uses the
uniform prior (GL-Uni) is also more precise than the OLS estimator,
though the margin is smaller. The latter is due to the fact that
the GL estimate uses information only from the OLS objective function.
We have not reported the bias. However, here is a summary of its behavior
which can also be learned from Figures \ref{Fig1}-\ref{Fig2}. When
$\lambda_{b}^{0}=0.5$, the bias is small and close to zero because
the finite-sample distributions of the estimators are symmetric. When
$\lambda_{b}^{0}<0.5$, the bias is positive which means that the
break date estimators tend to be on the right of $\lambda_{b}^{0}.$
The opposite hold for $\lambda_{b}^{0}>0.5$.\textbf{ }

\subsection{\label{subsec:Properties-of-the}Properties of the GL Confidence
Sets}

We now assess the performance of the suggested inference procedure
for the break date. We compare it with the following existing methods:
Bai's (1997) approach, \citeauthor{elliott/mueller:07}\textquoteright s
(2007) approach based on\textcolor{red}{{} }inverting a sequence of
locally best invariant tests using \citeauthor{nyblom:89}'s (1989)
statistic, the inverted likelihood-ratio (ILR) method of \citet{eo/morley:15}
which inverts the likelihood-ratio test of \citet{qu/perron:07} and
the HDR method proposed in \citet{casini/perron_CR_Single_Break}
based on continuous record asymptotics, labelled OLS-CR. These methods
have been discussed in detail in \citet{casini/perron_CR_Single_Break}
and in \citet{chang/perron:18}. We can summarize their properties
as follows. The confidence intervals obtained from Bai's (1997) method
display empirical coverage rates often below the nominal level when
the size of the break is small. In general, \citeauthor{elliott/mueller:07}\textquoteright s
(2007) approach achieves the most accurate coverage rates but the
average length of the confidence sets is always substantially larger
relative to other methods.\footnote{This problem is more severe when the errors are serially correlated
or the model includes lagged dependent variables.\nocite{casini_hac}
Regarding the former, this in part may be due to issues with \citeauthor{newey/west:87}
HAC-type estimators when there are breaks {[}see \citeauthor{casini_CR_Test_Inst_Forecast}
(2018, 2019), \nocite{casini_hac} \citet{casini/perron_Low_Frequency_Contam_Nonstat:2020},
\citet{casini/perron_Oxford_Survey}, \citet{chang/perron:18}, \citet{crainiceanu/vogelsang:07},
\citet{deng/perron:06}, \citet{fossati:17}, \citet{juhl/xiao:09},
\citet{kim/perron:09}, \citet{martins/perron:16}, \citet{perron/yamamoto:18}
and \citet{vogeslang:99}{]}.} In addition, this approach breaks down in models with serially correlated
errors or lagged dependent variables, whereby the length of the confidence
set approaches the whole sample as the magnitude of the break increases.
The ILR has coverage rates often above the nominal level and an average
length significantly longer than with the OLS-CR method when the magnitude
of the shift is small. Here, we shall show that the GL inference performs
well in terms of coverage probability compared with the other methods
and is characterized by shorter lengths of the confidence sets. 

When the errors are uncorrelated (i.e., M1 and M3) we simply estimate
variances rather than long-run variances. The least-squares estimation
method is employed with a trimming parameter $\epsilon=0.15$ and
we use the required degrees of freedom adjustment for the statistic
$\widehat{\textrm{U}}_{T}$ of \citet{elliott/mueller:07}. To construct
the OLS-CR method, we follow the steps outlined in \citet{casini/perron_CR_Single_Break}.
To implement Bai's (1997) method we use the usual steps described
in \citet{bai:97RES} and \citet{bai/perron:98}. We implement the
GL estimator using a least-absolute loss with the prior from Corollary
\ref{Corollary part (i) of Theorem LapBai97}. For model M2, the
estimate of the long-run variance is the pre-whitened heteroskedasticity
and autocorrelation (HAC) estimator of \citet{andrews/monahan:92}.
We consider the version $\widehat{\textrm{U}}_{T}$ proposed by \citet{elliott/mueller:07}
that allows for heterogeneity across regimes; using the restricted
version when applicable leads to similar results. Finally, the last
row of each panel includes the rejection probability of the 5\%-level
sup-Wald test using the asymptotic critical value of \citet{andrews:93};
it serves as a statistical measure of the magnitude of the break.

Overall, the results in Table \ref{Table M0}-\ref{Table M2} confirm
previous findings about the performance of existing methods. Bai's
(1997) method has a coverage rate below the nominal level when the
size of the break is small. For example, in model M2, with $\lambda_{0}=0.5$
and $\delta^{0}=0.8$, it has a coverage probability below 82\% even
though the Sup-Wald test rejects roughly for 70\% of the samples.
With smaller break sizes, it systematically fails to cover the true
break date with correct probability. In contrast, the method of \citet{elliott/mueller:07}
yields very accurate empirical coverage rates. However, the average
length of the confidence intervals obtained is systematically much
larger than those from all other methods across all DGPs, break sizes
and break locations. For large break sizes, Bai's (1997) method delivers
good coverage rates and the shortest average length among all methods. 

The GL method displays good coverage rates across different break
magnitudes and tends to have the shortest lengths among all methods
for all break magnitudes, except for $\delta^{0}=1.6$ in model M2
for which Bai's (1997) confidence interval is slightly shorter.
In Model M3, the coverage rates of OLS-CR are more accurate than those
with the GL method although the difference is not large. Thus the
GL method strikes a good balance between adequate coverage probability
and short average lengths, thus confirming the theoretical results
on bet-proofness. This is also consistent with Figures \ref{Fig1}-\ref{Fig2}
which show that the asymptotic distribution of the GL estimator does
not underestimate uncertainty about the break location even when the
break magnitude is small thereby yielding good coverage rates also
in this case. In model M1 the GL method leads to shorter lengths than
Bai's even for large breaks. This is not in contradiction with Figure
\ref{Fig2} because model M1 is a simpler model than that reported
in the figure which shows that the density of the asymptotic distribution
of the GL estimator is more spread out than that from \citet{bai:97RES}. 

Non-reported simulations show that the GL method is robust to heteroskedastic
errors $e_{t}=\left|z_{t}\right|u_{t}$ and non-normal errors. The
case of multiple breaks is not considered since they are expected
to be similar as in the single break case by virtue of the assumption
that the break dates are sufficiently separated. Finally, in the supplement
we compare the GL method above with its continuous record counterparts
developed in \citet{casini/perron_Lap_CR_Single_Inf}. Overall, we
find that both estimation and confidence intervals based on GL-LN
perform well relative to the continuous record counterparts, where
significant gains appear to occur when there is high serial correlation
in the errors. See the additional results reported in the supplement. 

\section{\label{Section Conclusions}Conclusions}

We developed large-sample results for a class of Generalized Laplace
estimators in multiple change-points models where popular methods
face some challenges due to the non-regularities of the problem. The
GL method exploits the insight of Laplace who proposed to generate
a density from taking an exponential transformation of a least-squares
criterion. The class of GL estimators exhibits a dual limiting distribution;
namely, the classical shrinkage asymptotic distribution of \citet{bai/perron:98},
or a Bayes-type asymptotic distribution {[}cf. \citet{ibragimov/has:81}{]}.
Simulations show that the GL estimator is more accurate than OLS.
Similarly, inference has superior finite-sample properties relative
to popular methods and these properties are shown to be supported
by theoretical results. Since the issues about the finite-sample performance
of OLS especially for small breaks continue to hold in more complex
structural change models, we believe that our method can be usefully
extended to those models. For example, the GL approach can be immediately
applied to nonlinear models (e.g., instrumental variable models, linear
model with restrictions, nonlinear regression models, etc.) even though
particular attention to the appropriate choice of the prior should
be given in each context. We believe that our approach can also be
relevant for high-dimensional regression with structural changes although
this would require a careful consideration of additional aspects related
to the growing number of regressors. 

\newpage{}

\section{Supplementary Material}

Casini, A. and P. Perron (2020c). Supplement to \textquotedblleft Generalized
Laplace Inference in Multiple Change-points Models\textquotedblright ,
Econometric Theory Supplementary Material. To view, please visit:
{[}{[}doi will be inserted here by typesetter{]}{]}\nocite{casini/perron_CR_Single_Break,casini/perron_Lap_CR_Single_Inf,casini/perron_Low_Frequency_Contam_Nonstat:2020,casini/perron_PrewhitedHAC,casini/perron_SC_BP_Lap,casini_diss,casini_hac}

\begin{singlespace}

\bibliographystyle{chicago}
\bibliography{References_JoE}

\addcontentsline{toc}{section}{References}

\end{singlespace}

\newpage{}

\setcounter{page}{1} \renewcommand{\thepage}{F-\arabic{page}}

\begin{singlespace}

\begin{center}
\begin{figure}[h]
\includegraphics[width=18cm,height=6.5cm]{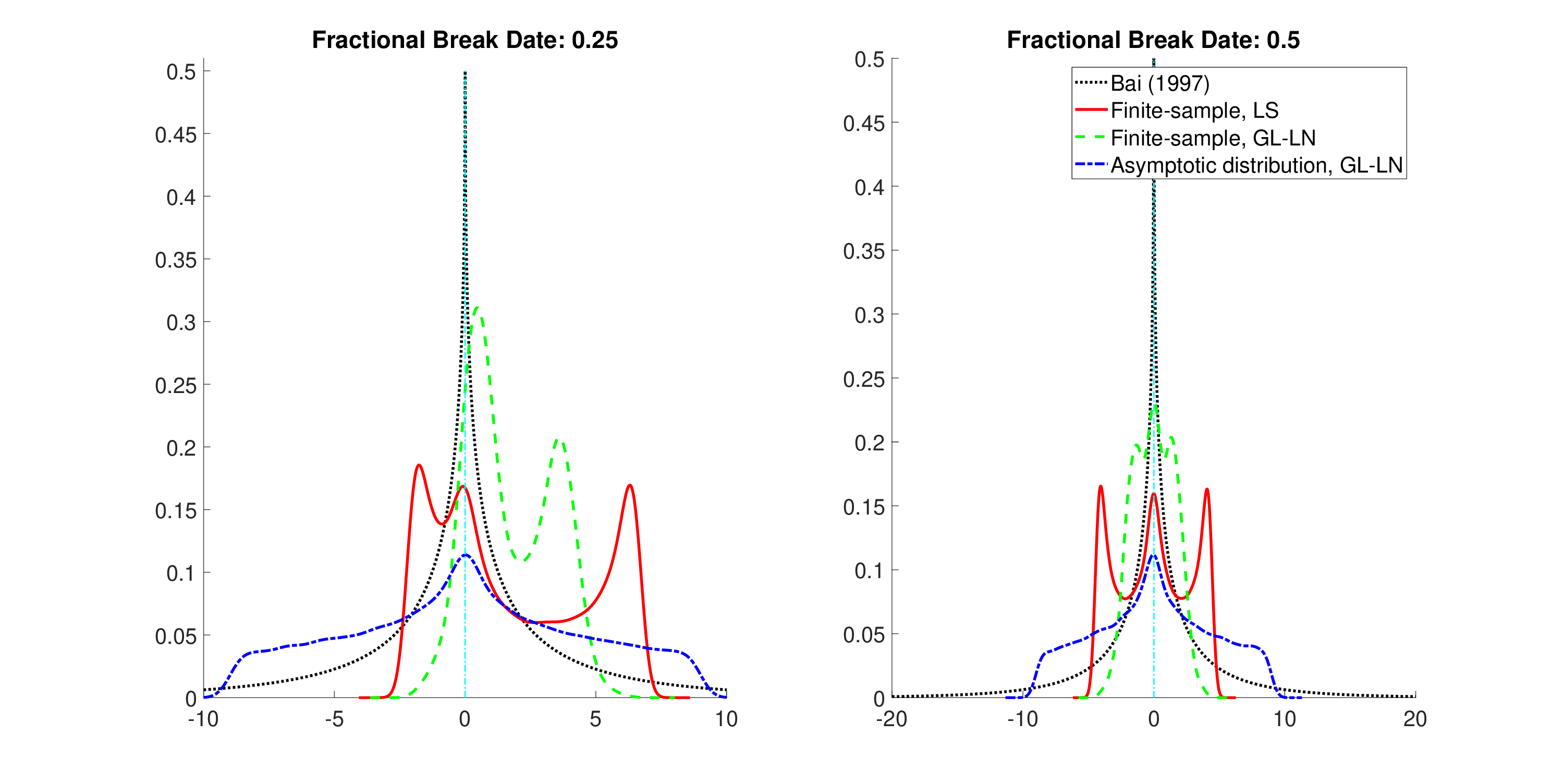}

{\footnotesize{}\caption{{\scriptsize{}\label{Fig1}The probability density of the LS estimator
for the model $y_{t}=\mu^{0}+Z_{t}\delta_{1}^{0}+Z_{t}\delta_{2}^{0}\mathbf{1}\left\{ t>\left\lfloor T\lambda_{0}\right\rfloor \right\} +e_{t},\,Z_{t}=0.3Z_{t-1}+u_{t}-0.1u_{t-1},\,u_{t}\sim\textrm{i.i.d.}\mathscr{N}\left(0,\,1\right),\,e_{t}\sim\textrm{i.i.d.}\mathscr{N}\left(0,\,1\right),$
$\left\{ u_{t}\right\} $ independent from $\left\{ e_{t}\right\} ,$
$T=100$ with $\delta^{0}=0.3$ and $\lambda_{0}=0.25$ and $0.5$
(the left and right panel, respectively). The black dotted line is
the density of the asymptotic distribution from Bai (1997), the red
broken line break is the density of the finite-sample distribution
of the LS estimator, the green broken line is the density of the finite-sample
distribution of the GL estimator, and the blue broken line is the
density of the asymptotic distribution of the GL estimator. }}
}{\footnotesize\par}
\end{figure}
\end{center}

\setlength{\belowcaptionskip}{-0pt}

\begin{center}
\begin{figure}[h]
\includegraphics[width=18cm,height=6.5cm]{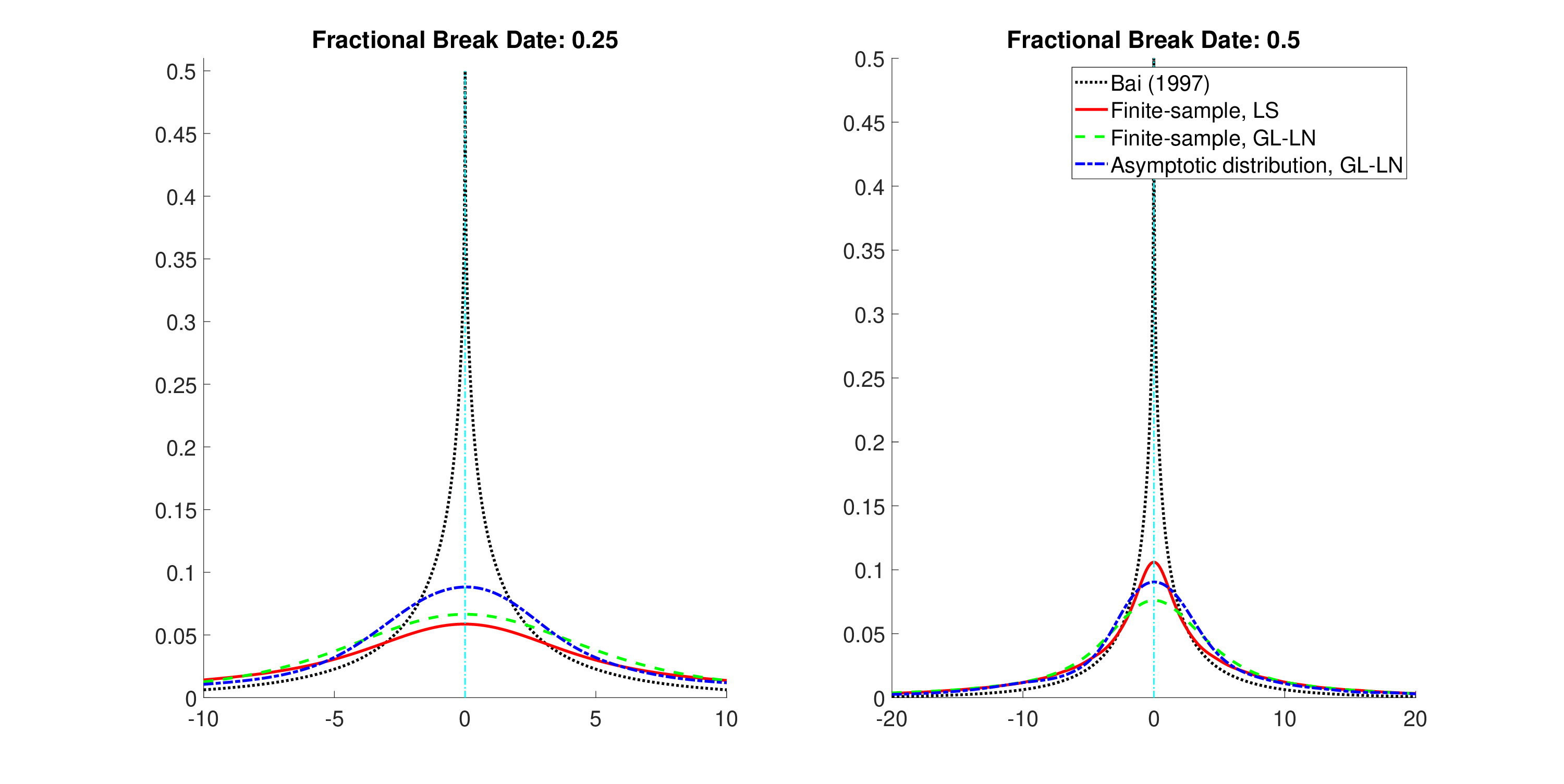}

{\footnotesize{}\caption{{\scriptsize{}\label{Fig2}The descriptions and comments given in
Figure \ref{Fig1} apply but with a break magnitude $\delta^{0}=1.5$. }}
}{\footnotesize\par}
\end{figure}
\end{center}

\end{singlespace}

\newpage{}

\setcounter{page}{1} \renewcommand{\thepage}{T-\arabic{page}}
\begin{center}
\begin{table}[H]
\caption{\label{Table M0 Bias}Small-sample accuracy of the estimate of the
break point $T_{b}^{0}$ for model M1}

\begin{singlespace}
\begin{centering}
{\footnotesize{}}%
\begin{tabular}{ccccccc|ccccc}
\hline 
 &  & {\footnotesize{}MAE} & {\footnotesize{}Std} & {\footnotesize{}$\textrm{RMSE}$} & {\footnotesize{}$Q_{0.25}$} & \multicolumn{1}{c}{{\footnotesize{}$Q_{0.75}$}} & {\footnotesize{}MAE} & {\footnotesize{}Std} & {\footnotesize{}$\textrm{RMSE}$} & {\footnotesize{}$Q_{0.25}$} & {\footnotesize{}$Q_{0.75}$}\tabularnewline
\cline{3-12} \cline{4-12} \cline{5-12} \cline{6-12} \cline{7-12} \cline{8-12} \cline{9-12} \cline{10-12} \cline{11-12} \cline{12-12} 
 &  & \multicolumn{5}{c|}{{\footnotesize{}$\lambda_{0}=0.3$}} & \multicolumn{5}{c}{{\footnotesize{}$\lambda_{0}=0.5$}}\tabularnewline
\cline{3-12} \cline{4-12} \cline{5-12} \cline{6-12} \cline{7-12} \cline{8-12} \cline{9-12} \cline{10-12} \cline{11-12} \cline{12-12} 
{\footnotesize{}$\delta^{0}=0.3$} & {\footnotesize{}OLS} & 21.99 & 27.51 & 30.53 & 24 & 66 & 21.51 & 26.85 & 26.79 & 34 & 71\tabularnewline
 & {\footnotesize{}GL-LN} & 13.44 & 15.03 & 18.99 & 28 & 54 & 11.85 & 14.51 & 14.93 & 38 & 60\tabularnewline
 & {\footnotesize{}GL-Uni} & 17.56 & 22.88 & 25.51 & 26 & 56 & 16.90 & 22.03 & 22.13 & 38 & 61\tabularnewline
\cline{3-12} \cline{4-12} \cline{5-12} \cline{6-12} \cline{7-12} \cline{8-12} \cline{9-12} \cline{10-12} \cline{11-12} \cline{12-12} 
{\footnotesize{}$\delta^{0}=0.4$} & {\footnotesize{}OLS} & 20.48 & 26.30 & 28.51 & 23 & 57 & 15.64 & 21.79 & 21.23 & 40 & 61\tabularnewline
 & {\footnotesize{}GL-LN} & 13.02 & 15.52 & 18.29 & 29 & 51 & 9.46 & 11.84 & 12.30 & 44 & 56\tabularnewline
 & {\footnotesize{}GL-Uni} & 17.68 & 22.30 & 24.64 & 27 & 54 & 12.38 & 17.69 & 17.15 & 42 & 57\tabularnewline
\cline{3-12} \cline{4-12} \cline{5-12} \cline{6-12} \cline{7-12} \cline{8-12} \cline{9-12} \cline{10-12} \cline{11-12} \cline{12-12} 
{\footnotesize{}$\delta^{0}=0.6$} & {\footnotesize{}OLS} & 13.04 & 20.82 & 15.92 & 28 & 41 & 11.06 & 16.05 & 16.89 & 45 & 55\tabularnewline
 & {\footnotesize{}GL-LN} & 9.20 & 13.67 & 13.67 & 28 & 40 & 7.04 & 9.92 & 10.46 & 47 & 53\tabularnewline
 & {\footnotesize{}GL-Uni} & 11.49 & 18.59 & 14.23 & 27 & 39 & 9.11 & 13.92 & 13.48 & 45 & 55\tabularnewline
\cline{3-12} \cline{4-12} \cline{5-12} \cline{6-12} \cline{7-12} \cline{8-12} \cline{9-12} \cline{10-12} \cline{11-12} \cline{12-12} 
{\footnotesize{}$\delta^{0}=1$} & {\footnotesize{}OLS} & 3.49 & 4.61 & 4.61 & 28 & 32 & 2.92 & 5.24 & 5.23 & 48 & 52\tabularnewline
 & {\footnotesize{}GL-LN} & 3.41 & 4.53 & 4.52 & 28 & 32 & 2.89 & 5.44 & 5.20 & 49 & 51\tabularnewline
 & {\footnotesize{}GL-Uni} & 3.63 & 4.56 & 4.61 & 28 & 32 & 2.90 & 5.21 & 5.22 & 48 & 52\tabularnewline
\hline 
\end{tabular}{\footnotesize\par}
\par\end{centering}
\end{singlespace}
\noindent\begin{minipage}[t]{1\columnwidth}%
{\scriptsize{}The model is $y_{t}=\delta^{0}\mathbf{1}\left\{ t>\left\lfloor T\lambda_{0}\right\rfloor \right\} +e_{t},\,e_{t}\sim i.i.d.\,\mathscr{N}\left(0,\,1\right),\,T=100$.
The columns refer to Mean Absolute Error (MAE), standard deviation
(Std), Root Mean Squared Error (RMSE) and the 25\% and 75\% empirical
quantiles. OLS is the least-squares estimator; GL-LN is the GL estimator
under a least-absolute loss function with the density of the long-span
asymptotic distribution as the prior; GL-Uni is the GL estimator under
a least-absolute loss function with a uniform prior. The number of
simulations is 3,000.}%
\end{minipage}
\end{table}
\par\end{center}

\begin{center}

\par\end{center}\begin{center}
\begin{table}[H]
\caption{\label{Table M1 Bias}Small-sample accuracy of the estimates of the
break point $T_{b}^{0}$ for model M2}

\begin{singlespace}
\begin{centering}
{\footnotesize{}}%
\begin{tabular}{ccccccc|ccccc}
\hline 
 &  & {\footnotesize{}MAE} & {\footnotesize{}Std} & {\footnotesize{}$\textrm{RMSE}$} & {\footnotesize{}$Q_{0.25}$} & \multicolumn{1}{c}{{\footnotesize{}$Q_{0.75}$}} & {\footnotesize{}MAE} & {\footnotesize{}Std} & {\footnotesize{}$\textrm{RMSE}$} & {\footnotesize{}$Q_{0.25}$} & {\footnotesize{}$Q_{0.75}$}\tabularnewline
\cline{3-12} \cline{4-12} \cline{5-12} \cline{6-12} \cline{7-12} \cline{8-12} \cline{9-12} \cline{10-12} \cline{11-12} \cline{12-12} 
 &  & \multicolumn{5}{c|}{{\footnotesize{}$\lambda_{0}=0.3$}} & \multicolumn{5}{c}{{\footnotesize{}$\lambda_{0}=0.5$}}\tabularnewline
\cline{3-12} \cline{4-12} \cline{5-12} \cline{6-12} \cline{7-12} \cline{8-12} \cline{9-12} \cline{10-12} \cline{11-12} \cline{12-12} 
{\footnotesize{}$\delta^{0}=0.3$} & {\footnotesize{}OLS} & 26.61 & 22.85 & 33.03 & 23 & 76 & 24.09 & 28.29 & 28.08 & 23 & 73\tabularnewline
 & {\footnotesize{}GL-LN} & 19.33 & 10.17 & 24.87 & 29 & 61 & 16.01 & 18.78 & 19.81 & 29 & 62\tabularnewline
 & {\footnotesize{}GL-Uni} & 24.76 & 21.05 & 31.34 & 26 & 70 & 20.93 & 25.37 & 25.39 & 28 & 65\tabularnewline
\cline{3-12} \cline{4-12} \cline{5-12} \cline{6-12} \cline{7-12} \cline{8-12} \cline{9-12} \cline{10-12} \cline{11-12} \cline{12-12} 
{\footnotesize{}$\delta^{0}=0.4$} & {\footnotesize{}OLS} & 23.10 & 27.99 & 30.85 & 21 & 68 & 20.47 & 25.55 & 25.54 & 33 & 70\tabularnewline
 & {\footnotesize{}GL-LN} & 16.59 & 18.59 & 22.75 & 29 & 60 & 13.68 & 17.06 & 17.12 & 38 & 61\tabularnewline
 & {\footnotesize{}GL-Uni} & 21.51 & 25.87 & 28.83 & 24 & 61 & 17.91 & 22.94 & 22.91 & 37 & 62\tabularnewline
\cline{3-12} \cline{4-12} \cline{5-12} \cline{6-12} \cline{7-12} \cline{8-12} \cline{9-12} \cline{10-12} \cline{11-12} \cline{12-12} 
{\footnotesize{}$\delta^{0}=0.6$} & {\footnotesize{}OLS} & 17.64 & 23.51 & 25.01 & 24 & 50 & 15.51 & 20.93 & 20.91 & 41 & 59\tabularnewline
 & {\footnotesize{}GL-LN} & 13.42 & 16.63 & 18.63 & 28 & 47 & 11.06 & 14.90 & 14.38 & 46 & 54\tabularnewline
 & {\footnotesize{}GL-Uni} & 16.01 & 21.54 & 22.75 & 25 & 47 & 13.92 & 19.11 & 19.91 & 40 & 58\tabularnewline
\cline{3-12} \cline{4-12} \cline{5-12} \cline{6-12} \cline{7-12} \cline{8-12} \cline{9-12} \cline{10-12} \cline{11-12} \cline{12-12} 
{\footnotesize{}$\delta^{0}=1$} & {\footnotesize{}OLS} & 8.71 & 15.87 & 15.79 & 27 & 34 & 7.24 & 10.73 & 10.72 & 47 & 54\tabularnewline
 & {\footnotesize{}GL-LN} & 8.25 & 15.27 & 15.61 & 27 & 34 & 6.88 & 9.21 & 9.19 & 47 & 52\tabularnewline
 & {\footnotesize{}GL-Uni} & 8.65 & 14.96 & 15.21 & 27 & 33 & 6.96 & 10.44 & 10.45 & 46 & 53\tabularnewline
\hline 
\end{tabular}{\footnotesize\par}
\par\end{centering}
\end{singlespace}
\noindent\begin{minipage}[t]{1\columnwidth}%
{\scriptsize{}The model is $y_{t}=\delta^{0}\mathbf{1}\left\{ t>\left\lfloor T\lambda_{0}\right\rfloor \right\} +e_{t},\,e_{t}=0.3e_{t-1}+u_{t},\,u_{t}\sim i.i.d.\,\mathscr{N}\left(0,\,1\right),\,T=100$.
The notes of Table \ref{Table M0 Bias} apply.}%
\end{minipage}
\end{table}
\par\end{center}

\begin{center}

\par\end{center}\begin{center}
\begin{table}[H]
\caption{\label{Table M5 Bias}Small-sample accuracy of the estimates of the
break point $T_{b}^{0}$ for model M3}

\begin{centering}
{\footnotesize{}}%
\begin{tabular}{ccccccc|ccccc}
\hline 
 &  & {\footnotesize{}MAE} & {\footnotesize{}Std} & {\footnotesize{}$\textrm{RMSE}$} & {\footnotesize{}$Q_{0.25}$} & \multicolumn{1}{c}{{\footnotesize{}$Q_{0.75}$}} & {\footnotesize{}MAE} & {\footnotesize{}Std} & {\footnotesize{}$\textrm{RMSE}$} & {\footnotesize{}$Q_{0.25}$} & {\footnotesize{}$Q_{0.75}$}\tabularnewline
\cline{3-12} \cline{4-12} \cline{5-12} \cline{6-12} \cline{7-12} \cline{8-12} \cline{9-12} \cline{10-12} \cline{11-12} \cline{12-12} 
 &  & \multicolumn{5}{c|}{{\footnotesize{}$\lambda_{0}=0.3$}} & \multicolumn{5}{c}{{\footnotesize{}$\lambda_{0}=0.5$}}\tabularnewline
\cline{3-12} \cline{4-12} \cline{5-12} \cline{6-12} \cline{7-12} \cline{8-12} \cline{9-12} \cline{10-12} \cline{11-12} \cline{12-12} 
{\footnotesize{}$\delta^{0}=0.3$} & {\footnotesize{}OLS} & 23.66 & 28.14 & 31.32 & 22 & 69 & 22.01 & 26.61 & 26.59 & 33 & 72\tabularnewline
 & {\footnotesize{}GL-LN} & 19.31 & 19.22 & 26.28 & 30 & 57 & 14.89 & 18.18 & 19.08 & 39 & 61\tabularnewline
 & {\footnotesize{}GL-Uni} & 21.38 & 24.12 & 28.08 & 24 & 64 & 18.76 & 22.88 & 22.01 & 31 & 66\tabularnewline
\cline{3-12} \cline{4-12} \cline{5-12} \cline{6-12} \cline{7-12} \cline{8-12} \cline{9-12} \cline{10-12} \cline{11-12} \cline{12-12} 
{\footnotesize{}$\delta^{0}=0.4$} & {\footnotesize{}OLS} & 19.31 & 25.76 & 27.71 & 23 & 57 & 18.14 & 23.43 & 23.44 & 38 & 60\tabularnewline
 & {\footnotesize{}GL-LN} & 15.04 & 17.64 & 21.29 & 29 & 51 & 12.36 & 16.43 & 16.52 & 40 & 60\tabularnewline
 & {\footnotesize{}GL-Uni} & 18.46 & 22.74 & 25.18 & 25 & 58 & 15.91 & 20.42 & 20.42 & 37 & 62\tabularnewline
\cline{3-12} \cline{4-12} \cline{5-12} \cline{6-12} \cline{7-12} \cline{8-12} \cline{9-12} \cline{10-12} \cline{11-12} \cline{12-12} 
{\footnotesize{}$\delta^{0}=0.6$} & {\footnotesize{}OLS} & 12.02 & 19.02 & 19.82 & 25 & 37 & 10.28 & 15.51 & 15.58 & 45 & 55\tabularnewline
 & {\footnotesize{}GL-LN} & 9.29 & 12.86 & 14.61 & 29 & 40 & 8.46 & 11.84 & 11.86 & 45 & 55\tabularnewline
 & {\footnotesize{}GL-Uni} & 12.33 & 18.43 & 19.54 & 27 & 41 & 8.90 & 14.54 & 14.53 & 45 & 55\tabularnewline
\cline{3-12} \cline{4-12} \cline{5-12} \cline{6-12} \cline{7-12} \cline{8-12} \cline{9-12} \cline{10-12} \cline{11-12} \cline{12-12} 
{\footnotesize{}$\delta^{0}=1$} & {\footnotesize{}OLS} & 3.72 & 6.88 & 6.89 & 28 & 32 & 3.85 & 6.98 & 6.98 & 48 & 52\tabularnewline
 & {\footnotesize{}GL-LN} & 3.49 & 6.44 & 6.57 & 28 & 32 & 3.45 & 6.09 & 6.10 & 48 & 52\tabularnewline
 & {\footnotesize{}GL-Uni} & 4.37 & 8.12 & 8.24 & 28 & 32 & 3.86 & 6.97 & 6.96 & 48 & 52\tabularnewline
\hline 
\end{tabular}{\footnotesize\par}
\par\end{centering}
\noindent\begin{minipage}[t]{1\columnwidth}%
{\scriptsize{}The model is $y_{t}=\delta^{0}\mathbf{1}\left\{ t>\left\lfloor T\lambda_{0}\right\rfloor \right\} +\alpha^{0}y_{t-1}+e_{t},\,e_{t}\sim i.i.d.\,\mathscr{N}\left(0,\,0.5\right),\,\alpha^{0}=0.6,\,T=100$.
The notes of Table \ref{Table M0 Bias} apply.}%
\end{minipage}
\end{table}
\par\end{center}

\begin{center}

\par\end{center}\begin{center}
\begin{table}[H]
\caption{\label{Table M0}Small-sample coverage rates and lengths of the confidence
sets for model M1}

\begin{centering}
{\footnotesize{}}%
\begin{tabular}{cccccccccc}
\hline 
 &  & \multicolumn{2}{c}{{\footnotesize{}$\delta^{0}=0.4$}} & \multicolumn{2}{c}{{\footnotesize{}$\delta^{0}=0.8$}} & \multicolumn{2}{c}{{\footnotesize{}$\delta^{0}=1.2$}} & \multicolumn{2}{c}{{\footnotesize{}$\delta^{0}=1.6$}}\tabularnewline
 &  & {\footnotesize{}$\textrm{Cov.}$} & {\footnotesize{}$\textrm{Lgth.}$} & {\footnotesize{}$\textrm{Cov.}$} & {\footnotesize{}$\textrm{Lgth.}$} & {\footnotesize{}$\textrm{Cov.}$} & {\footnotesize{}$\textrm{Lgth.}$} & {\footnotesize{}$\textrm{Cov.}$} & {\footnotesize{}$\textrm{Lgth.}$}\tabularnewline
\cline{3-10} \cline{4-10} \cline{5-10} \cline{6-10} \cline{7-10} \cline{8-10} \cline{9-10} \cline{10-10} 
{\footnotesize{}$\lambda_{0}=0.5$} & {\footnotesize{}OLS-CR} & 0.922 & 77.52 & 0.934 & 49.46 & 0.946 & 22.51 & 0.938 & 10.48\tabularnewline
 & {\footnotesize{}Bai (1997)} & 0.812 & 58.12 & 0.862 & 28.75 & 0.928 & 13.78 & 0.928 & 8.16\tabularnewline
 & {\footnotesize{}$\widehat{U}_{T}\left(T_{\textrm{m}}\right).\textrm{neq}$} & 0.950 & 75.45 & 0.950 & 41.68 & 0.950 & 21.78 & 0.950 & 14.79\tabularnewline
 & {\footnotesize{}ILR} & 0.959 & 76.14 & 0.973 & 35.79 & 0.976 & 14.44 & 0.977 & 7.15\tabularnewline
 & {\footnotesize{}GL-LN} & 0.942 & 49.76 & 0.948 & 22.45 & 0.958 & 10.47 & 0.965 & 5.15\tabularnewline
 & {\footnotesize{}sup-W} & \multicolumn{2}{c}{0.384} & \multicolumn{2}{c}{0.916} & \multicolumn{2}{c}{1.000} & \multicolumn{2}{c}{1.000}\tabularnewline
\cline{3-10} \cline{4-10} \cline{5-10} \cline{6-10} \cline{7-10} \cline{8-10} \cline{9-10} \cline{10-10} 
{\footnotesize{}$\lambda_{0}=0.3$} & {\footnotesize{}OLS-CR} & 0.928 & 74.95 & 0.928 & 46.68 & 0.930 & 21.47 & 0.958 & 10.22\tabularnewline
 & {\footnotesize{}Bai (1997)} & 0.830 & 56.64 & 0.870 & 28.72 & 0.904 & 13.89 & 0.962 & 8.27\tabularnewline
 & {\footnotesize{}$\widehat{U}_{T}\left(T_{\textrm{m}}\right).\textrm{neq}$} & 0.952 & 77.51 & 0.952 & 44.72 & 0.952 & 22.51 & 0.952 & 14.21\tabularnewline
 & {\footnotesize{}ILR} & 0.952 & 78.28 & 0.966 & 39.78 & 0.969 & 31.29 & 0.968 & 18.23\tabularnewline
 & {\footnotesize{}GL-LN} & 0.942 & 49.60 & 0.948 & 23.89 & 0.958 & 11.14 & 0.980 & 5.60\tabularnewline
 & {\footnotesize{}sup-W} & \multicolumn{2}{c}{0.316} & \multicolumn{2}{c}{0.866} & \multicolumn{2}{c}{0.992} & \multicolumn{2}{c}{1.000}\tabularnewline
\hline 
\end{tabular}{\footnotesize\par}
\par\end{centering}
\noindent\begin{minipage}[t]{1\columnwidth}%
{\scriptsize{}The model is $y_{t}=\delta^{0}\mathbf{1}_{\left\{ t>\left\lfloor T\lambda_{0}\right\rfloor \right\} }+e_{t},\,e_{t}\sim i.i.d.\,\mathscr{N}\left(0,\,1\right),\,T=100$.
Cov. and Lgth. refer to the coverage probability and the average length
of the confidence set (i.e., the average number of dates in the confidence
set). sup-W refers to the rejection probability of the sup-Wald test
using a 5\% asymptotic critical value. The number of simulations is
3,000.}%
\end{minipage}
\end{table}
\par\end{center}

\begin{center}

\par\end{center}\begin{center}
\begin{table}[H]
\caption{\label{Table M1}Small-sample coverage rates and lengths of the confidence
sets for model M2}

\begin{centering}
{\footnotesize{}}%
\begin{tabular}{cccccccccc}
\hline 
 &  & \multicolumn{2}{c}{{\footnotesize{}$\delta^{0}=0.4$}} & \multicolumn{2}{c}{{\footnotesize{}$\delta^{0}=0.8$}} & \multicolumn{2}{c}{{\footnotesize{}$\delta^{0}=1.2$}} & \multicolumn{2}{c}{{\footnotesize{}$\delta^{0}=1.6$}}\tabularnewline
 &  & {\footnotesize{}$\textrm{Cov.}$} & {\footnotesize{}$\textrm{Lgth.}$} & {\footnotesize{}$\textrm{Cov.}$} & {\footnotesize{}$\textrm{Lgth.}$} & {\footnotesize{}$\textrm{Cov.}$} & {\footnotesize{}$\textrm{Lgth.}$} & {\footnotesize{}$\textrm{Cov.}$} & {\footnotesize{}$\textrm{Lgth.}$}\tabularnewline
\cline{3-10} \cline{4-10} \cline{5-10} \cline{6-10} \cline{7-10} \cline{8-10} \cline{9-10} \cline{10-10} 
{\footnotesize{}$\lambda_{0}=0.5$} & {\footnotesize{}OLS-CR} & 0.952 & 80.29 & 0.954 & 57.70 & 0.957 & 30.04 & 0.963 & 15.10\tabularnewline
 & {\footnotesize{}Bai (1997)} & 0.804 & 64.64 & 0.824 & 43.53 & 0.907 & 13.03 & 0.930 & 7.81\tabularnewline
 & {\footnotesize{}$\widehat{U}_{T}\left(T_{\textrm{m}}\right).\textrm{neq}$} & 0.967 & 87.30 & 0.967 & 72.70 & 0.957 & 36.70 & 0.957 & 30.20\tabularnewline
 & {\footnotesize{}ILR} & 0.937 & 81.88 & 0.945 & 57.43 & 0.972 & 21.99 & 0.972 & 18.96\tabularnewline
 & {\footnotesize{}GL-LN} & 0.933 & 55.13 & 0.912 & 32.97 & 0.935 & 20.03 & 0.961 & 10.62\tabularnewline
 & {\footnotesize{}sup-W} & \multicolumn{2}{c}{0.316} & \multicolumn{2}{c}{0.699} & \multicolumn{2}{c}{1.000} & \multicolumn{2}{c}{1.000}\tabularnewline
\cline{3-10} \cline{4-10} \cline{5-10} \cline{6-10} \cline{7-10} \cline{8-10} \cline{9-10} \cline{10-10} 
{\footnotesize{}$\lambda_{0}=0.3$} & {\footnotesize{}OLS-CR} & 0.945 & 79.25 & 0.957 & 54.93 & 0.962 & 29.91 & 0.970 & 15.37\tabularnewline
 & {\footnotesize{}Bai (1997)} & 0.823 & 63.79 & 0.851 & 26.33 & 0895 & 13.07 & 0.946 & 7.87\tabularnewline
 & {\footnotesize{}$\widehat{U}_{T}\left(T_{\textrm{m}}\right).\textrm{neq}$} & 0.966 & 88.23 & 0.953 & 59.66 & 0.950 & 39.65 & 0.951 & 32.39\tabularnewline
 & {\footnotesize{}ILR} & 0.945 & 84.37 & 0.945 & 62.97 & 0.971 & 33.74 & 0.987 & 17.92\tabularnewline
 & {\footnotesize{}GL-LN} & 0.945 & 53.79 & 0.923 & 34.75 & 0.934 & 19.92 & 0.944 & 10.04\tabularnewline
 & {\footnotesize{}sup-W} & \multicolumn{2}{c}{0.314} & \multicolumn{2}{c}{{\footnotesize{}0.881}} & \multicolumn{2}{c}{0.999} & \multicolumn{2}{c}{1.000}\tabularnewline
\hline 
\end{tabular}{\footnotesize\par}
\par\end{centering}
\noindent\begin{minipage}[t]{1\columnwidth}%
{\scriptsize{}The model is $y_{t}=\delta^{0}\mathbf{1}\left\{ t>\left\lfloor T\lambda_{0}\right\rfloor \right\} +e_{t},\,e_{t}=0.3e_{t-1}+u_{t},\,u_{t}\sim i.i.d.\,\mathscr{N}\left(0,\,1\right),\,T=100$.
The notes of Table \ref{Table M0} apply.}%
\end{minipage}
\end{table}
\par\end{center}

\begin{center}

\par\end{center}\begin{center}
\begin{table}[H]
\caption{\label{Table M2}Small-sample coverage rates and lengths of the confidence
sets for model M3}

\begin{centering}
{\footnotesize{}}%
\begin{tabular}{cccccccccc}
\hline 
 &  & \multicolumn{2}{c}{{\footnotesize{}$\delta^{0}=0.4$}} & \multicolumn{2}{c}{{\footnotesize{}$\delta^{0}=0.8$}} & \multicolumn{2}{c}{{\footnotesize{}$\delta^{0}=1.2$}} & \multicolumn{2}{c}{{\footnotesize{}$\delta^{0}=1.6$}}\tabularnewline
 &  & {\footnotesize{}$\textrm{Cov.}$} & {\footnotesize{}$\textrm{Lgth.}$} & {\footnotesize{}$\textrm{Cov.}$} & {\footnotesize{}$\textrm{Lgth.}$} & {\footnotesize{}$\textrm{Cov.}$} & {\footnotesize{}$\textrm{Lgth.}$} & {\footnotesize{}$\textrm{Cov.}$} & {\footnotesize{}$\textrm{Lgth.}$}\tabularnewline
\cline{3-10} \cline{4-10} \cline{5-10} \cline{6-10} \cline{7-10} \cline{8-10} \cline{9-10} \cline{10-10} 
{\footnotesize{}$\lambda_{0}=0.5$} & {\footnotesize{}OLS-CR} & 0.954 & 80.29 & 0.952 & 57.23 & 0.957 & 30.21 & 0.963 & 15.20\tabularnewline
 & {\footnotesize{}Bai (1997)} & 0.781 & 55.85 & 0.845 & 26.23 & 0.902 & 13.03 & 0.932 & 7.81\tabularnewline
 & {\footnotesize{}$\widehat{U}_{T}\left(T_{\textrm{m}}\right).\textrm{neq}$} & 0.958 & 81.28 & 0.959 & 55.34 & 0.957 & 36.71 & 0.957 & 30.20\tabularnewline
 & {\footnotesize{}ILR} & 0.934 & 65.96 & 0.956 & 33.73 & 0.975 & 21.96 & 0.984 & 17.45\tabularnewline
 & {\footnotesize{}GL-LN} & 0.912 & 60.90 & 0.925 & 32.93 & 0.964 & 19.23 & 0.971 & 9.23\tabularnewline
 & {\footnotesize{}sup-W} & \multicolumn{2}{c}{0.407} & \multicolumn{2}{c}{0.931} & \multicolumn{2}{c}{1.000} & \multicolumn{2}{c}{1.000}\tabularnewline
\cline{3-10} \cline{4-10} \cline{5-10} \cline{6-10} \cline{7-10} \cline{8-10} \cline{9-10} \cline{10-10} 
{\footnotesize{}$\lambda_{0}=0.3$} & {\footnotesize{}OLS-CR} & 0.968 & 83.69 & 0.951 & 54.13 & 0.962 & 29.31 & 0.970 & 15.37\tabularnewline
 & {\footnotesize{}Bai (1997)} & 0.795 & 64.06 & 0.853 & 26.33 & 0.896 & 13.07 & 0.946 & 7.85\tabularnewline
 & {\footnotesize{}$\widehat{U}_{T}\left(T_{\textrm{m}}\right).\textrm{neq}$} & 0.960 & 86.42 & 0.953 & 59.13 & 0.950 & 39.65 & 0.951 & 32.28\tabularnewline
 & {\footnotesize{}ILR} & 0.934 & 67.73 & 0.964 & 35.30 & 0.971 & 33.74 & 0.987 & 17.92\tabularnewline
 & {\footnotesize{}GL-LN} & 0.912 & 60.28 & 0.945 & 36.08 & 0.974 & 22.72 & 0.975 & 12.71\tabularnewline
 & {\footnotesize{}sup-W} & \multicolumn{2}{c}{0.232} & \multicolumn{2}{c}{0.884} & \multicolumn{2}{c}{0.999} & \multicolumn{2}{c}{1.000}\tabularnewline
\hline 
\end{tabular}{\footnotesize\par}
\par\end{centering}
\noindent\begin{minipage}[t]{1\columnwidth}%
{\scriptsize{}The model is $y_{t}=\delta^{0}\mathbf{1}\left\{ t>\left\lfloor T\lambda_{0}\right\rfloor \right\} +\alpha^{0}y_{t-1}+e_{t},\,e_{t}\sim i.i.d.\,\mathscr{N}\left(0,\,0.5\right),\,\alpha^{0}=0.6,\,T=100$.
The notes of Table \ref{Table M0} apply.}%
\end{minipage}
\end{table}
\par\end{center}

\pagebreak{}

\section*{}
\addcontentsline{toc}{part}{Supplemental Material}
\begin{center}
\Large{\uline{Supplemental Material} to} 
\end{center}

\begin{center}
\title{\textbf{\Large{Generalized Laplace Inference in Multiple Change-Points Models}}} 
\maketitle
\end{center}
\medskip{} 
\medskip{} 
\medskip{} 
\thispagestyle{empty}

\begin{center}
$\qquad$ \textsc{\textcolor{MyBlue}{Alessandro Casini}} $\qquad$ \textsc{\textcolor{MyBlue}{Pierre Perron}}\\
\small{University of Rome Tor Vergata} $\quad$ \small{Boston University} %\small{{Boston University}}
\\
\medskip{}
\medskip{} 
\medskip{} 
\medskip{} 
\date{\small{\today} \\
%\footnotesize{First Vesrion: \printdate{28.10.2015}}
}
\medskip{} 
\medskip{} 
\medskip{} 
\end{center}
\begin{abstract}
{\footnotesize{}This supplemental material is structured as follows.
Section \ref{Supp Math App} contains the Mathematical Appendix which
includes all proofs of the results in the paper. Section \ref{Section Comparison-to}
includes further simulation results comparing the GL-LN method to
the GL estimators proposed in \citet{casini/perron_Lap_CR_Single_Inf}.}{\footnotesize\par}
\end{abstract}
\setcounter{page}{1}
\setcounter{section}{1}
\renewcommand*{\theHsection}{\the\value{section}}

\newpage{}

\begin{singlespace} 
\noindent 
%\fontsize{10}{15}
\small

\appendix

\allowdisplaybreaks

\simpleheading

\setcounter{page}{1} \renewcommand{\thepage}{S-\arabic{page}}

\section{\label{Supp Math App}Mathematical Appendix }

The mathematical appendix is structured as follows. Section \ref{subsection: Preliminary Lemmas}
presents some preliminary lemmas which will be used in the sequel.
The proofs of the theoretical results in the paper are in Section
\ref{subsection Proofs-of-Section 3}-\ref{subsec:Proofs-Multi}.

\subsection{Additional Notation}

The $\left(i,\,j\right)$ element of $A$ is denoted by $A^{\left(i,j\right)}$.
For a matrix $A$, the orthogonal projection matrices $P_{A},\,M_{A}$
are defined as $P_{A}=A\left(A'A\right)^{-1}A'$ and $M_{A}=I-P_{A}$,
respectively. Also, for a projection matrix $P$, $\left\Vert PA\right\Vert \leq\left\Vert A\right\Vert .$
We denote the $d$-dimensional identity matrix by $I_{d}.$ When the
context is clear we omit the subscript notation in the projection
matrices. We denote the $i\times j$ upper-left (resp., lower-right)
sub-block of $A$ as $\left[A\right]_{\left\{ i\times j,\cdot\right\} }$
(resp., $\left[A\right]_{\left\{ \cdot,i\times j\right\} }$). Note
that the norm of $A$ is equal to the square root of the maximum eigenvalue
of $A'A,$ and thus, $\left\Vert A\right\Vert \leq\left[\textrm{tr}\left(A'A\right)\right]^{1/2}.$
 For a sequence of matrices $\left\{ A_{T}\right\} ,$ we write $A_{T}=o_{\mathbb{P}}\left(1\right)$
if each of its elements is $o_{\mathbb{P}}\left(1\right)$ and likewise
for $O_{\mathbb{P}}\left(1\right).$ For a random variable $\xi$
and a number $r\geq1,$ $\left\Vert \xi\right\Vert _{r}=\left(\mathbb{E}\left\Vert \xi\right\Vert ^{r}\right)^{1/r}.$
$K$ is a generic constant that may vary from line to line; we may
sometime write $K_{r}$ to emphasize the dependence of $K$ on a number
$r.$ For two scalars $a$ and $b$,  $a\wedge b=\inf\left\{ a,\,b\right\} $.
We may use $\sum_{k}$ when the limits of the summation are clear
from the context. Unless otherwise sated $\mathbf{A}^{c}$ denotes
the complementary set of $\mathbf{A}$.

\subsection{Preliminary Lemmas \label{subsection: Preliminary Lemmas}}

We first present results related to the extremum criterion function
$Q_{T}\left(\delta\left(T_{b}\right),\,T_{b}\right)$ under the following
assumption (Assumptions \ref{Assumption The-loss-function LapBai97}-\ref{Assumption Prior LapBai97}
are not needed in this section). 
\begin{assumption}
\label{Ass Summary LapBai97}We consider model \eqref{Eq Matrix Format of the model}
with Assumptions \ref{Assumption A Bai97}-\ref{Assumption A4 Bai97}
and \ref{Assumption Small Shift BP}-\ref{Assumption A.9b Bai 97, LapBai97}.
\end{assumption}
\begin{lem}
\label{Lemma A1, Bai (1997)} The following inequalities hold $\mathbb{P}$-a.s.:
\begin{align}
\left(Z_{0}'MZ_{0}\right)-\left(Z_{0}'MZ_{2}\right)\left(Z'_{2}MZ_{2}\right)^{-1}\left(Z'_{2}MZ_{0}\right) & \geq D'\left(X'_{\Delta}X_{\Delta}\right)\left(X'_{2}X_{2}\right)^{-1}\left(X'_{0}X_{0}\right)D,\qquad T_{b}<T_{b}^{0}\label{Equation (36)}\\
\left(Z_{0}'MZ_{0}\right)-\left(Z_{0}'MZ_{2}\right)\left(Z'_{2}MZ_{2}\right)^{-1}\left(Z'_{2}MZ_{0}\right) & \geq D'\left(X'_{\Delta}X_{\Delta}\right)\left(X'X-X'_{2}X_{2}\right)^{-1}\left(X'X-X'_{0}X_{0}\right)D,\qquad T_{b}\geq T_{b}^{0}\label{Eq. 37}
\end{align}
\end{lem}
\begin{proof}
See Lemma A.1 in \citet{bai:97RES}. 
\end{proof}
Recall that $Q_{T}\left(\delta\left(\lambda_{b}\right),\,\lambda_{b}\right)=\delta\left(T_{b}\right)\left(Z_{2}'MZ_{2}\right)\delta\left(T_{b}\right)$.
We decompose $Q_{T}\left(\delta\left(\lambda_{b}\right),\,\lambda_{b}\right)-Q_{T}\left(\delta\left(\lambda_{b}^{0}\right),\,\lambda_{b}^{0}\right)$
into a ``deterministic'' and a ``stochastic'' component. It follows
by definition that,
\begin{align*}
\delta\left(\lambda_{b}\right) & =\left(Z'_{2}MZ_{2}\right)^{-1}\left(Z'_{2}MY\right)=\left(Z'_{2}MZ_{2}\right)^{-1}\left(Z'_{2}MZ_{0}\right)\delta_{T}+\left(Z'_{2}MZ_{2}\right)^{-1}Z_{2}Me,
\end{align*}
 and 
\begin{align*}
\delta\left(\lambda_{b}^{0}\right) & =\left(Z'_{0}MZ_{0}\right)^{-1}\left(Z'_{0}MY\right)=\delta_{T}+\left(Z'_{0}MZ_{0}\right)^{-1}\left(Z'_{0}Me\right).
\end{align*}
 Therefore
\begin{align}
Q_{T}\left(\delta\left(\lambda_{b}\right),\,\lambda_{b}\right)-Q_{T}\left(\delta\left(\lambda_{b}^{0}\right),\,\lambda_{b}^{0}\right) & =\delta\left(\lambda_{b}\right)'\left(Z_{2}'MZ_{2}\right)\delta\left(\lambda_{b}\right)-\delta\left(\lambda_{b}^{0}\right)'\left(Z_{0}'MZ_{0}\right)\delta\left(\lambda_{b}^{0}\right)\label{eq. A.2.0}\\
 & \triangleq g_{d}\left(\delta_{T},\,\lambda_{b}\right)+g_{e}\left(\delta_{T},\,\lambda_{b}\right),
\end{align}
 where 
\begin{align}
g_{d}\left(\delta_{T},\,\lambda_{b}\right) & =\delta'_{T}\left\{ \left(Z'_{0}MZ_{2}\right)\left(Z'_{2}MZ_{2}\right)^{-1}\left(Z'_{2}MZ_{0}\right)-Z_{0}'MZ_{0}\right\} \delta_{T},\label{eq. A.2.1}
\end{align}
and
\begin{align}
g_{e}\left(\delta_{T},\,\lambda_{b}\right) & =2\delta_{T}'\left(Z'_{0}MZ_{2}\right)\left(Z'_{2}MZ_{2}\right)^{-1}Z_{2}Me-2\delta_{T}'\left(Z'_{0}Me\right)\label{eq. A.2.3}\\
 & \quad+e'MZ_{2}\left(Z'_{2}MZ_{2}\right)^{-1}Z_{2}Me-e'MZ_{0}\left(Z'_{0}MZ_{0}\right)^{-1}Z'_{0}Me.\label{eq. A.2.4}
\end{align}
 \eqref{eq. A.2.1} constitutes the deterministic component and $g_{e}\left(\delta_{T},\,\lambda_{b}\right)$
the stochastic one. Denote 
\begin{flalign*}
 &  & X_{\Delta} & \triangleq X_{2}-X_{0}=\left(0,\,\ldots,\,0,\,x_{T_{b}+1},\ldots,\,x_{T_{b}^{0}},\,0,\ldots,\,\right)', & \textrm{for }T_{b}<T_{b}^{0}\\
 &  & X_{\Delta} & \triangleq-\left(X_{2}-X_{0}\right)=\left(0,\,\ldots,\,0,\,x_{T_{b}^{0}+1},\ldots,\,x_{T_{b}},\,0,\ldots,\,\right)', & \textrm{for }T_{b}>T_{b}^{0}
\end{flalign*}
whereas $X_{\Delta}\triangleq0$ when $T_{b}=T_{b}^{0}$. Observe
that $X_{2}=X_{0}+X_{\Delta}\textrm{sign}\left(T_{b}^{0}-T_{b}\right)$.
When the sign is immaterial, we simply write $X_{2}=X_{0}+X_{\Delta}$.
Next, let $Z_{\Delta}=X_{\Delta}D$, and define
\begin{align}
\overline{g}_{d}\left(\delta_{T},\,\lambda_{b}\right) & \triangleq-\frac{g_{d}\left(\delta_{T},\,\lambda_{b}\right)}{\left|T_{b}-T_{b}^{0}\right|}.\label{eq. A.2.5}
\end{align}
We arbitrarily define $\overline{g}_{d}\left(\delta^{0},\,\lambda_{b}\right)=\delta'_{T}\delta_{T}$
when $\lambda_{b}=\lambda_{b}^{0}$. Observe that $\overline{g}_{d}\left(\delta_{T},\,\lambda_{b}\right)$
is non-negative because the matrix inside the braces in \eqref{eq. A.2.1}
is negative semidefinite. \eqref{eq. A.2.0} can be written as
\begin{align}
Q_{T}\left(\delta\left(\lambda_{b}\right),\,\lambda_{b}\right)-Q_{T}\left(\delta\left(\lambda_{b}^{0}\right),\,\lambda_{b}^{0}\right) & =-\left|T_{b}-T_{b}^{0}\right|\overline{g}_{d}\left(\delta_{T},\,\lambda_{b}\right)+g_{e}\left(\delta_{T},\,\lambda_{b}\right),\qquad\textrm{for all }\lambda_{b}.\label{eq. A.2.6 LapBai97}
\end{align}
We use the notation $u=T\left\Vert \delta_{T}\right\Vert ^{2}\left(\lambda_{b}-\lambda_{b}^{0}\right)$.
For $\eta>0,$ let $B_{T,\eta}\triangleq\left\{ T_{b}:\,\left|T_{b}-T_{b}^{0}\right|\leq T\eta\right\} ,$
$B_{T,K}\triangleq\left\{ T_{b}:\,\left|T_{b}-T_{b}^{0}\right|\leq K/\left\Vert \delta_{T}\right\Vert ^{2}\right\} $
and $B_{T,K}^{c}\triangleq\left\{ T_{b}:\,T\eta\geq\left|T_{b}-T_{b}^{0}\right|>K/\left\Vert \delta_{T}\right\Vert ^{2}\right\} ,$
with $K>0.$ Note that $B_{T,\eta}=B_{T,K}\cup B_{T,K}^{c}$. Further,
let $B_{T,\eta}^{c}\triangleq\left\{ T_{b}:\,\left|T_{b}-T_{b}^{0}\right|>T\eta\right\} $.
\begin{lem}
\label{Lemma, Area (ia), Bai A.5}Under Assumption \ref{Ass Summary LapBai97},
$Q_{T}\left(\delta\left(\lambda_{b}\right),\,\lambda_{b}\right)-Q_{T}\left(\delta\left(\lambda_{b}^{0}\right),\,\lambda_{b}^{0}\right)=-\delta'_{T}Z'_{\Delta}Z_{\Delta}\delta{}_{T}+2\mathrm{sgn}\left(T_{b}^{0}-T_{b}\right)\delta'_{T}Z'_{\Delta}e+o_{\mathbb{P}}\left(1\right),$
uniformly on $B_{T,K}$ for $K$ large enough. 
\end{lem}
\begin{proof}
It follows from Lemma A.5 in \citet{bai:97RES}.
\end{proof}
\begin{lem}
\label{Lemma ZZ conv prob}Under Assumption \ref{Ass Summary LapBai97},
for $T_{b}=T_{b}^{0}+\left\lfloor u/\left\Vert \delta_{T}\right\Vert ^{2}\right\rfloor ,$
we have $\delta'_{T}Z'_{\Delta}Z_{\Delta}\delta{}_{T}=\delta'_{T}\sum_{t=T_{b}+1}^{T_{b}^{0}}z_{t}z'_{t}\delta{}_{T}=\left|u\right|\left(\delta^{0}\right)'\overline{V}\delta^{0}+o_{\mathbb{P}}\left(1\right)$,
where $\overline{V}=V_{1}$ if $u\leq0$ and $\overline{V}=V_{2}$
if $u>0$. 
\end{lem}
\begin{proof}
It follows from basic arguments (cf. Assumptions \ref{A.9a Bai 97}-\ref{Assumption A.9b Bai 97, LapBai97}).
\end{proof}
\begin{lem}
\label{Lemma, Area (ib) Q}Under Assumption \ref{Ass Summary LapBai97},
for any $\epsilon>0$ there exists a $C<\infty$ and a positive sequence
$\left\{ \nu_{T}\right\} $, with $\nu_{T}\rightarrow\infty$ as $T\rightarrow\infty$,
such that 
\begin{align*}
\liminf_{T\rightarrow\infty}\,\mathbb{P}\left[\sup_{K\leq\left|u\right|\leq\eta T\left\Vert \delta_{T}\right\Vert ^{2}}Q_{T}\left(\delta\left(\lambda_{b}\right),\,\lambda_{b}\right)-Q_{T}\left(\delta\left(\lambda_{b}^{0}\right),\,\lambda_{b}^{0}\right)<-C\nu_{T}\right] & \geq1-\epsilon,
\end{align*}
for all sufficiently large $K$ and a sufficiently small $\eta>0.$ 
\end{lem}
\begin{proof}
Note that on $\left\{ K\leq\left|u\right|\leq\eta T\left\Vert \delta_{T}\right\Vert ^{2}\right\} $
we have $K/\left\Vert \delta_{T}\right\Vert ^{2}\leq\left|T_{b}-T_{b}^{0}\right|\leq\eta T$.
In view of \eqref{eq. A.2.5}, the statement $Q_{T}\left(\delta\left(\lambda_{b}\right),\,\lambda_{b}\right)-Q_{T}\left(\delta\left(\lambda_{b}^{0}\right),\,\lambda_{b}^{0}\right)<-C\nu_{T}$
follows from showing that as $T\rightarrow\infty,$ 
\begin{align*}
\mathbb{P}\left(\sup_{T_{b}\in B_{K,T}^{c}}g_{e}\left(\delta_{T},\,\lambda_{b}\right)\geq\inf_{T_{b}\in B_{K,T}^{c}}\left|T_{b}-T_{b}^{0}\right|^{\kappa}\overline{g}_{d}\left(\delta_{T},\,\lambda_{b}\right)\right)<\epsilon & ,
\end{align*}
 where $\kappa\in\left(1/2,\,1\right).$ Suppose $T_{b}<T_{b}^{0}$.
We show that
\begin{align}
\mathbb{P}\left(\sup_{T\lambda_{b}\in B_{K,T}^{c}}\frac{\left\Vert \delta_{T}\right\Vert }{K}g_{e}\left(\delta_{T},\,\lambda_{b}\right)\geq\frac{1}{\left\Vert \delta_{T}\right\Vert ^{2\kappa-1}}\left(\frac{1}{K}\right)^{1-\kappa}\inf_{T\lambda_{b}\in B_{K,T}^{c}}\overline{g}_{d}\left(\delta_{T},\,\lambda_{b}\right)\right) & <\epsilon.\label{eq. AA1}
\end{align}
 Lemma \ref{Lemma A.2, LapBai97}-(ii) stated below implies that $\inf_{T_{b}\in B_{T,K}^{c}}\overline{g}_{d}\left(\delta_{T},\,\lambda_{b}\right)$
is bounded away from zero as $T\rightarrow\infty$ for large $K$
and small $\eta.$ Next, we show that 
\begin{align}
\sup_{T\lambda_{b}\in B_{K,T}^{c}}K^{-1}\left\Vert \delta_{T}\right\Vert g_{e}\left(\delta_{T},\,\lambda_{b}\right) & =o_{\mathbb{P}}\left(1\right).\label{eq AA2}
\end{align}
 Consider the first term of \eqref{eq. A.2.3},
\begin{align*}
2\delta_{T}'\left(Z'_{0}MZ_{2}\right)\left(Z'_{2}MZ_{2}\right)^{-1}Z_{2}Me & =2\delta_{T}'\left(Z'_{0}MZ_{2}/T\right)\left(Z'_{2}MZ_{2}/T\right)^{-1}Z_{2}Me\\
 & =2C\left\Vert \delta_{T}\right\Vert O_{\mathbb{P}}\left(1\right)O_{\mathbb{P}}\left(1\right)O_{\mathbb{P}}\left(T^{1/2}\right)=CO_{\mathbb{P}}\left(\left\Vert \delta_{T}\right\Vert T^{1/2}\right).
\end{align*}
When multiplied by $\left\Vert \delta_{T}\right\Vert /K,$ this term
is $O_{\mathbb{P}}\left(\left\Vert \delta_{T}\right\Vert ^{2}T^{1/2}/K\right)$
which goes to zero for large $K$.. The second term in \eqref{eq. A.2.3},
when multiplied by $\left\Vert \delta_{T}\right\Vert /K,$ is
\begin{align*}
2K^{-1}\left\Vert \delta_{T}\right\Vert \delta_{T}'\left(Z'_{0}Me\right) & =K^{-1}\left\Vert \delta_{T}\right\Vert O_{\mathbb{P}}\left(\left\Vert \delta_{T}\right\Vert T^{1/2}\right)=K^{-1}O_{\mathbb{P}}\left(\left\Vert \delta_{T}\right\Vert ^{2}T^{1/2}\right),
\end{align*}
 which converges to zero using the same argument as for the first
term. Consider now the first term of \eqref{eq. A.2.4}, $T^{-1/2}e'MZ_{2}\left(Z'_{2}MZ_{2}/T\right)^{-1}T^{-1/2}Z_{2}Me=O_{\mathbb{P}}\left(1\right).$
A similar argument can be used for the second term which is also $O_{\mathbb{P}}\left(1\right)$.
The latter two terms multiplied by $\left\Vert \delta_{T}\right\Vert /K$
is $O_{\mathbb{P}}\left(\left\Vert \delta_{T}\right\Vert /K\right)=o_{\mathbb{P}}\left(1\right).$
This proves \eqref{eq AA2} and thus \eqref{eq. AA1}. To conclude
the proof, note that $\kappa\in\left(1/2,\,1\right)$ implies $\left\Vert \delta_{T}\right\Vert ^{-\left(2\kappa-1\right)}\rightarrow\infty$,
so that we can choose $\nu_{T}=\left(\left\Vert \delta_{T}\right\Vert ^{2}/K\right)^{-\left(1-\kappa\right)}$. 
\end{proof}
\begin{lem}
\label{Lemma A.2, LapBai97}Let $\widetilde{g}_{d}\triangleq\inf_{T\left|\lambda_{b}-\lambda_{b}^{0}\right|>K\left\Vert \delta_{T}\right\Vert ^{-2}}\overline{g}_{d}\left(\delta_{T},\,\lambda_{b}\right).$
Under Assumption \ref{Ass Summary LapBai97}, \\
(i) for any $\epsilon>0$ there exists some $C>0$ such that $\liminf_{T\rightarrow\infty}\mathbb{P}\left(\widetilde{g}_{d}>C\left\Vert \delta_{T}\right\Vert ^{2}\right)\leq1-\epsilon$;\\
(ii) with $B_{T,K}^{c}=\left\{ T_{b}:\,T\eta\geq\left|T_{b}-T_{b}^{0}\right|\geq K/\left\Vert \delta_{T}\right\Vert ^{2}\right\} ,$
for any $\epsilon>0$ there exists a $C>0$ such that $\liminf_{T\rightarrow\infty}\mathbb{P}\left(\inf_{T\lambda_{b}\in B_{T,K}^{c}}\overline{g}_{d}\left(\delta_{T},\,\lambda_{b}\right)>C\right)\leq1-\epsilon.$ 
\end{lem}
\begin{proof}
Part (i) was proved in Lemma A.2 of \citet{bai:97RES}. As for part
(ii), by Lemma \ref{Lemma A1, Bai (1997)},
\begin{align*}
\overline{g}_{d}\left(\delta^{0},\,\lambda_{b}\right) & \geq\delta_{T}D'\frac{X'_{\Delta}X_{\Delta}}{T_{b}^{0}-T_{b}}\left(X'_{2}X_{2}\right)^{-1}\left(X'_{0}X_{0}\right)D\delta_{T}\geq\lambda_{J,T_{b}},
\end{align*}
 where $\lambda_{J,T_{b}}$ is the minimum eigenvalue of $D'J\left(T_{b}\right)D$,
with $J\left(T_{b}\right)\triangleq\left\Vert \delta_{T}\right\Vert ^{2}\left(T_{b}^{0}-\lambda_{b}\right)^{-1}X'_{\Delta}X_{\Delta}\left(X'_{2}X_{2}\right)^{-1}\left(X'_{0}X_{0}\right).$
It is sufficient to show that, for $T_{b}\in B_{T,K}^{c},$ $\lambda_{J,T_{b}}$
is bounded away from zero with large probability for large $K$ and
small $\eta.$ We have $\left\Vert J\left(T_{b}\right)^{-1}\right\Vert \leq\left\Vert \left[\left\Vert \delta_{T}\right\Vert ^{2}\left(T_{b}^{0}-T_{b}\right)^{-1}X'_{\Delta}X_{\Delta}\right]^{-1}\right\Vert \left\Vert \left(X'_{2}X_{2}\right)\left(X'_{0}X_{0}\right)^{-1}\right\Vert $
and by Assumptions \ref{Assumption A3 Bai97}-\ref{Assumption A4 Bai97}
$\left\Vert \left(X'_{2}X_{2}\right)\left(X'_{0}X_{0}\right)^{-1}\right\Vert \leq\left\Vert X'X\right\Vert \left\Vert \left(X'_{0}X_{0}\right)^{-1}\right\Vert $
is bounded. Next, note that $\left(T_{b}^{0}-T_{b}\right)^{-1}X'_{\Delta}X_{\Delta}=\left(T_{b}^{0}-T_{b}\right)^{-1}\sum_{t=T_{b}+1}^{T_{b}^{0}}x_{t}x'_{t}$
is larger than $\left(T\eta\right)^{-1}\sum_{t=T_{b}^{0}-\left\lfloor K/\left\Vert \delta_{T}\right\Vert ^{2}\right\rfloor }^{T_{b}^{0}}x_{t}x'_{t}$
on $B_{T,K}^{c},$ and for all $K,$ $\left(\left\Vert \delta_{T}\right\Vert ^{2}/K\right)\sum_{t=T_{b}^{0}-\left\lfloor K/\left\Vert \delta_{T}\right\Vert ^{2}\right\rfloor }^{T_{b}^{0}}x_{t}x'_{t}$
is positive definite with large probability as $T\rightarrow\infty$
by Assumption \ref{Assumption A3 Bai97}. Now, $\left(K/T\eta\right)\left(\left\Vert \delta_{T}\right\Vert ^{2}/K\right)\sum_{t=T_{b}^{0}-\left\lfloor K/\left\Vert \delta_{T}\right\Vert ^{2}\right\rfloor }^{T_{b}^{0}}x_{t}x'_{t}=O_{\mathbb{P}}\left(1\right),$
by choosing sufficiently large $K$ and  small $\eta.$ Thus, $\left\Vert \left[\left\Vert \delta_{T}\right\Vert ^{2}\left(T_{b}^{0}-T_{b}\right)^{-1}X'_{\Delta}X_{\Delta}\right]^{-1}\right\Vert $
is bounded with large probability for such large $K$ and small $\eta,$
which in turn implies that $\left\Vert J\left(T_{b}\right){}^{-1}\right\Vert $
is bounded. Since $D$ has full column rank, $\lambda_{J,T_{b}}$
is bounded away from zero for sufficiently large $K$ and small $\eta.$ 
\end{proof}
\begin{lem}
\label{Lemma, Area (ii), Q LapBai97}Under Assumption \ref{Ass Summary LapBai97},
for any $\epsilon>0$ there exists a $C>0$ such that 
\begin{align*}
\liminf_{T\rightarrow\infty}\mathbb{P}\left[\sup_{\left|u\right|\geq T\left\Vert \delta_{T}\right\Vert ^{2}\eta}Q_{T}\left(\delta\left(\lambda_{b}\right),\,\lambda_{b}\right)-Q_{T}\left(\delta\left(\lambda_{b}^{0}\right),\,\lambda_{b}^{0}\right)<-C\nu_{T}\right] & \geq1-\epsilon,
\end{align*}
for every $\eta>0,$ where $\nu_{T}\rightarrow\infty.$ 
\end{lem}
\begin{proof}
Fix any $\eta>0.$ Note that on $\left\{ \left|u\right|\geq T\left\Vert \delta_{T}\right\Vert ^{2}\eta\right\} $
we have $\left|T_{b}-T_{b}^{0}\right|\geq T\eta$.  We proceed
in a similar manner to Lemma \ref{Lemma, Area (ib) Q}. Let $B_{T,\eta}^{c}\triangleq\left\{ T_{b}:\,\left|T_{b}-T_{b}^{0}\right|\geq T\eta\right\} $
and recall \eqref{eq. A.2.5}. First,  as in Lemma \ref{Lemma A.2, LapBai97}-(i),
we have $\inf_{T\lambda_{b}\in B_{T,\eta}^{c}}\overline{g}_{d}\left(\delta_{T},\,\lambda_{b}\right)\geq C\left\Vert \delta_{T}\right\Vert ^{2}$
with large probability for some $C>0.$ Noting that $T\eta\inf_{T\lambda_{b}\in B_{T,\eta}^{c}}\overline{g}_{d}\left(\delta_{T},\,\lambda_{b}\right)$
diverges at rate $\tau_{T}=T\left\Vert \delta_{T}\right\Vert ^{2},$
the claim follows if we can show that $g_{e}\left(\delta_{T},\,\lambda_{b}\right)=O_{\mathbb{P}}\left(\tau_{T}^{\varpi}\right)$,
with $0\leq\varpi<1$ uniformly on $B_{T,\eta}^{c}.$ This is shown
in Lemma \ref{Lemma B.11, Q LapBai97} below, which suggests setting
$\varpi\in\left(1/2,\,1\right)$. Then, choose $\nu_{T}=\left(T\left\Vert \delta_{T}\right\Vert ^{2}\right)^{1-\varpi}$.
\end{proof}
\begin{lem}
\label{Lemma B.11, Q LapBai97}Under Assumption \ref{Ass Summary LapBai97},
uniformly on $B_{T,\eta}^{c},$ $\left|g_{e}\left(\delta_{T},\,\lambda_{b}\right)\right|=O_{\mathbb{P}}\left(\left\Vert \delta_{T}\right\Vert T^{1/2}\log T\right).$
\end{lem}
\begin{proof}
We show that $T^{-1}\left|g_{e}\left(\delta^{0},\,\lambda_{b}\right)\right|=O_{\mathbb{P}}\left(\left\Vert \delta_{T}\right\Vert T^{-1/2}\log T\right)$
uniformly on $B_{T,\eta}^{c}.$ Note that 
\begin{align*}
\sup_{T\lambda_{b}\in B_{T,\eta}^{c}}\left|g_{e}\left(\delta_{T},\,\lambda_{b}\right)\right| & \leq\sup_{q\leq T\lambda_{b}\leq T-q}\left|g_{e}\left(\delta_{T},\,\lambda_{b}\right)\right|,
\end{align*}
 and recall that $q=\mathrm{dim}\left(z_{t}\right)$ is needed for
identification. Observe that 
\begin{align}
\sup_{q\leq T_{b}\leq T-q}\left\Vert \left(Z'_{2}MZ_{2}\right)^{-1/2}Z'_{2}Me\right\Vert  & =O_{\mathbb{P}}\left(\log T\right),\label{eq 15}
\end{align}
 by the law of iterated logarithms {[}cf. \citet{billingsley:95},
Ch. 1, Theorem 9.5{]}. Next, 
\begin{align}
\sup_{q\leq T_{b}\leq T-q}T^{-1/2}\left(Z'_{0}MZ_{2}\right)\left(Z'_{2}MZ_{2}\right)^{-1/2} & =O_{\mathbb{P}}\left(1\right),\label{eq 15b}
\end{align}
 which can be proved using the inequality $\left(Z'_{0}MZ_{2}\right)\left(Z'_{2}MZ_{2}\right)\left(Z'_{0}MZ_{2}\right)\leq Z'_{0}MZ_{0}=O_{\mathbb{P}}\left(T\right)$
(valid for all $T_{b}$). Thus, by \eqref{eq 15} and \eqref{eq 15b},
the first term on the right-hand side of \eqref{eq. A.2.3} multiplied
by $T^{-1}$ is such that 
\begin{align}
\sup_{q\leq T_{b}\leq T-q}2\delta'_{T}T^{-1}\left(Z'_{0}MZ_{2}\right)\left(Z'_{2}MZ_{2}\right)^{-1}Z'_{2}Me & =O_{\mathbb{P}}\left(\left\Vert \delta_{T}\right\Vert T^{-1/2}\log T\right).\label{eq A.2.3- 1}
\end{align}
 The second term on the right-hand side of \eqref{eq. A.2.3} is $2\delta_{T}'Z'_{0}Me=O_{\mathbb{P}}\left(\left\Vert \delta_{T}\right\Vert T^{1/2}\right).$
Using \eqref{eq 15}, and dividing by $T$, the first term of \eqref{eq. A.2.4}
is $O_{\mathbb{P}}\left(\left(\log T\right)^{2}/T\right)$ while the
last term is $O_{\mathbb{P}}\left(T^{-1}\right).$ When divided by
$T$, they are of order $O_{\mathbb{P}}\left(\left(\log T\right)^{2}/T\right)$
and $O_{\mathbb{P}}\left(T^{-1}\right),$ respectively. Therefore,
$\left|g_{e}\left(\delta^{0},\,\lambda_{b}\right)\right|=O_{\mathbb{P}}\left(\left\Vert \delta_{T}\right\Vert T^{1/2}\log T\right),$
uniformly on $B_{T,\eta}^{c}.$ 
\end{proof}

\subsection{\label{subsection Proofs-of-Section 3}Proofs of Results in Section
\ref{Section Asymptotic Results LapBai97}}

We denote by $\boldsymbol{P}$ the class of polynomial functions $p:\,\mathbb{R}\rightarrow\mathbb{R}$.
Let $\mathbf{U}_{T}\triangleq\left\{ u\in\mathbb{R}:\,\lambda_{b}^{0}+u/\psi_{T}\in\varGamma^{0}\right\} $,
$\Gamma_{T,\psi}\triangleq\left\{ u\in\mathbb{R}:\,\left|u\right|\leq\psi_{T}\right\} ,$
$\Gamma_{T,\psi}^{c}\triangleq\mathbb{R}-\Gamma_{T,\psi},$ and $\widetilde{\mathbf{U}}_{T}^{c}\triangleq\mathbf{U}_{T}-\Gamma_{T,\psi}$.
 For $u\in\mathbb{R}$, let $R_{T,v}\left(u\right)\triangleq Q_{T,v}\left(u\right)-\varLambda^{0}\left(u\right)$
and $\overline{G}_{T,v}\left(u\right)\triangleq\sup_{\widetilde{v}\in\mathbf{V}}\widetilde{G}_{T,v}\left(u,\,\widetilde{v}\right)$.
The generic constant $0<C<\infty$ used below may change from line
to line. Finally, let $\widetilde{\gamma}_{T}\triangleq\gamma_{T}/T\left\Vert \delta_{T}\right\Vert ^{2}.$ 

\subsubsection{Proof of Proposition \ref{Proposition: Consistency and Rate of Convergence}}

We begin with the proof for the case of a fixed shift.
\begin{lem}
\label{Lemma Consistency Fixed delta}Under Assumptions \ref{Assumption A Bai97}-\ref{Assumption A4 Bai97},
\ref{Assumption The-loss-function LapBai97}-\ref{Assumption Small Shift BP}
(except that $\delta_{T}=\delta^{0}$) and \ref{Assumption Gaussian Process for Lap LapBai97}-(i),
$\widehat{\lambda}_{b}^{\mathrm{GL}}=\lambda_{b}^{0}+o_{\mathbb{P}}\left(1\right)$.
\end{lem}
\begin{proof}
Let $\overline{S}_{T}\left(\delta\left(\lambda_{b}\right),\,\lambda_{b}\right)\triangleq Q_{T}\left(\delta\left(\lambda_{b}\right),\,\lambda_{b}\right)-Q_{T}\left(\delta\left(\lambda_{b}^{0}\right),\,\lambda_{b}^{0}\right)$.
From \eqref{eq. A.2.6 LapBai97}, 
\[
\overline{S}_{T}\left(\widehat{\delta}\left(\lambda_{b}\right),\,\lambda_{b}\right)=-\left|T_{b}-T_{b}^{0}\right|\overline{g}_{d}\left(\delta^{0},\,T_{b}\right)+g_{e}\left(\delta^{0},\,T_{b}\right),
\]
where $g_{e}\left(\delta^{0},\,T_{b}\right)$ and $\overline{g}_{d}\left(\delta^{0},\,T_{b}\right)$
are defined in \eqref{eq. A.2.3}-\eqref{eq. A.2.5}. By Lemma A.24
in \citet{bai:97RES}, $\liminf_{T\rightarrow\infty}\overline{g}_{d}$
$\left(\delta^{0},\,T_{b}\right)>0$ and $T^{-1}\sup_{T_{b}}\left|g_{e}\left(\delta^{0},\,T_{b}\right)\right|=O_{\mathbb{P}}\left(T^{-1/2}\log T\right)$.
Thus, for any $B>0$ if $\left|\widehat{\lambda}_{b}^{\mathrm{GL}}-\lambda_{b}^{0}\right|>B$
we have that,
\begin{align}
-\overline{S}_{T}\left(\widehat{\delta}\left(\lambda_{b}\right),\,\lambda_{b}\right)\rightarrow\infty & \,\mathrm{at\,rate}\,TB.\label{eq" S_bar goe to - inf}
\end{align}
Let $p_{T}\left(u\right)\triangleq p_{1,T}\left(u\right)/\overline{p}_{T}$
with $p_{1,T}\left(u\right)=\exp\left(Q_{T}\left(\delta\left(u\right),\,u\right)\right)$
and $\overline{p}_{T}\triangleq\int_{\mathbf{U}_{T}}p_{1,T}\left(w\right)dw$.
By definition, $\widehat{\lambda}_{b}^{\mathrm{GL}}$ is the minimum
of the function $\int_{\varGamma^{0}}l\left(s-u\right)p_{1,T}\left(u\right)\pi\left(u\right)du$
with $s\in\varGamma^{0}$. Using a change in variables,
\begin{align*}
\int_{\varGamma^{0}} & l\left(s-u\right)p_{1,T}\left(u\right)\pi\left(u\right)du\\
 & =T^{-1}\overline{p}_{T}\int_{\mathbf{U}_{T}}l\left(T\left(s-\lambda_{b}^{0}\right)-u\right)p_{T}\left(\lambda_{b}^{0}+T^{-1}u\right)\pi\left(\lambda_{b}^{0}+T^{-1}u\right)du,
\end{align*}
 where $\mathbf{U}_{T}\triangleq\left\{ u\in\mathbb{R}:\,\lambda_{b}^{0}+T^{-1}u\in\varGamma^{0}\right\} $.
Thus, $\lambda_{\delta,T}\triangleq T\left(\widehat{\lambda}_{b}^{\mathrm{GL}}-\lambda_{b}^{0}\right)$
is the\textbf{ }minimum of the function, 
\begin{align*}
\mathcal{S}_{T}\left(s\right) & \triangleq\int_{\mathbf{U}_{T}}l\left(s-u\right)\frac{p_{T}\left(\lambda_{b}^{0}+T^{-1}u\right)\pi\left(\lambda_{b}^{0}+T^{-1}u\right)}{\int_{\mathbf{U}_{T}}p_{T}\left(\lambda_{b}^{0}+T^{-1}w\right)\pi\left(\lambda_{b}^{0}+T^{-1}w\right)dw}du,
\end{align*}
where the optimization is over $\mathbf{U}_{T}$.  We shall show
that for any $B>0,$ 
\begin{align}
\mathbb{P}\left[\left|\widehat{\lambda}_{b}^{\mathrm{GL}}-\lambda_{b}^{0}\right|>B\right] & \leq\mathbb{P}\left[\inf_{\left|s\right|>TB}\mathcal{S}_{T}\left(s\right)\leq\mathcal{S}_{T}\left(0\right)\right]\rightarrow0.\label{eq. AA5 LapBai97-1-1}
\end{align}
 By assumption the prior is bounded and so we can proceed the proof
for the case $\pi\left(u\right)=1$ for all $u$. By the properties
of the family $\boldsymbol{L}$ of loss functions, we can find $\overline{u}_{1},\,\overline{u}_{2}\in\mathbb{R},$
with $0<\overline{u}_{1}<\overline{u}_{2}$ such that as $T$ increases,
\begin{align*}
\overline{l}_{1,T}\triangleq\sup\left\{ l\left(u\right):\,u\in\Gamma_{1,T}\right\}  & <\overline{l}_{2,T}\triangleq\inf\left\{ l\left(u\right):\,u\in\Gamma_{2,T}\right\} ,
\end{align*}
 where $\Gamma_{1,T}\triangleq\mathbf{U}_{T}\cap\left(\left|u\right|\leq\overline{u}_{1}\right)$
and $\Gamma_{2,T}\triangleq\mathbf{U}_{T}\cap\left(\left|u\right|>\overline{u}_{2}\right)$.
With this notation, 
\begin{align*}
\mathcal{S}_{T}\left(0\right) & \leq\overline{l}_{1,T}\int_{\Gamma_{1,T}}p_{T}\left(u\right)du+\int_{\mathbf{U}_{T}\cap\left(\left|u\right|>\overline{u}_{1}\right)}l\left(u\right)p_{T}\left(u\right)du.
\end{align*}
 If $l\in\boldsymbol{L}$ then for a sufficiently large $T$ the following
relationship holds: $l\left(u\right)-\inf_{\left|v\right|>TB/2}l\left(v\right)\leq0$,
$\left|u\right|\leq\left(TB/2\right)^{\vartheta}$ for some $\vartheta>0$.
It also follows that for large $T$ we have $TB>2\overline{u}_{2}$
and $\left(TB/2\right)^{\vartheta}>\overline{u}_{2}$. Let $\Gamma_{T,B}\triangleq\left\{ u:\,\left(\left|u\right|>TB/2\right)\cap\mathbf{U}_{T}\right\} $.
Then, whenever $\left|s\right|>TB$ and $\left|u\right|\leq TB/2$,
we have, 
\begin{align}
\left|u-s\right|>TB/2>\overline{u}_{2} & \qquad\textrm{and}\qquad\inf_{u\in\Gamma_{T,B}}l\left(u\right)\geq\overline{l}_{2,T}.\label{eq. (5.13) I=000026H LapBai97-1-1}
\end{align}
 With this notation,
\begin{align*}
\inf_{\left|s\right|>TB}\mathcal{S}_{T}\left(s\right) & \geq\inf_{u\in\Gamma_{T,B}}l_{T}\left(u\right)\int_{\left(\left|w\right|\leq TB/2\right)\cap\mathbf{U}_{T}}p_{T}\left(w\right)dw\\
 & \geq\overline{l}_{2,T}\int_{\left(\left|w\right|\leq TB/2\right)\cap\mathbf{U}_{T}}p_{T}\left(w\right)dw,
\end{align*}
 from which it follows that 
\begin{align*}
\mathcal{S}_{T}\left(0\right)-\inf_{\left|s\right|>TB}\mathcal{S}_{T}\left(s\right) & \leq-\varpi\int_{\Gamma_{1,T}}p_{T}\left(u\right)du\\
 & \quad+\int_{\mathbf{U}_{T}\cap\left(\left(TB/2\right)^{\vartheta}\geq\left|u\right|\geq\overline{u}_{1}\right)}\left(l\left(u\right)-\inf_{\left|s\right|>TB/2}l_{T}\left(s\right)\right)p_{T}\left(u\right)du\\
 & \quad+\int_{\mathbf{U}_{T}\cap\left(\left|u\right|>\left(TB/2\right)^{\vartheta}\right)}l\left(u\right)p_{T}\left(u\right)du,
\end{align*}
where $\varpi\triangleq\overline{l}_{2,T}-\overline{l}_{1,T}$. 
The last inequality can be manipulated further using \eqref{eq. (5.13) I=000026H LapBai97-1-1},
\begin{align}
\mathcal{S}_{T}\left(0\right)-\inf_{\left|s\right|>TB}\mathcal{S}_{T}\left(s\right) & \leq-\varpi\int_{\Gamma_{1,T}}p_{T}\left(u\right)du\label{eq. (5.14) I=000026H LapBai97-1-1}\\
 & \quad+\int_{\mathbf{U}_{T}\cap\left(\left|u\right|>\left(TB/2\right)^{\vartheta}\right)}l_{T}\left(u\right)p_{T}\left(u\right)du.\nonumber 
\end{align}
Since $l\in\boldsymbol{L}$, we have $l\left(u\right)\leq\left|u\right|^{a},\,a>0$
when $u$ is large enough. Thus, given \eqref{eq" S_bar goe to - inf},
the second term of \eqref{eq. (5.14) I=000026H LapBai97-1-1} converges
to zero. Since $\int_{\Gamma_{1,T}}p_{T}\left(u\right)du>0$ the
first term of \eqref{eq. (5.14) I=000026H LapBai97-1-1} is negative
which then leads to $\mathcal{S}_{T}\left(0\right)-\inf_{\left|s\right|>TB}\mathcal{S}_{T}\left(s\right)<0$
or $\mathcal{S}_{T}\left(0\right)<\inf_{\left|s\right|>TB}\mathcal{S}_{T}\left(s\right).$
Thus, we have \eqref{eq. AA5 LapBai97-1-1}. 
\end{proof}
\begin{lem}
\label{Lemma Consistency Small Delta}Under Assumptions \ref{Assumption A Bai97}-\ref{Assumption A4 Bai97},
\ref{Assumption The-loss-function LapBai97}-\ref{Assumption Small Shift BP}
and \ref{Assumption Gaussian Process for Lap LapBai97}-(i), for $l\in\boldsymbol{L}$
and any $B>0$ and $\varepsilon>0$, we have for all large $T$, $\mathbb{P}\left[\left|\widehat{\lambda}_{b}^{\mathrm{GL}}-\lambda_{b}^{0}\right|>B\right]<\varepsilon$. 
\end{lem}
\begin{proof}
The structure of the proof is similar to that of Lemma \ref{Lemma Consistency Fixed delta}.
By Proposition 1 in \citet{bai:97RES}, eq. \eqref{eq" S_bar goe to - inf}
holds with $O_{\mathbb{P}}\left(T\left\Vert \delta_{T}\right\Vert ^{2}\right)$
in place of $O_{\mathbb{P}}\left(TB\right),\,B>0.$ One can then follow
the same steps as in the previous lemma to yield the result.
\end{proof}
\begin{lem}
\label{Lemma Rate of Convergence Small Delta}Under Assumptions \ref{Assumption A Bai97}-\ref{Assumption A4 Bai97},
\ref{Assumption The-loss-function LapBai97}-\ref{Assumption Small Shift BP}
and \ref{Assumption Gaussian Process for Lap LapBai97}-(i), for $l\in\boldsymbol{L}$
and for every $\varepsilon>0$ there exists a $B<\infty$ such that
for all large $T$, $\mathbb{P}\left[Tv_{T}^{2}\left|\widehat{\lambda}_{b}^{\mathrm{GL}}-\lambda_{b}^{0}\right|>B\right]<\varepsilon$.
\end{lem}
\begin{proof}
See Lemma \ref{Theorem 5.2 I=000026H LapBai97} which proves a stronger
result needed for Theorem \ref{Theorem Geneal Laplace Estimator LapBai97}.
\end{proof}
Parts (i) and (ii) of Proposition \ref{Proposition: Consistency and Rate of Convergence}
follow from Lemma \ref{Lemma Consistency Small Delta} and Lemma \ref{Lemma Rate of Convergence Small Delta},
respectively. 

\subsubsection{Proof of Theorem \ref{Theorem Posterior Mean Lap Estimation in Bai 97}}

We start with the following lemmas.
\begin{lem}
\label{Lemma B.3, LapBai97}For any $a\in\mathbb{R}$, $\left|c\right|\leq1,$
and integer $i\geq0$, $\left|\exp\left(ca\right)-\sum_{j=0}^{i}\left(ca\right)^{j}/j!\right|\leq\left|c\right|^{i+1}\exp\left(\left|a\right|\right).$ 
\end{lem}
\begin{proof}
The proof is immediate and the same as the one in \citet{jun/pinkse/wan:15}.
Using simple manipulations,
\begin{align*}
\left|\exp\left(ca\right)-\sum_{j=0}^{i}\left(ca\right)^{j}/j!\right| & \leq\left|\sum_{j=i+1}^{\infty}\frac{\left(ca\right)^{j}}{j!}\right|\leq\left|c\right|^{i+1}\left|\sum_{j=i+1}^{\infty}\frac{\left(a\right)^{j}}{j!}\right|\leq\left|c\right|^{i+1}\exp\left(\left|a\right|\right).
\end{align*}
\end{proof}

\begin{lem}
\label{Lemma B.2, LapBai97}$\widetilde{G}_{T,v}\left(u,\,\widetilde{v}\right)\Rightarrow\mathscr{W}\left(u\right)$
in $\mathbb{D}_{b}\left(\mathbf{C}\times\mathbf{V}\right)$, where
$\mathbf{C}\subset\mathbb{R}$ and $\mathbf{V}\subset\mathbb{R}^{p+2q}$
are both compact sets, and 
\begin{align*}
\mathscr{W}\left(u\right) & \triangleq\begin{cases}
2\left(\left(\delta^{0}\right)'\Sigma_{1}\delta^{0}\right)^{1/2}W_{1}\left(-u\right), & \textrm{if }u<0\\
2\left(\left(\delta^{0}\right)'\Sigma_{2}\delta^{0}\right)^{1/2}W_{2}\left(u\right), & \textrm{if }u\geq0.
\end{cases}
\end{align*}
\end{lem}
\begin{proof}
Consider $u<0.$ According to the expansion of the criterion function
given in Lemma \ref{Lemma, Area (ia), Bai A.5}, for any $\left(u,\,\widetilde{v}\right)\in\mathbf{C}\times\mathbf{V},$
$\widetilde{G}_{T,v}\left(u,\,\widetilde{v}\right)$ satisfies $2\mathrm{sgn}\left(T_{b}^{0}-T_{b}\left(u\right)\right)\delta'_{T}Z'_{\Delta}e+o_{\mathbb{P}}\left(1\right).$
Then, $\delta'_{T}Z'_{\Delta}e=\left(\delta^{0}\right)'v_{T}\sum_{t=\left\lfloor u/v_{T}^{2}\right\rfloor }^{T_{b}^{0}}z_{t}e_{t}\Rightarrow\left(\delta^{0}\right)'\mathscr{G}_{1}\left(-u\right),$
where $\mathscr{G}_{1}$ is a multivariate Gaussian process. In particular,
$\left(\delta^{0}\right)'\mathscr{G}_{1}\left(-u\right)$ is equivalent
in law to $\left(\left(\delta^{0}\right)'\Sigma_{1}\delta^{0}\right)^{1/2}W_{1}\left(-u\right)$,
where $W_{1}\left(\cdot\right)$ is a standard Wiener process on $[0,\,\infty).$
Similarly, for $u\geq0,$ $\delta'_{T}Z'_{\Delta}e\Rightarrow\left(\left(\delta^{0}\right)'\Sigma_{2}\delta^{0}\right)^{1/2}W_{2}\left(u\right)$,
where $W_{2}\left(\cdot\right)$ is another standard Wiener process
on $[0,\,\infty)$ which is independent of $W_{1}.$ Hence, $\widetilde{G}_{T,v}\left(u,\,\widetilde{v}\right)\Rightarrow\mathscr{W}\left(u\right)$
in $\mathbb{D}_{b}\left(\mathbf{C}\times\mathbf{V}\right)$.
\end{proof}
\begin{lem}
\label{Lemma B.5, LapBai97}Fix any $a>0$ and let $\varpi\in(1/2,\,1]$.
(i) For any $\nu>0$ and any $\varepsilon>0,$ 
\begin{align*}
\limsup_{T\rightarrow\infty}\mathbb{P}\left[\sup_{u\in\Gamma_{T,\psi}^{c}}\left\{ \overline{G}_{T,v}\left(u\right)-a\left\Vert \delta^{0}\right\Vert ^{2}\left|u\right|^{\varpi}\right\} >\nu\right]<\varepsilon & .
\end{align*}
(ii) For $\widetilde{u}\in\mathbb{R}_{+}$ let $\widetilde{\Gamma}\triangleq\left\{ u\in\mathbb{R}:\,\left|u\right|>\widetilde{u}\right\} $.
Then, for every $\epsilon>0,$ 
\begin{align*}
\lim_{\widetilde{u}\rightarrow\infty}\lim_{T\rightarrow\infty}\mathbb{P}\left[\sup_{u\in\widetilde{\Gamma}}\left\{ \overline{G}_{T,v}\left(u\right)-a\left\Vert \delta^{0}\right\Vert ^{2}\left|u\right|^{\varpi}\right\} >\epsilon\right]=0 & .
\end{align*}
\end{lem}
\begin{proof}
We begin with part (i). Upon using Lemma \ref{Lemma B.2, LapBai97}
and the continuous mapping theorem, with any nonnegative integer $i,$
\begin{align*}
\limsup_{T\rightarrow\infty}\mathbb{P}\left[\sup_{u\in\Gamma_{T,\psi}^{c}}\left\{ \overline{G}_{T,v}\left(u\right)-a\left\Vert \delta^{0}\right\Vert ^{2}\left|u\right|^{\varpi}\right\} >\nu\right] & \leq\lim_{T\rightarrow\infty}\mathbb{P}\left[\sup_{\left|u\right|>\overline{u}}\left\{ \overline{G}_{T,v}\left(u\right)-a\left\Vert \delta^{0}\right\Vert \left|u\right|^{\varpi}\right\} >\nu\right]\\
 & \leq\lim_{T\rightarrow\infty}\mathbb{P}\left[\sup_{\left|u\right|\geq i}\left\{ \overline{G}_{T,v}\left(u\right)>a\left\Vert \delta^{0}\right\Vert \left|u\right|^{\varpi}\right\} >\nu\right]\\
 & \leq\mathbb{P}\left[\sup_{\left|u\right|\geq i}\left\{ \left|\mathscr{W}\left(u\right)\right|-a\left\Vert \delta^{0}\right\Vert \left|u\right|^{\varpi}\right\} >\nu\right]\\
 & \leq\sum_{r=i+1}^{\infty}\mathbb{P}\left[\sup_{r-1\leq\left|u\right|<r}\left\{ \left|\mathscr{W}\left(u\right)\right|-a\left\Vert \delta^{0}\right\Vert \left|u\right|^{\varpi}\right\} >\nu\right].
\end{align*}
 Then, 
\begin{align}
\sum_{r=i+1}^{\infty} & \mathbb{P}\left[\sup_{r-1\leq\left|u\right|<r}\frac{1}{\sqrt{r}}\left|\mathscr{W}\left(u\right)\right|>\inf_{r-1<\left|u\right|<r}a\frac{1}{\sqrt{r}}\left\Vert \delta^{0}\right\Vert \left|u\right|^{\varpi}\right]\nonumber \\
 & =\sum_{r=i+1}^{\infty}\mathbb{P}\left[\sup_{1-1/r\leq\left|u\right|/r\leq1}\left|\mathscr{W}\left(u/r\right)\right|>\inf_{1-1/r<\left|u\right|/r\leq1}a\left(\frac{r}{r}\right)^{\varpi-1/2}\frac{\left|u\right|^{\varpi}}{\sqrt{r}}\left\Vert \delta^{0}\right\Vert \right]\nonumber \\
 & =\sum_{r=i+1}^{\infty}\mathbb{P}\left[\sup_{1-1/r<s\leq1}\left|\mathscr{W}\left(s\right)\right|>\inf_{c<s\leq1}ar^{\varpi-1/2}s^{\varpi}\left\Vert \delta^{0}\right\Vert \right]\nonumber \\
 & =\sum_{r=i+1}^{\infty}\mathbb{P}\left[\sup_{s\leq1}\left|\mathscr{W}\left(s\right)\right|>r^{\varpi-1/2}c^{\varpi}C\left\Vert \delta^{0}\right\Vert \right],\label{eq (27b)}
\end{align}
where $0<c\leq1$. By Markov's inequality, 
\begin{align}
\sum_{r=i+1}^{\infty}\mathbb{P}\left[\sup_{c<s\leq1}\left|\mathscr{W}\left(s\right)\right|^{4}>C^{4}\left\Vert \delta^{0}\right\Vert ^{4}r^{4\left(\varpi-1/2\right)}c^{4\varpi}\right] & \leq\frac{C}{\left\Vert \delta^{0}\right\Vert ^{4}}\frac{\mathbb{E}\left(\sup_{s\leq1}\left|\mathscr{W}\left(s\right)\right|^{4}\right)}{c^{4\varpi}}\sum_{r=i+1}^{\infty}r^{-\left(4\varpi-2\right)}.\label{Eq. (27)-1}
\end{align}
By Proposition A.2.4 in \citet{vaart/wellner:96}, $\mathbb{E}(\sup_{s\leq1}\left|\mathscr{W}\left(s\right)\right|^{4})\leq C\mathbb{E}\left(\sup_{s\leq1}\left|\mathscr{W}\left(s\right)\right|\right)^{4}$
for some $C<\infty$, which is finite by Corollary 2.2.8 in \citet{vaart/wellner:96}.
Choose $K$ (thus $\overline{u}$) large enough such that the right-hand
side in \eqref{Eq. (27)-1} can be made arbitrarily smaller than $\varepsilon>0.$
The proof of the second part is similar and omitted. 
\end{proof}
\begin{lem}
\label{Lemma B.6 LapBai97}Fix any $a>0$. For any $\varepsilon>0$
there exists a $C<\infty$ such that
\begin{align*}
\mathbb{P}\left[\sup_{u\in\mathbb{R}}\left\{ \overline{G}_{T,v}\left(u\right)-a\left\Vert \delta^{0}\right\Vert ^{2}\left|u\right|\right\} >C\right] & <\varepsilon,\qquad\mathrm{for\,all\,}T.
\end{align*}
\end{lem}
\begin{proof}
For any finite $T,$ $\overline{G}_{T,v}\left(u\right)\in\mathbb{D}_{b}$
by definition. As for the limiting case, fix any $0<\overline{u}<\infty,$
\begin{align*}
\limsup_{T\rightarrow\infty}\,\mathbb{P}\left[\sup_{u\in\mathbb{R}}\left\{ \overline{G}_{T,v}\left(u\right)-a\left\Vert \delta^{0}\right\Vert ^{2}\left|u\right|\right\} >C\right] & \leq\limsup_{T\rightarrow\infty}\mathbb{P}\left[\sup_{\left|u\right|\leq\overline{u}}\overline{G}_{T,v}\left(u\right)>C\right]\\
 & \quad+\limsup_{T\rightarrow\infty}\mathbb{P}\left[\sup_{\left|u\right|>\overline{u}}\overline{G}_{T,v}\left(u\right)>a\left\Vert \delta^{0}\right\Vert ^{2}\overline{u}\right].
\end{align*}
The second term converges to zero letting $\overline{u}\rightarrow\infty$
from Lemma \ref{Lemma B.5, LapBai97}-(ii). For the first term, let
$C\rightarrow\infty$, use the continuous mapping theorem and Lemma
\ref{Lemma B.2, LapBai97} to deduce that it converges to zero by
the properties of $\mathscr{W}\in\mathbb{D}_{b}.$ 
\end{proof}
\begin{lem}
\label{Lemma before B.8 LapBai97}Let
\begin{align}
A_{1}\left(u,\,\widetilde{v}\right) & =u^{m}\pi_{T,v}\left(u\right)\exp\left(\widetilde{\gamma}_{T}\widetilde{G}_{T,v}\left(u,\,\widetilde{v}\right)+Q_{T,v}\left(u\right)\right),\label{Def A1 and A2}\\
A_{2}\left(u,\,\widetilde{v}\right) & =u^{m}\pi^{0}\exp\left(\widetilde{\gamma}_{T}\widetilde{G}_{T,v}\left(u,\,\widetilde{v}\right)-\Lambda_{0}\left(u\right)\right).\nonumber 
\end{align}
For $m\geq0$,  
\begin{align*}
\liminf_{T\rightarrow\infty}\mathbb{P}\left[\sup_{\widetilde{v}\in\mathbf{V}}\left|\int_{\Gamma_{T,\psi}^{c}}\left(A_{1}\left(u,\,\widetilde{v}\right)-A_{2}\left(u,\,\widetilde{v}\right)\right)\right|<\epsilon\right]\geq & 1-\epsilon.
\end{align*}
\end{lem}
\begin{proof}
We consider each integrand $A_{i}\left(u,\,\widetilde{v}\right)$
$\left(i=1,\,2\right)$ separately on $\Gamma_{T,\psi}^{c}$. Let
us consider $A_{1}$ first. Lemma \ref{Lemma, Area (ib) Q} yields
that whenever $\widetilde{\gamma}_{T}\rightarrow\kappa_{\gamma}<\infty,$
$A_{1}\left(u,\,\widetilde{v}\right)\leq C_{1}\exp\left(-C_{2}\nu_{T}\right)$
where $0<C_{1},\,C_{2}<\infty$ and $\nu_{T}$ is a divergent sequence.
Note that the number $C_{1}$ follows from Assumption \ref{Assumption Prior LapBai97}
(cf. $\pi\left(\cdot\right)<\infty$). The argument for $A_{2}\left(u,\,\widetilde{v}\right)$
relies on Lemma \ref{Lemma B.5, LapBai97}-(i), which shows that $G_{T,v}\left(u,\,\widetilde{v}\right)$
is always less than $C\left|u\right|^{\varpi}$ uniformly on $\Gamma_{T,\psi}^{c}$,
with $C>0$ and $\varpi\in\left(1/2,\,1\right)$. Thus, $A_{2}\left(u,\,\widetilde{v}\right)=o_{\mathbb{P}}\left(1\right)$
uniformly on $\mathbf{V}$. 
\end{proof}
Let $\Gamma_{T,K}\triangleq\left\{ u\in\mathbb{R}:\,\left|u\right|<K,\,K>0\right\} ,$
and $\Gamma_{T,\eta}\triangleq\left\{ u\in\mathbb{R}:\,K\leq\left|u\right|\leq\eta\psi_{T},\,K,\eta>0\right\} .$
\begin{lem}
\label{Lemma B.8 LapBai97}For any polynomial function $p\in\boldsymbol{P}$
and any $C<\infty$, let 
\begin{align*}
D_{T} & \triangleq\sup_{\widetilde{v}\in\mathbf{W}}\int_{\Gamma_{T,K}}\left|p\left(u\right)\right|\exp\left\{ C\widetilde{G}_{T,v}\left(u,\,\widetilde{v}\right)\right\} \left|\exp\left(R_{T,v}\left(u\right)\right)-1\right|\exp\left(-\varLambda^{0}\left(u\right)\right)du=o_{\mathbb{P}}\left(1\right).
\end{align*}
\end{lem}
\begin{proof}
Let $0<\epsilon<1.$ We shall use Lemma \ref{Lemma B.3, LapBai97}
with $i=0$, $a=R_{T,v}\left(u\right)/c$, and $c=\epsilon$ to deduce
that $D_{T}=O_{\mathbb{P}}\left(\epsilon\right)$ and then let $\epsilon\rightarrow0.$
Note that
\begin{align*}
\epsilon^{-1}D_{T}\leq C\int_{\Gamma_{T,K}}\left|p\left(u\right)\right|\exp\left(C\overline{G}_{T,v}\left(u,\,\widetilde{v}\right)+\left|\epsilon^{-1}R_{T,v}\left(u\right)\right|-\varLambda^{0}\left(u\right)\right)du & .
\end{align*}
By definition, $K\geq u=\left\Vert \delta_{T}\right\Vert ^{2}\left(T_{b}-T_{b}^{0}\right)$
on $\Gamma_{T,K}.$ By Lemma \ref{Lemma, Area (ia), Bai A.5}-\ref{Lemma ZZ conv prob},
on $\Gamma_{T,K}$ we have $R_{T,v}\left(u\right)=O_{\mathbb{P}}\left(\left\Vert \delta_{T}\right\Vert ^{2}\right)$
for each $u.$ Thus, for large enough $T$, the right-hand side above
is $O_{\mathbb{P}}\left(1\right)$ and does not depend on $\epsilon.$
Thus, $D_{T}=\epsilon O_{\mathbb{P}}\left(1\right).$ The claim of
the lemma follows by letting $\epsilon$ approach zero. 
\end{proof}
\begin{lem}
\label{Lemma B.9 LapBai97}For $p\in\boldsymbol{P},$ 
\begin{align*}
D_{2,T} & \triangleq\sup_{\widetilde{v}\in\mathbf{V}}\int_{\Gamma_{T,\eta}}\left|p\left(u\right)\right|\exp\left\{ \widetilde{\gamma}_{T}\widetilde{G}_{T,v}\left(u,\,\widetilde{v}\right)\right\} \exp\left(-\varLambda^{0}\left(u\right)\right)\left|\pi_{T,v}\left(u\right)-\pi^{0}\right|du=o_{\mathbb{P}}\left(1\right).
\end{align*}
\end{lem}
\begin{proof}
By the differentiability of $\pi\left(\cdot\right)$ at $\lambda_{b}^{0}$
(cf. Assumption \ref{Assumption Prior LapBai97}), for any $u\in\mathbb{R}$
$\left|\pi_{T,v}\left(u\right)-\pi^{0}\right|\leq\left|\pi\left(\lambda_{b,T}^{0}\left(v\right)\right)-\pi^{0}\right|+C\psi_{T}^{-1}\left|u\right|,$
with $C>0$. The first term on the right-hand side is $o\left(1\right)$
and does not depend on $u.$ Recalling that $\overline{G}_{T,v}\left(u,\,\widetilde{v}\right)=\sup_{\widetilde{v}\in\mathbf{V}}\left|\widetilde{G}_{T,v}\left(u,\,\widetilde{v}\right)\right|,$
\begin{align*}
D_{2,T} & \leq K\left[o\left(1\right)\int_{\Gamma_{T,\eta}}d_{T}\left(u\right)du+\psi_{T}^{-1}\int_{\Gamma_{T,\eta}}\left|u\right|d_{T}\left(u\right)du\right]\leq K\left[o\left(1\right)O_{\mathbb{P}}\left(1\right)+\psi_{T}^{-1}O_{\mathbb{P}}\left(1\right)\right],
\end{align*}
 where $d_{T}\left(u\right)\triangleq\left|p\left(u\right)\right|\exp\left\{ \widetilde{\gamma}_{T}\overline{G}_{T,v}\left(u,\,\widetilde{v}\right)\right\} \left|\exp\left(-\varLambda^{0}\left(u\right)\right)\right|$
and the $O_{\mathbb{P}}\left(1\right)$ terms follows from Lemma \ref{Lemma B.6 LapBai97}
and $\widetilde{\gamma}_{T}\rightarrow\kappa_{\gamma}<\infty$. Since
$\psi_{T}\rightarrow\infty,$ we have $D_{2,T}=o_{\mathbb{P}}\left(1\right)$.
\end{proof}
\begin{lem}
\label{Lemma B.10 LapBai97}For any $p\in\boldsymbol{P}$ and constants
$C_{1},\,C_{2}>0$, $\int_{\Gamma_{T,\psi}^{c}}\left|p\left(u\right)\right|\exp\left(C_{1}\overline{G}_{T}\left(u\right)-C_{2}\left|u\right|\right)du=o_{\mathbb{P}}\left(1\right).$ 
\end{lem}
\begin{proof}
It follows from Lemma \ref{Lemma B.5, LapBai97}.
\end{proof}
\begin{lem}
\label{Lemma B.11 LapBai97}For $p\in\boldsymbol{P}$ and constants
$a_{1},\,a_{2},\,a_{3}\geq0$, with $a_{2}+a_{3}>0$, let 
\begin{align*}
D_{3,T} & \triangleq\int_{\widetilde{\mathbf{U}}_{T}^{c}}\left|p\left(u\right)\right|\exp\left(\widetilde{\gamma}_{T}\left\{ a_{1}\overline{G}_{T,v}\left(u\right)+a_{2}Q_{T,v}\left(u\right)-a_{3}\varLambda^{0}\left(u\right)\right\} \right)du=o_{\mathbb{P}}\left(1\right).
\end{align*}
\end{lem}
\begin{proof}
It follows from Lemma \ref{Lemma, Area (ii), Q LapBai97}.
\end{proof}
\begin{lem}
\label{Lemma B.12 Posterior Mean}For any integer $m\geq0,$ 
\begin{align*}
\sup_{\widetilde{v}\in\mathbf{V}} & \left|\int_{\mathbb{R}}u^{m}\exp\left(\widetilde{\gamma}_{T}\widetilde{G}_{T,v}\left(u,\,\widetilde{v}\right)\right)\left[\pi_{T,v}\left(u\right)\exp\left(Q_{T,v}\left(u\right)\right)-\pi^{0}\exp\left(-\varLambda^{0}\left(u\right)\right)\right]du\right|\\
 & =\sup_{\widetilde{v}\in\mathbf{V}}\left|\int_{\mathbb{R}}\left(A_{1}\left(u,\,\widetilde{v}\right)-A_{2}\left(u,\,\widetilde{v}\right)\right)du\right|\\
 & =o_{\mathbb{P}}\left(1\right).
\end{align*}
\end{lem}
\begin{proof}
By Assumption \ref{Assumption Prior LapBai97}, $A_{1}\left(u,\,\widetilde{v}\right)=0$
for $u\in\Gamma_{T,\psi}^{c}-\widetilde{\mathbf{U}}_{T}^{c}$. Then,
omitting arguments, we can write, 
\begin{align}
\sup\left|\int_{\mathbb{R}}\left(A_{1}-A_{2}\right)\right| & \leq\sup\left|\int_{\Gamma_{T,\psi}}\left(A_{1}-A_{2}\right)\right|+\sup\left|\int_{\Gamma_{T,\psi}^{c}}A_{2}\right|+\sup\left|\int_{\widetilde{\mathbf{U}}_{T}^{c}}A_{1}\right|.\label{eq 31}
\end{align}
 The first right-hand side term above converges in probability to
zero by Lemma \ref{Lemma B.8 LapBai97}-\ref{Lemma B.9 LapBai97}.
The second and the last term are each $o_{\mathbb{P}}\left(1\right)$
by, receptively, Lemma \ref{Lemma B.10 LapBai97} and Lemma \ref{Lemma B.11 LapBai97}.
\end{proof}
We are now in a position to conclude the proof of Theorem \ref{Theorem Posterior Mean Lap Estimation in Bai 97}.
\begin{proof}
Let $\mathbf{V}\subset\mathbb{R}^{p+2q}$ be a compact set. From \eqref{eq. (17)-1},
\begin{align*}
\psi_{T}\left(\widehat{\lambda}_{b}^{\mathrm{GL},*}\left(\widetilde{v},\,v\right)-\lambda_{b,T}^{0}\left(v\right)\right) & =\frac{\int_{\mathbb{R}}u\exp\left(\widetilde{\gamma}_{T}\left[\widetilde{G}_{T,v}\left(u,\,\widetilde{v}\right)+Q_{T,v}\left(u\right)\right]\right)\pi_{T,v}\left(u\right)du}{\int_{\mathbb{R}}\exp\left(\widetilde{\gamma}_{T}\left[\widetilde{G}_{T,v}\left(u,\,\widetilde{v}\right)+Q_{T,v}\left(u\right)\right]\right)\pi_{T,v}\left(u\right)du}.
\end{align*}
 For a large enough $T$, by Lemma \ref{Lemma B.12 Posterior Mean}
the right-hand is uniformly in $\widetilde{v}\in\mathbf{V}$ equal
to 
\begin{align*}
\frac{\int_{\mathbb{R}}u\exp\left(\widetilde{\gamma}_{T}\widetilde{G}_{T,v}\left(u,\,\widetilde{v}\right)\right)\exp\left(-\varLambda^{0}\left(u\right)\right)du}{\int_{\mathbb{R}}\exp\left(\widetilde{\gamma}_{T}\widetilde{G}_{T,v}\left(u,\,\widetilde{v}\right)\right)\exp\left(-\varLambda^{0}\left(u\right)\right)du} & +o_{\mathbb{P}}\left(1\right).
\end{align*}
 The first term is integrable with large probability by Lemma \ref{Lemma B.5, LapBai97}-\ref{Lemma B.6 LapBai97}.
Thus, by Lemma \ref{Lemma B.2, LapBai97} and the continuous mapping
theorem, we have for each $v\in\mathbf{V}$,
\begin{align}
T\left\Vert \delta_{T}\right\Vert ^{2}\left(\widehat{\lambda}_{b}^{\mathrm{GL},*}\left(\widetilde{v},\,v\right)-\lambda_{b,T}^{0}\left(v\right)\right) & \Rightarrow\frac{\int_{\mathbb{R}}u\exp\left(\mathscr{W}\left(u\right)\right)\exp\left(-\varLambda^{0}\left(u\right)\right)du}{\int_{\mathbb{R}}\exp\left(\mathscr{W}\left(u\right)\right)\exp\left(-\varLambda^{0}\left(u\right)\right)du}.\label{eq. Finite-dimens conv Th. 3.1}
\end{align}
 Note that $\partial_{\theta}Q_{T}^{0}\left(\theta,\,\cdot\right)$
is monotonic and bounded for all $\theta\in\mathbf{S}$. The argument
of Theorem 4.1 in \citet{jureckova:77} can be used in \eqref{eq. Finite-dimens conv Th. 3.1}
to achieve uniformity in $v$.
\end{proof}

\subsubsection{\label{Subsection Proof-of-Theorem 1}Proof of Proposition \ref{Proposition}}

We first need to introduce further notation. For a scalar $\overline{u}>0$
define $\Gamma_{\overline{u}}\triangleq\left\{ u:\in\mathbb{R}:\,\left|u\right|\leq\overline{u}\right\} .$
Note that $\widetilde{\gamma}_{T}^{-1}=o\left(1\right)$.  We shall
be concerned with the asymptotic properties of the following statistic:
\begin{align*}
\xi_{T}\left(\widetilde{v}\right) & =\frac{\int_{\Gamma_{\overline{u}}}u\exp\left(\widetilde{\gamma}_{T}\left(\widetilde{G}_{T,v}\left(u,\,\widetilde{v}\right)+Q_{T,v}\left(u\right)\right)\right)\pi_{T,v}\left(u\right)du}{\int_{\Gamma_{\overline{u}}}\exp\left(\widetilde{\gamma}_{T}\left(\widetilde{G}_{T,v}\left(u,\,\widetilde{v}\right)+Q_{T,v}\left(u\right)\right)\right)\pi_{T,v}\left(u\right)du}.
\end{align*}
 Furthermore, for every $\widetilde{v}\in\mathbf{V}$, let $\xi_{0}\left(\widetilde{v}\right)=\arg\max_{u\in\Gamma_{\overline{u}}}\mathscr{V}\left(u\right)$.
It turns out that $\xi_{0}\left(\widetilde{v}\right)$ is flat in
$\widetilde{v}$ and thus we write $\xi_{0}=\xi_{0}\left(\widetilde{v}\right)$.
Finally, recall that $u=T\left\Vert \delta_{T}\right\Vert ^{2}\left(\lambda_{b}-\lambda_{b,T}^{0}\left(v\right)\right).$
\begin{lem}
\label{Lemma C.2 LapBai97}Let $\Gamma_{T,\overline{u}}^{c}=\mathbf{U}_{T}-\Gamma_{\overline{u}}$.
Then for any $\epsilon>0$ and $m=0,\,1,$
\begin{align*}
\lim_{\overline{u}\rightarrow\infty}\lim_{T\rightarrow\infty}\mathbb{P}\left(\frac{\sup_{\widetilde{v}\in\mathbf{V}}\int_{\Gamma_{T,\overline{u}}^{c}}\left|u\right|^{m}\exp\left(\widetilde{\gamma}_{T}\left(\widetilde{G}_{T,v}\left(u,\,\widetilde{v}\right)+Q_{T,v}\left(u\right)\right)\right)\pi_{T,v}\left(u\right)du}{\sup_{\widetilde{v}\in\mathbf{V}}\int_{\mathbb{R}}\exp\left(\widetilde{\gamma}_{T}\left(\widetilde{G}_{T,v}\left(u,\,\widetilde{v}\right)+Q_{T,v}\left(u\right)\right)\right)\pi_{T,v}\left(u\right)du}>\epsilon\right) & =0.
\end{align*}
\end{lem}
\begin{proof}
Let $J_{1}$ and $J_{2}$ denote the numerator and denominator, respectively,
in the display of the lemma. Then, 
\begin{align}
\mathbb{P}\left(J_{1}/J_{2}>\epsilon\right) & \leq\mathbb{P}\left(J_{2}\leq\exp\left(-\overline{a}\widetilde{\gamma}_{T}\right)\right)+\mathbb{P}\left(J_{1}>\epsilon\exp\left(-\overline{a}\widetilde{\gamma}_{T}\right)\right),\label{eq. (36) Lap Bai97}
\end{align}
for any constant $\overline{a}>0.$ Let us consider the second term
term in \eqref{eq. (36) Lap Bai97}. For an arbitrary $a>0$, let
$\mathbf{H}\left(\overline{u},\,a\right)=\left\{ u\in\Gamma_{T,\overline{u}}^{c}:\,\sup_{\widetilde{v}\in\mathbf{V}}\left|\widetilde{G}_{T,v}\left(u,\,\widetilde{v}\right)\right|\leq a\left|u\right|\right\} .$
Let $\overline{\lambda}=2\sup_{\lambda_{b}\in\varGamma^{0}}\left|\lambda_{b}\right|.$
Note that $\overline{\lambda}<2$ and $\sup_{u\in\mathbf{H}\left(\overline{u},\,a\right)}\left|u\right|\leq\overline{\lambda}T\left\Vert \delta_{T}\right\Vert ^{2}$.
By Assumptions \ref{Assumption A4 Bai97} and \ref{A.9a Bai 97},
and Lemma \ref{Lemma, Area (ii), Q LapBai97}, $Q_{T,v}\left(u\right)\leq-\min\left(\varLambda^{0}\left(u\right)/2,\,\eta\overline{\lambda}\left\Vert \delta_{T}\right\Vert ^{2}T\right)$
uniformly for all large $T$ where $\eta>0$. Thus, 
\begin{align}
\sup_{u\in\mathbf{H}\left(\overline{u},\,a\right)} & \sup_{\widetilde{v}\in\mathbf{V}}\exp\left(\widetilde{\gamma}_{T}\left[\widetilde{G}_{T,v}\left(u,\,\widetilde{v}\right)+Q_{T,v}\left(u\right)\right]\right)\label{eq. 37 LapBai97}\\
 & \leq\sup_{u\in\mathbf{H}\left(\overline{u},\,a\right)}\sup_{\widetilde{v}\in\mathbf{V}}\exp\left(\widetilde{\gamma}_{T}\left[a\left|u\right|-\varLambda^{0}\left(u\right)/4+\left[\varLambda^{0}\left(u\right)/2+Q_{T,v}\left(u\right)\right]\right]\right)\nonumber \\
 & \leq\sup_{u\in\mathbf{H}\left(\overline{u},\,a\right)}\exp\left(\widetilde{\gamma}_{T}\left[a\left|u\right|-\varLambda^{0}\left(u\right)-\min\left(\varLambda^{0}\left(u\right)/4,\,\varLambda^{0}\left(u\right)/4+\eta\left\Vert \delta_{T}\right\Vert ^{2}T\right)\right]\right)\nonumber \\
 & \leq\sup_{u\in\mathbf{H}\left(\overline{u},\,c\right)}\exp\left(\widetilde{\gamma}_{T}\left[a\left|u\right|-C_{2}\left|u\right|\right]\right)+\exp\left(\gamma_{T}\left[a\overline{\lambda}-\eta C\right]\right)\nonumber \\
 & \leq\sup_{u\in\mathbf{H}\left(\overline{u},\,c\right)}\exp\left(\gamma_{T}\left[a-C_{2}\right]\right)+\exp\left(\gamma_{T}\left[a\overline{\lambda}-\eta C\right]\right)=o\left(\exp\left(-\gamma_{T}\overline{a}_{1}\right)\right),\nonumber 
\end{align}
when $a>0$ is chosen sufficiently small and for some $\overline{a}_{1}>0.$
Furthermore, by Lemma \ref{Lemma B.5, LapBai97}-(ii) below with
$\varpi=1,$ 
\begin{align}
\lim_{\overline{u}\rightarrow\infty}\lim_{T\rightarrow\infty}\mathbb{P}\left(u\in\left\{ \Gamma_{T,\overline{u}}^{c}-\mathbf{H}\left(\overline{u},\,c\right)\right\} \right) & \leq\lim_{\overline{u}\rightarrow\infty}\lim_{T\rightarrow\infty}\mathbb{P}\left(\sup_{\left|u\right|>\overline{u}}\frac{\widetilde{G}_{T,v}\left(u,\,\widetilde{v}\right)}{\left|u\right|}>a\right)=0.\label{Eq. (38) Lap Bai97-1}
\end{align}
 By combining  \eqref{eq. 37 LapBai97}-\eqref{Eq. (38) Lap Bai97-1},
$\mathbb{P}\left(J_{1}>\epsilon\exp\left(-\overline{a}\widetilde{\gamma}_{T}\right)\right)\rightarrow0$
as $T\rightarrow\infty$. Next, we consider the first right-hand
side term in \eqref{eq. (36) Lap Bai97}. Recall the definition of
$\lambda_{+}$ from Assumption \ref{Assumption A.9b Bai 97, LapBai97}
and let $0<b\leq\overline{a}/4\lambda_{+}$. Note that for $G_{T,v}\left(b\right)\triangleq\sup_{\left|u\right|\leq b}\sup_{\widetilde{v}\in\mathbf{V}}\left|\widetilde{G}_{T,v}\left(u,\,\widetilde{v}\right)\right|,$
\begin{align}
\mathbb{P}\left(J_{2}\leq\exp\left(-\overline{a}\widetilde{\gamma}_{T}\right)\right) & \leq\mathbb{P}\left(G_{T,v}\left(b\right)\leq\overline{a},\,J_{2}\leq\exp\left(-\overline{a}\widetilde{\gamma}_{T}\right)\right)+\mathbb{P}\left(G_{T,v}\left(b\right)>\overline{a}\right).\label{eq. (39) Lap Bai97}
\end{align}
 Under Assumption \ref{Assumption Prior LapBai97} and the second
part of Assumption \ref{Assumption A.9b Bai 97, LapBai97}, using
the definition of $b,$ 
\begin{align*}
\mathbb{P}\left(G_{T,v}\left(b\right)\leq\overline{a},\,J_{2}\leq\exp\left(-\overline{a}\widetilde{\gamma}_{T}\right)\right) & \leq\mathbb{P}\left(C_{\pi}\int_{\left|u\right|\leq b}\exp\left(\widetilde{\gamma}_{T}\left(-\overline{a}/2-\lambda_{+}b\right)\right)du\leq\exp\left(-\overline{a}\widetilde{\gamma}_{T}\right)\right)\\
 & \leq\mathbb{P}\left(C_{\pi}b\exp\left(\overline{a}\widetilde{\gamma}_{T}/2\right)\leq1\right)\rightarrow0,
\end{align*}
as $T\rightarrow\infty$. We shall use the uniform convergence in
Lemma \ref{Lemma B.2, LapBai97} for the second right-hand side term
in \eqref{eq. (39) Lap Bai97} to deduce that (recall that $\overline{a}$
was chosen sufficiently small and $b\leq\overline{a}/4\lambda_{+}$),
\begin{align*}
\lim_{b\rightarrow0}\lim_{T\rightarrow\infty}\mathbb{P}\left(G_{T,v}\left(b\right)>\overline{a}\right) & \leq\lim_{b\rightarrow0}\mathbb{P}\left(\sup_{\left|u\right|\leq b}\left|\mathscr{\mathscr{W}}\left(u\right)\right|>\overline{a}\right)=0.
\end{align*}
\end{proof}
\begin{lem}
\label{Lemma C.1. LapBai97}As $T\rightarrow\infty$, $\xi_{T}\left(\widetilde{v}\right)\overset{}{\Rightarrow}\xi_{0}$
in $\mathbb{D}_{b}\left(\mathbf{V}\right).$ 
\end{lem}
\begin{proof}
Let $\mathbf{B}=\Gamma_{\overline{u}}\times\mathbf{V}$. For any
fixed $\overline{u},$ Lemma \ref{Lemma B.2, LapBai97} and the result
$\sup_{\left(u,\,\widetilde{v}\right)\in\mathbf{B}}\left|Q_{T,v}\left(u\right)-\varLambda^{0}\left(u\right)\right|=o_{\mathbb{P}}\left(1\right)$
(cf. Lemma \ref{Lemma ZZ conv prob}), imply that $\overline{Q}_{T}\Rightarrow\mathscr{V}$
in $\mathbb{D}_{b}\left(\mathbf{B}\right)$. By the Skorokhod representation
theorem {[}cf. Theorem 6.4 in \citet{billingsley:99}{]} we can find
a probability space $\left(\widetilde{\Omega},\,\widetilde{\mathscr{F}},\,\widetilde{\mathbb{P}}\right)$
on which there exist processes $\widetilde{Q}_{T}\left(u,\,\widetilde{v}\right)$
and $\widetilde{\mathscr{V}}\left(u\right)$ which have the same law
as $\overline{Q}_{T}\left(u,\,\widetilde{v}\right)$ and $\mathscr{V}\left(u\right)$,
respectively, and with the property that 
\begin{align}
\sup_{\left(u,\,\widetilde{v}\right)\in\mathbf{B}}\left|\widetilde{Q}_{T}\left(u,\,\widetilde{v}\right)-\widetilde{\mathscr{V}}\left(u\right)\right| & \rightarrow0\qquad\widetilde{\mathbb{P}}-\textrm{a.s.}\label{eq (32), Conv Sko}
\end{align}
Let
\begin{align*}
\widetilde{\xi}_{T}\left(\widetilde{v}\right) & \triangleq\frac{\int_{\Gamma_{\overline{u}}}u\exp\left(\widetilde{\gamma}_{T}\widetilde{Q}_{T,v}\left(u,\,\widetilde{v}\right)\right)\pi_{T,v}\left(u\right)du}{\int_{\Gamma_{\overline{u}}}\exp\left(\widetilde{\gamma}_{T}\widetilde{Q}_{T,v}\left(u,\,\widetilde{v}\right)\right)\pi_{T,v}\left(u\right)du},
\end{align*}
 and $\widetilde{\xi}_{0}\triangleq\arg\max_{u\in\Gamma_{\overline{u}}}\widetilde{\mathscr{V}}\left(u\right)$.
We shall rely on \eqref{eq (32), Conv Sko} to establish that
\begin{align}
\sup_{\widetilde{v}\in\mathbf{V}}\left|\widetilde{\xi}_{T}\left(\widetilde{v}\right)-\widetilde{\xi}_{0}\right| & \rightarrow0\qquad\widetilde{\mathbb{P}}-\textrm{a.s.}.\label{eq. (33) Bai97}
\end{align}
Let us indicate any pair of sample paths of $\widetilde{Q}_{T}\left(u,\,\widetilde{v}\right)$
and $\widetilde{\mathscr{V}}$, for which \eqref{eq (32), Conv Sko}
holds with a superscript $\omega$, by $\widetilde{Q}_{T,v}^{\omega}$
and $\widetilde{\mathscr{V}}^{\omega}$, respectively. For arbitrary
sets $\mathbf{S}_{1},\,\mathbf{S}_{2}\subset\mathbf{B},$ let $\widetilde{\rho}\left(\mathbf{S}_{1},\,\mathbf{S}_{2}\right)\triangleq\mathrm{Leb}\left(\mathbf{S}_{1}-\mathbf{S}_{2}\right)+\mathrm{Leb}\left(\mathbf{S}_{2}-\mathbf{S}_{1}\right)$
where $\mathrm{Leb}\left(\mathbf{A}\right)$ is the Lebesgue measure
of the set $\mathbf{A}$. Further, for an arbitrary scalar $c>0$
and function $\varUpsilon:\,\mathbf{B}\rightarrow\mathbb{R},$ define
$\mathbf{S}\left(\varUpsilon,\,c\right)\triangleq\left\{ \left(u,\,\widetilde{v}\right)\in\mathbf{B}:\,\left|\varUpsilon\left(u,\,\widetilde{v}\right)-\widetilde{\mathscr{V}}_{\mathrm{M}}\right|\leq c\right\} $
where $\widetilde{\mathscr{V}}_{\mathrm{M}}\triangleq\max_{u\in\Gamma_{\overline{u}}}\widetilde{\mathscr{V}}^{\omega}\left(u\right)$.
The first step is to show that
\begin{align}
\widetilde{\rho}\left(\mathbf{S}\left(\widetilde{Q}_{T,v}^{\omega},\,c\right),\,\mathbf{S}\left(\widetilde{\mathscr{V}}^{\omega},\,c\right)\right) & =o\left(1\right).\label{eq. (34)}
\end{align}
Let $\mathbf{S}_{1,T}\left(c\right)=\mathbf{S}\left(\widetilde{Q}_{T,v}^{\omega},\,c\right)-\mathbf{S}\left(\widetilde{\mathscr{V}}^{\omega},\,c\right)$
and $\mathbf{S}_{2,T}\left(c\right)=\mathbf{S}\left(\widetilde{\mathscr{V}}^{\omega},\,c\right)-\mathbf{S}\left(\widetilde{Q}_{T,v}^{\omega},\,c\right)$.
We first establish that $\mathrm{Leb}\left(\mathbf{S}_{2,T}\left(c\right)\right)=o\left(1\right)$.
For an arbitrary $\overline{c}>0,$ define the set $\widetilde{\mathbf{S}}_{T}\left(\overline{c}\right)\triangleq\left\{ \left(u,\,\widetilde{v}\right)\in\mathbf{B}:\,\left|\widetilde{Q}_{T,v}^{\omega}\left(u,\,\widetilde{v}\right)-\widetilde{\mathscr{V}}^{\omega}\left(u\right)\right|\leq\overline{c}\right\} $
and its complement (relative to $\mathbf{B}$) $\widetilde{\mathbf{S}}_{T}^{c}\left(\overline{c}\right)\triangleq\left\{ \left(u,\,\widetilde{v}\right)\in\mathbf{B}:\,\left|\widetilde{Q}_{T,v}^{\omega}\left(u,\,\widetilde{v}\right)-\widetilde{\mathscr{V}}^{\omega}\left(u\right)\right|>\overline{c}\right\} $.
We have 
\begin{align*}
\mathrm{Leb}\left(\mathbf{S}_{2,T}\left(c\right)\right) & =\mathrm{Leb}\left(\mathbf{S}_{2,T}\left(c\right)\cap\widetilde{\mathbf{S}}_{T}\left(\overline{c}\right)\right)+\mathrm{Leb}\left(\mathbf{S}_{2,T}\left(c\right)\cap\widetilde{\mathbf{S}}_{T}^{c}\left(\overline{c}\right)\right)\\
 & \leq\mathrm{Leb}\left(\mathbf{S}_{2,T}\left(c\right)\cap\widetilde{\mathbf{S}}_{T}\left(\overline{c}\right)\right)+\mathrm{Leb}\left(\widetilde{\mathbf{S}}_{T}^{c}\left(\overline{c}\right)\right).
\end{align*}
Note that $\mathrm{Leb}\left(\widetilde{\mathbf{S}}_{T}^{c}\left(\overline{c}\right)\right)=o\left(1\right)$
since the path $\omega$ satisfies \eqref{eq (32), Conv Sko}. Furthermore,
$\mathbf{S}_{2,T}\left(c\right)\cap\widetilde{\mathbf{S}}_{T}\left(\overline{c}\right)\subset\mathbf{C}_{T}\left(c,\,\overline{c}\right)$
where $\mathbf{C}_{T}\left(c,\,\overline{c}\right)\triangleq\left\{ \left(u,\,\widetilde{v}\right)\in\mathbf{B}:\,c\leq\left|\widetilde{Q}_{T,v}^{\omega}\left(u,\,\widetilde{v}\right)-\widetilde{\mathscr{V}}_{\mathrm{M}}\right|\leq c+\overline{c}\right\} .$
In view of \eqref{eq (32), Conv Sko}, 
\begin{align*}
\lim_{\overline{c}\downarrow0}\lim_{T\rightarrow\infty}\mathrm{Leb}\left(\mathbf{C}_{T}\left(c,\,\overline{c}\right)\right) & =\lim_{\overline{c}\downarrow0}\mathrm{Leb}\left\{ \left(u,\,\widetilde{v}\right)\in\mathbf{B}:\,c\leq\left|\widetilde{\mathscr{V}}^{\omega}\left(u\right)-\widetilde{\mathscr{V}}_{\mathrm{M}}\right|\leq c+\overline{c}\right\} \\
 & =\mathrm{Leb}\left\{ \left(u,\,\widetilde{v}\right)\in\mathbf{B}:\,\left|\widetilde{\mathscr{V}}^{\omega}\left(u\right)-\widetilde{\mathscr{V}}_{\mathrm{M}}\right|=c\right\} =0,
\end{align*}
 by the path properties of $\widetilde{\mathscr{V}}^{\omega}$. Since
$\mathrm{Leb}\left(\mathbf{S}_{1,T}\left(c\right)\right)=o\left(1\right)$
can be proven in a similar fashion, \eqref{eq. (34)} holds. For
$m=0,\,1$, $C_{1}<\infty$ and by Assumption \ref{Assumption Prior LapBai97}
we know there exists some $C_{2}<\infty$ such that 
\begin{align*}
\sup_{\widetilde{v}\in\mathbf{V}}\int_{\mathbf{S}^{c}\left(\widetilde{Q}_{T,v}^{\omega}\left(u,\,\widetilde{v}\right),\,c\right)}\left|u\right|^{m}\exp\left(\widetilde{\gamma}_{T}\left(\widetilde{Q}_{T,v}^{\omega}\left(u,\,\widetilde{v}\right)-\widetilde{\mathscr{V}}_{\mathrm{M}}\right)\right)\pi_{T,v}\left(u\right)du & \leq C_{1}\exp\left(-c\widetilde{\gamma}_{T}\right)C_{2}\int_{\Gamma_{\overline{u}}}\left|u\right|^{m}du=o\left(1\right),
\end{align*}
since $\left\{ u\leq\overline{u}\right\} $ on $\Gamma_{\overline{u}}$
and recalling that $\widetilde{\gamma}_{T}\rightarrow\infty.$ This
gives an upper bound to the same function where $u$ replaces $\left|u\right|$.
Then, 
\begin{align*}
\sup_{\widetilde{v}\in\mathbf{V}}\frac{\int_{\Gamma_{\overline{u}}}u\exp\left(\widetilde{\gamma}_{T}\widetilde{Q}_{T,v}^{\omega}\left(u,\,\widetilde{v}\right)\right)\pi_{T,v}\left(u\right)du}{\int_{\Gamma_{\overline{u}}}\exp\left(\widetilde{\gamma}_{T}\widetilde{Q}_{T,v}^{\omega}\left(u,\,\widetilde{v}\right)\right)\pi_{T,v}\left(u\right)du} & \leq\mathrm{ess\,sup}\,\mathbf{S}\left(\widetilde{Q}_{T,v}^{\omega},\,c\right)+o\left(1\right).
\end{align*}
By \eqref{eq (32), Conv Sko} we deduce $\mathrm{ess\,sup}\,\mathbf{S}\left(\widetilde{Q}_{T,v}^{\omega},\,c\right)+o\left(1\right)=\mathrm{ess\,sup}\,\mathbf{S}\left(\widetilde{\mathscr{V}}^{\omega},\,c\right)+o\left(1\right)$.
The same argument yields 
\begin{align*}
\inf_{\widetilde{v}\in\mathbf{V}}\frac{\int_{\Gamma_{\overline{u}}}u\exp\left(\widetilde{\gamma}_{T}\widetilde{Q}_{T,v}^{\omega}\left(u,\,\widetilde{v}\right)\right)\pi_{T,v}\left(u\right)du}{\int_{\Gamma_{\overline{u}}}\exp\left(\widetilde{\gamma}_{T}\widetilde{Q}_{T,v}^{\omega}\left(u,\,\widetilde{v}\right)\right)\pi_{T,v}\left(u\right)du} & \geq\mathrm{ess\,inf}\,\mathbf{S}\left(\widetilde{\mathscr{V}}^{\omega},\,c\right)+o\left(1\right).
\end{align*}
 Since almost every path $\omega$ of the Gaussian process $\widetilde{\mathscr{V}}$
achieves its maximum at a unique point on compact sets {[}cf. \citet{bai:97RES}
and Lemma 2.6 in \citet{kim/pollard:90}{]}, we have 
\begin{align*}
\lim_{c\downarrow0}\mathrm{ess\,inf}\,\mathbf{S}\left(\widetilde{\mathscr{V}}^{\omega},\,c\right) & =\lim_{c\downarrow0}\mathrm{ess\,sup}\,\mathbf{S}\left(\widetilde{\mathscr{V}}^{\omega},\,c\right)=\arg\max_{u\in\Gamma_{\overline{u}}}\widetilde{\mathscr{V}}^{\omega}\left(u\right).
\end{align*}
 Hence, we have proved \eqref{eq. (33) Bai97} which by the dominated
convergence theorem then implies the weak convergence of $\widetilde{\xi}_{T}$
toward $\widetilde{\xi}_{0}.$ Since the law of $\widetilde{\xi}_{T}$
($\widetilde{\xi}_{0}$) under $\widetilde{\mathbb{P}}$ is the same
as the law of $\xi_{T}$ ($\xi_{0}$) under $\mathbb{P},$ the claim
of the Lemma follows.
\end{proof}
We are now in a position to conclude the proof of Proposition \ref{Proposition}.
For a set $\mathbf{T}\subset\mathbb{R}$ and $m=0,\,1$ we define
$J_{m}\left(\mathbf{T}\right)\triangleq\int_{\mathbf{T}}u^{m}\exp\left(\widetilde{\gamma}_{T}\left(\widetilde{G}_{T,v}\left(u,\,\widetilde{v}\right)+Q_{T,v}\left(u\right)\right)\right)\pi_{T,v}\left(u\right)du.$
Hence, with this notation equation \eqref{eq. (17)-1} can be rewritten
as $T\left\Vert \delta_{T}\right\Vert ^{2}\left(\widehat{\lambda}_{b}^{\mathrm{GL,}*}\left(\widetilde{v},\,v\right)-\lambda_{b,T}^{0}\left(v\right)\right)=J_{1}\left(\mathbb{R}\right)/J_{0}\left(\mathbb{R}\right).$
Applying simple manipulations, we obtain, 
\begin{align}
J_{1}\left(\mathbb{R}\right)/J_{0}\left(\mathbb{R}\right) & =\frac{J_{1}\left(\Gamma_{\overline{u}}\right)+J_{1}\left(\Gamma_{\overline{u},T}^{c}\right)}{J_{0}\left(\Gamma_{\overline{u}}\right)+J_{0}\left(\Gamma_{\overline{u},T}^{c}\right)}=\frac{J_{1}\left(\Gamma_{\overline{u}}\right)}{J_{0}\left(\Gamma_{\overline{u}}\right)}\left[1-\frac{J_{0}\left(\Gamma_{\overline{u},T}^{c}\right)}{J_{0}\left(\mathbb{R}\right)}\right]+\frac{J_{1}\left(\Gamma_{\overline{u},T}^{c}\right)}{J_{0}\left(\mathbb{R}\right)}.\label{eq. J1/J0}
\end{align}
By Lemma \ref{Lemma C.2 LapBai97}, $J_{m}\left(\Gamma_{\overline{u},T}^{c}\right)/J_{0}\left(\mathbb{R}\right)=o_{\mathbb{P}}\left(1\right)$
($m=0,\,1$) uniformly in $\widetilde{v}\in\mathbf{V}.$ By Lemma
\ref{Lemma C.1. LapBai97}, with $\xi_{T}\left(\widetilde{v}\right)=J_{1}\left(\Gamma_{\overline{u}}\right)/J_{0}\left(\Gamma_{\overline{u}}\right),$
the first right-hand side term in \eqref{eq. J1/J0} converges weakly
to $\arg\max_{u\in\mathbb{R}}\mathscr{V}\left(u\right)$ in $\mathbb{D}_{b}\left(\mathbf{V}\right)$.

\subsubsection{Proof of Corollary \ref{Corollary part (i) of Theorem LapBai97}}

The proof involves a simple change in variable. We refer to Proposition
3 in \citet{bai:97RES}.

\subsubsection{Proof of Theorem \ref{Theorem Geneal Laplace Estimator LapBai97}}

We begin by introducing some notation. Since $l\in\boldsymbol{L}$,
for all real numbers $B$ sufficiently large and $\vartheta$ sufficiently
small the following relationship holds
\begin{align}
\inf_{\left|u\right|>B}l\left(u\right)-\sup_{\left|u\right|\leq B^{\vartheta}}l\left(u\right) & \geq0.\label{eq. (10.8) I=000026H LapBai97}
\end{align}
 Let $\zeta_{T,v}\left(u,\,\widetilde{v}\right)=\exp\left(G_{T,v}\left(u,\,\widetilde{v}\right)-\varLambda^{0}\left(u\right)\right)$,
$\Gamma_{T}\triangleq\left\{ u\in\mathbb{R}:\,\lambda_{b}\in\varGamma^{0}\right\} $
and 
\begin{align*}
\Gamma_{M} & =\left\{ u\in\mathbb{R}:\,M\leq\left|u\right|<M+1\right\} \cap\Gamma_{T},
\end{align*}
 and define
\begin{align}
J_{1,M}\triangleq\int_{\Gamma_{M}}\zeta_{T,v}\left(u,\,\widetilde{v}\right)\pi_{T,v}\left(u\right)du, & \qquad\qquad J_{2}\triangleq\int_{\Gamma_{T}}\zeta_{T,v}\left(u,\,\widetilde{v}\right)\pi_{T,v}\left(u\right)du.\label{eq. J_1,M}
\end{align}
 In some steps in the proof we shall be working with elements of
the following families of functions. A function $f_{T}:\,\mathbb{R}\rightarrow\mathbb{R}$
is said to belong to the family $\boldsymbol{F}$ if it satisfies
the following properties: (1) For fixed $T,$ $f_{T}\left(x\right)$
increases monotocically to infinity with $x\in[0,\,\infty)$; (2)
For any $b<\infty,$ $x^{b}\exp\left(-f_{T}\left(x\right)\right)\rightarrow0$
as both $T$ and $x$ diverge to infinity.  
\begin{proof}
The random variable $T\left\Vert \delta_{T}\right\Vert ^{2}\left(\widehat{\lambda}_{b}^{\mathrm{GL}}-\lambda_{0}\right)=\widetilde{\tau}_{T}$
is a minimizer of the function 
\begin{align*}
\Psi_{l,T}\left(s\right) & =\int_{\Gamma_{T}}l\left(s-u\right)\frac{\exp\left(\widetilde{G}_{T,v}\left(u,\,\widetilde{v}\right)+Q_{T,v}\left(u\right)\right)\pi_{T,v}\left(u\right)}{\int_{\Gamma_{T}}\exp\left(\widetilde{G}_{T,v}\left(w,\,\widetilde{v}\right)+Q_{T,v}\left(u\right)\right)\pi_{T,v}\left(w\right)dw}du.
\end{align*}
 Observe that Lemma \ref{Lemma B.8 LapBai97}-\ref{Lemma B.12 Posterior Mean}
apply to any polynomial $p\in\boldsymbol{P}$; therefore, they are
still valid for $l\in\boldsymbol{L}$. We then have that the asymptotic
behavior of $\Psi_{l,T}\left(s\right)$ only matters when $u$ (and
thus $s$) varies on $\Gamma_{K}=\left\{ u\in\mathbb{R}:\,u\leq K\right\} $.
By Lemma \ref{Lemma I.5.2.(i) I=000026H LapBai97}-\ref{Lemma 7.5. I.5.2.(ii) },
for any $\vartheta>0$, there exists a $\overline{T}$ such that for
all $T>\overline{T},$
\begin{align}
\mathbb{E}\left[\int_{\Gamma_{K}}\frac{\exp\left(\widetilde{G}_{T,v}\left(u,\,\widetilde{v}\right)+Q_{T,v}\left(u\right)\right)}{\int_{\Gamma_{T}}\exp\left(\widetilde{G}_{T,v}\left(w,\,\widetilde{v}\right)+Q_{T,v}\left(w\right)\right)dw}du\right] & \leq\frac{c_{\vartheta}}{K^{\vartheta}}.\label{eq. AA.4 LapBai97}
\end{align}
Therefore, for all $T>\overline{T},$ 
\begin{align}
\Psi_{l,T}\left(s\right) & =\frac{\int_{\left|u\right|\leq K}l\left(s-u\right)\exp\left(\widetilde{G}_{T,v}\left(u,\,\widetilde{v}\right)+Q_{T,v}\left(u\right)\right)du}{\int_{\left|w\right|\leq K}\exp\left(\widetilde{G}_{T,v}\left(w,\,\widetilde{v}\right)+Q_{T,v}\left(w\right)\right)dw}+o_{\mathbb{P}}\left(1\right),\label{eq. (10.9) I=000026H LapBai97}
\end{align}
 where the $o_{\mathbb{P}}\left(1\right)$ term is uniform in $T>\overline{T}$
as $K$ increases to infinity. By Assumption \eqref{Assumption Prior LapBai97},
 $\left|\pi_{T,v}\left(u\right)-\pi^{0}\right|\leq\left|\pi\left(\lambda_{b,T}^{0}\left(v\right)\right)-\pi^{0}\right|+C\psi_{T}^{-1}\left|u\right|,$
with $C>0$. On $\left\{ \left|u\right|\leq K\right\} $, the first
term on the right-hand side is $o\left(1\right)$ and does not depend
on $u.$ The second term is negligible when $T$ is large.  Thus,
without loss of generality we set $\pi_{T,v}\left(u\right)=1$ for
all $u$ in what follows. 

Next, we show the convergence of the marginal distributions of the
estimate $\Psi_{l,T}\left(s\right)$ to the marginals of the random
function $\Psi_{l}\left(s\right)$, where the region of integration
in the definition of both the numerator and denominator of $\Psi_{l,T}\left(s\right)$
and $\Psi_{l}\left(s\right)$ is restricted to $\left\{ \left|u\right|\leq K\right\} $
only, in view of \eqref{eq. (10.9) I=000026H LapBai97}. For a finite
integer $n,$ choose arbitrary real numbers $a_{j}$ ($j=0,\ldots,\,n$)
and introduce the following estimate: 
\begin{align}
\sum_{j=1}^{n}a_{j}\int_{\left|u\right|\leq K}l\left(s_{j}-u\right)\zeta_{T,v}\left(u,\,\widetilde{v}\right)du+a_{0}\int_{\left|u\right|\leq K}l\left(s_{0}-u\right)\zeta_{T,v}\left(u,\,\widetilde{v}\right)du & .\label{eq. conv AA3 LapBai97}
\end{align}
By Lemma \ref{Lemma VII.Lemma 2.2 I=000026H LapBai97 } and \ref{Lemma: FD Convergence: Condition 2 in I=000026H p.107},
we can invoke Theorem I.A.22 in \citet{ibragimov/has:81} which gives
the convergence in distribution of the estimate in \eqref{eq. conv AA3 LapBai97}
towards the distribution of the following random variable:
\begin{align*}
\sum_{j=1}^{n}a_{j}\int_{\left|u\right|\leq K}l\left(s_{j}-u\right)\exp\left(\mathscr{V}\left(u\right)\right)du+a_{0}\int_{\left|u\right|\leq K}l\left(s_{0}-u\right)\exp\left(\mathscr{V}\left(u\right)\right)du & .
\end{align*}
 By the Cramer-Wold Theorem {[}cf. Theorem 29.4 in \citet{billingsley:95}{]}
this suffices for the convergence in distribution of the vector 
\begin{align*}
\int_{\left|u\right|\leq K}l\left(s_{i}-u\right)\zeta_{T,v}\left(u,\,\widetilde{v}\right)du,\ldots,\,\int_{\left|u\right|\leq K}l\left(s_{n}-u\right)\zeta_{T,v}\left(u,\,\widetilde{v}\right)du, & \qquad\int_{\left|u\right|\leq K}l\left(s_{0}-u\right)\zeta_{T,v}\left(u,\,\widetilde{v}\right)du,
\end{align*}
 to the distribution of the vector
\begin{align*}
\int_{\left|u\right|\leq K}l\left(s_{i}-u\right)\exp\left(\mathscr{V}\left(u\right)\right)du,\ldots,\,\int_{\left|u\right|\leq K}l\left(s_{n}-u\right)\exp\left(\mathscr{V}\left(u\right)\right)du, & \qquad\int_{\left|u\right|\leq K}l\left(s_{0}-u\right)\exp\left(\mathscr{V}\left(u\right)\right)du.
\end{align*}
As a consequence, for any $K_{1},\,K_{2}<\infty$, the marginal distributions
of
\begin{align*}
\frac{\int_{\left|u\right|\leq K_{1}}l\left(s-u\right)\exp\left(\widetilde{G}_{T,v}\left(u,\,\widetilde{v}\right)+Q_{T,v}\left(u\right)\right)du}{\int_{\left|w\right|\leq K_{2}}\exp\left(\widetilde{G}_{T,v}\left(w,\,\widetilde{v}\right)+Q_{T,v}\left(w\right)\right)dw} & ,
\end{align*}
 converge to the marginals of $\int_{\left|u\right|\leq K_{1}}l\left(s-u\right)\exp\left(\mathscr{V}\left(u\right)\right)du/\left(\int_{\left|w\right|\leq K_{2}}\exp\left(\mathscr{V}\left(w\right)\right)dw\right)$.
The same convergence result extends to the distribution of
\begin{align*}
\int_{M\leq\left|u\right|<M+1}\frac{\exp\left(\widetilde{G}_{T,v}\left(u,\,\widetilde{v}\right)+Q_{T,v}\left(u\right)\right)}{\int_{\left|w\right|\leq K_{2}}\exp\left(\widetilde{G}_{T,v}\left(w,\,\widetilde{v}\right)+Q_{T,v}\left(w\right)\right)dw}du & ,
\end{align*}
towards the distribution of $\int_{M\leq\left|u\right|<M+1}(\exp\left(\mathscr{V}\left(u\right)\right)du/\int_{\left|w\right|\leq K_{2}}\exp\left(\mathscr{V}\left(w\right)\right)dw)$.
By choosing $K_{2}>M+1$ we deduce 
\begin{align*}
\mathbb{E}\left[\int_{M\leq\left|u\right|<M+1}\frac{\exp\left(\mathscr{V}\left(u\right)\right)}{\int_{\mathbb{R}}\exp\left(\mathscr{V}\left(w\right)\right)dw}du\right]\leq & \lim_{T\rightarrow\infty}\mathbb{E}\left[\int_{\Gamma_{M}}\frac{\exp\left(\widetilde{G}_{T,v}\left(u,\,\widetilde{v}\right)+Q_{T,v}\left(u\right)\right)}{\int_{\left|w\right|\leq K_{2}}\exp\left(\widetilde{G}_{T,v}\left(w,\,\widetilde{v}\right)+Q_{T,v}\left(w\right)\right)dw}du\right]\leq c_{\vartheta}M^{-\vartheta},
\end{align*}
 in view of \eqref{eq. AA.4 LapBai97}. This leads to 
\begin{align}
\Psi_{l}\left(s\right) & =\int_{\left|u\right|\leq K}l\left(s-u\right)\frac{\exp\left(\mathscr{V}\left(u\right)\right)du}{\int_{\left|w\right|\leq K}\exp\left(\mathscr{V}\left(w\right)\right)dw}+o_{\mathbb{P}}\left(1\right),\label{eq. (10.10) I=000026H LapBai97}
\end{align}
where the $o_{\mathbb{P}}\left(1\right)$ term is uniform as $K$
increases to infinity. We then have the convergence of the finite-dimensional
distributions of $\Psi_{l,T}\left(s\right)$ toward $\Psi_{l}\left(s\right).$
Next, we need to prove the tightness of the sequence $\left\{ \Psi_{l,T}\left(s\right),\,T\geq1\right\} $.
More specifically, we shall show that the family of distributions
on the space of continuous functions $\mathbb{C}_{b}\left(K\right)$
generated by the contractions of $\Psi_{l,T}\left(s\right)$ on $\left\{ \left|s\right|\leq K\right\} $
are dense. For any $l\in\boldsymbol{L}$ the inequality $l\left(u\right)\leq2^{r}\left(1+\left|u\right|^{2}\right)^{r}$
holds for some $r$. Let 
\begin{align*}
\Upsilon_{K}\left(\varpi\right) & \triangleq\int_{\mathbb{R}}\sup_{\left|s\right|\leq K,\,\left|y\right|\leq\varpi}\left|l\left(s+y-u\right)-l\left(s-u\right)\right|\left(1+\left|u\right|^{2}\right)^{-r-1}du.
\end{align*}
Fix $K<\infty$. We show $\lim_{\varpi\downarrow0}\Upsilon_{K}\left(\varpi\right)=0$.
Note that for any $\kappa>0$, we can choose a $M$ such that 
\begin{align*}
\int_{\left|u\right|>M}\sup_{\left|s\right|\leq K,\,\left|y\right|\leq\varpi}\left|l\left(s+y-u\right)-l\left(s-u\right)\right|\left(1+\left|u\right|^{2}\right)^{-r-1}du & <\kappa.
\end{align*}
We now use Lusin's Theorem {[}cf. Section 3.3 in \citet{royden/fitzpatrick:10}{]}.
Since $l\left(\cdot\right)$ is measurable, there exists a continuous
function $g\left(u\right)$ in the interval $\left\{ u\in\mathbb{R}:\,\left|u\right|\leq K+2M\right\} $
which agrees with $l\left(u\right)$ except on a set whose measure
does not exceed $\kappa\left(2\overline{L}\right)^{-1}$, where $\overline{L}$
is the upper bound of $l\left(\cdot\right)$ on $\left\{ u\in\mathbb{R}:\,\left|u\right|\leq K+2M\right\} $.
Denote the modulus of continuity of $g\left(\cdot\right)$ by $w_{g}\left(\varpi\right)$.
Without loss of generality assume $\left|g\left(u\right)\right|\leq\overline{L}$
for all $u$ satisfying $\left|u\right|\leq K+2M$. Then, 
\begin{align*}
\int_{\left|u\right|>M} & \sup_{\left|s\right|\leq K,\,\left|y\right|\leq\varpi}\left|l\left(s+y-u\right)-l\left(s-u\right)\right|\left(1+\left|u\right|^{2}\right)^{-r-1}du\\
 & \leq\int_{\mathbb{R}}\sup_{\left|s\right|\leq K,\,\left|y\right|\leq\varpi}\left|l\left(s+y-u\right)-l\left(s-u\right)\right|\left(1+\left|u\right|^{2}\right)^{-r-1}du\\
 & \leq w_{g}\left(\varpi\right)\int_{\mathbb{R}}\sup_{\left|s\right|\leq K,\,\left|y\right|\leq\varpi}\left(1+\left|u\right|^{2}\right)^{-r-k}du+2\overline{L}\mathrm{Leb}\left\{ u\in\mathbb{R}:\,\left|u\right|\leq K+2M,\,l\neq g\right\} ,
\end{align*}
 and $\overline{L}\leq Cw_{g}\left(\varpi\right)+\kappa$ for some
$C$. Hence, $\Upsilon_{K}\left(\varpi\right)\leq Cw_{g}\left(\varpi\right)+2\kappa$
since $\kappa$ can be chosen arbitrarily small and (for each fixed
$\kappa$) $w_{g}\left(\varpi\right)\rightarrow0$ as $\varpi\downarrow0$
by definition. By Assumption \ref{Assumption Uniquess loss Fucntion LapBai97},
there exists a number $C<\infty$ such that 
\begin{align*}
\mathbb{E} & \left[\sup_{\left|s\right|\leq K,\,\left|y\right|\leq\varpi}\left|\Psi_{l,T}\left(s+y\right)-\Psi_{l,T}\left(s\right)\right|\right]\\
 & \leq\int_{\mathbb{R}}\sup_{\left|s\right|\leq K,\,\left|y\right|\leq\varpi}\left|l\left(s+y-u\right)-l\left(s-u\right)\right|\mathbb{E}\left(\frac{\exp\left(\widetilde{G}_{T,v}\left(u,\,\widetilde{v}\right)+Q_{T,v}\left(u\right)\right)}{\int_{\mathbf{U}_{T}}\exp\left(\widetilde{G}_{T,v}\left(w,\,\widetilde{v}\right)+Q_{T,v}\left(w\right)\right)dw}\right)du\\
 & \leq C\Upsilon_{K}\left(\varpi\right).
\end{align*}
Markov's inequality together with the above bound establish that
the family of distributions generated by the contractions of $\Psi_{T,l}$
is dense in $\mathbb{C}_{b}\left(K\right)$. Since the finite-dimensional
convergence in distribution was demonstrated above, we can deduce
the weak convergence $\Psi_{l,T}\Rightarrow\Psi_{l}$ in $\mathbb{D}_{b}\left(\mathbf{V}\right)$
uniformly in $\lambda_{b}^{0}\in\mathbf{K}$. Finally, we examine
the oscillations of the minimum points of the sample criterion $\Psi_{l,T}.$
Consider an open bounded interval $\mathbf{A}$ that satisfies $\mathbb{P}\left\{ \xi_{l}^{0}\in b\left(\mathbf{A}\right)\right\} =0$,
where $b\left(\mathbf{A}\right)$ denotes the boundary of the set
$\mathbf{A}$. Choose a real number $K$ sufficiently large such that
$\mathbf{A}\subset\left\{ s:\,\left|s\right|\leq K\right\} $ and
define for $\left|s\right|\leq K$ the functionals $H_{\mathbf{A}}\left(\Psi\right)=\inf_{s\in\mathbf{A}}\Psi_{l}\left(s\right)$
and $H_{\mathbf{A}^{c}}\left(\Psi\right)=\inf_{s\in\mathbf{A}^{c}}\Psi_{l}\left(s\right)$.
Let $\mathbf{M}_{T}$ denote the set of minimum points of $\Psi_{l,T}$.
We have 
\begin{align*}
\mathbb{P}\left[\mathbf{M}_{T}\subset\mathbf{A}\right] & =\mathbb{P}\left[H_{\mathbf{A}}\left(\Psi\right)<H_{\mathbf{A}^{c}}\left(\Psi\right),\,\mathbf{M}_{T}\subset\left\{ s:\,\left|s\right|\leq K\right\} \right]\\
 & \geq\mathbb{P}\left[H_{\mathbf{A}}\left(\Psi\right)<H_{\mathbf{A}^{c}}\left(\Psi\right)\right]-\mathbb{P}\left[\mathbf{M}_{T}\nsubseteq\left\{ s:\,\left|s\right|\leq K\right\} \right].
\end{align*}
Therefore, 
\begin{align*}
\liminf_{T\rightarrow\infty}\mathbb{P}\left[\mathbf{M}_{T}\subset\mathbf{A}\right] & \geq\mathbb{P}\left[H_{\mathbf{A}}\left(\Psi\right)<H_{\mathbf{A}^{c}}\left(\Psi\right)\right]-\sup_{T}\mathbb{P}\left[\mathbf{M}_{T}\nsubseteq\left\{ s:\,\left|s\right|\leq K\right\} \right],
\end{align*}
 and $\limsup_{T\rightarrow\infty}\mathbb{P}\left[\mathbf{M}_{T}\subset\mathbf{A}\right]\leq\mathbb{P}\left[H_{\mathbf{A}}\left(\Psi\right)<H_{\mathbf{A}^{c}}\left(\Psi\right)\right].$
Moreover, the minimum of the population criterion $\Psi_{l}\left(\cdot\right)$
satisfies $\mathbb{P}\left[\xi_{l}^{0}\in\mathbf{A}\right]\leq\mathbb{P}\left[H_{\mathbf{A}}\left(\Psi\right)<H_{\mathbf{A}^{c}}\left(\Psi\right)\right]$
and $\mathbb{P}\left[\xi_{l}^{0}\in\mathbf{A}\right]+\mathbb{P}\left[\left|\xi_{l}^{0}\right|>K\right]\geq\mathbb{P}\left[H_{\mathbf{A}}\left(\Psi\right)\leq H_{\mathbf{A}^{c}}\left(\Psi\right)\right].$
Lemma \ref{Theorem 5.2 I=000026H LapBai97} shall be used to deduce
that the following relationship holds, 
\begin{align*}
\limsup_{T\rightarrow\infty}\mathbb{E}\left[l\left(T\left\Vert \delta_{T}\right\Vert ^{2}\left(\widehat{\lambda}_{b}^{\mathrm{GL}}-\lambda_{b}^{0}\right)\right)\right] & <\infty,
\end{align*}
for any loss function $l\in\boldsymbol{L}$. Hence, the set $\mathbf{M}_{T}$
of absolute minimum points of the function $\Psi_{l,T}\left(s\right)$
are uniformly stochastically bounded for all $T$ large enough: $\lim_{K\rightarrow\infty}\mathbb{P}\left[\mathbf{M}_{T}\nsubseteq\left\{ s:\,\left|s\right|\leq K\right\} \right]=0$.
The latter result together with the uniqueness assumption (cf. Assumption
\ref{Assumption Uniquess loss Fucntion LapBai97}) yield 
\begin{align*}
\lim_{K\rightarrow\infty}\left\{ \sup_{T}\mathbb{P}\left[\mathbf{M}_{T}\nsubseteq\left\{ s:\,\left|s\right|\leq K\right\} \right]+\mathbb{P}\left[\left|\xi_{l}^{0}\right|>K\right]\right\}  & =0.
\end{align*}
Hence, we have
\begin{align}
\lim_{T\rightarrow\infty}\mathbb{P}\left[\mathbf{M}_{T}\subset\mathbf{A}\right] & =\mathbb{P}\left[\xi_{l}^{0}\in\mathbf{A}\right].\label{eq. AA.5 LapBai}
\end{align}
 The last step involves showing that the length of the set $\mathbf{M}_{T}$
approaches zero in probability as $T\rightarrow\infty$. Let $\mathbf{A}_{d}$
denote an interval in $\mathbb{R}$ centered at the origin and of
length $d<\infty.$ Equation \eqref{eq. AA.5 LapBai} guarantees that
$\lim_{d\rightarrow\infty}\sup_{T\rightarrow\infty}\mathbb{P}\left[\mathbf{M}_{T}\nsubseteq\mathbf{A}_{d}\right]=0.$
Choose any $\epsilon>0$ and divide $\mathbf{A}_{d}$ into admissible
subintervals whose lengths do not exceed $\epsilon/2$. Then, 
\begin{align*}
\mathbb{P}\left[\sup_{s_{i},s_{j}\in\mathbf{M}_{T}}\left|s_{i}-s_{j}\right|>\epsilon\right] & \leq\mathbb{P}\left[\mathbf{M}_{T}\nsubseteq\mathbf{A}_{d}\right]+\left(1+2d/\epsilon\right)\sup\mathbb{P}\left[H_{\mathbf{A}}\left(\Psi_{l,T}\right)=H_{\mathbf{A}^{c}}\left(\Psi_{l,T}\right)\right],
\end{align*}
where the term $1+2d/\epsilon$ is an upper bound on the admissible
number of subintervals and the supremum in the second term  is over
all possible open bounded subintervals $\mathbf{A}\subset\mathbf{A}_{d}$.
The weak convergence result implies $\mathbb{P}\left[H_{\mathbf{A}}\left(\Psi_{l,T}\right)=H_{\mathbf{A}^{c}}\left(\Psi_{l,T}\right)\right]\rightarrow\mathbb{P}\left[H_{\mathbf{A}}\left(\Psi_{l}\right)=H_{\mathbf{A}^{c}}\left(\Psi_{l}\right)\right]$
as $T\rightarrow\infty.$ Since $\mathbb{P}\left[H_{\mathbf{A}}\left(\Psi_{l}\right)=H_{\mathbf{A}^{c}}\left(\Psi_{l}\right)\right]=0$
and $\mathbb{P}\left[\mathbf{M}_{T}\nsubseteq\mathbf{A}_{d}\right]\rightarrow0$
for large $d,$ then $\mathbb{P}\left[\sup_{s_{i},s_{j}\in\mathbf{M}_{T}}\left|s_{i}-s_{j}\right|>\epsilon\right]=o\left(1\right)$.
Since $\epsilon>0$ can be chosen arbitrarily small we deduce that
the distribution of $T\left\Vert \delta_{T}\right\Vert ^{2}\left(\widehat{\lambda}_{b}^{\mathrm{GL}}-\lambda_{b}^{0}\right)$
converges to the distribution of $\xi_{l}^{0}$. 
\end{proof}
\begin{lem}
\label{Lemma III.5.2 I=000026H LapBai97}Let $u_{1},\,u_{2}\in\mathbb{R}$
be of the same sign with $0<\left|u_{1}\right|<\left|u_{2}\right|$.
For any integer $r>0$ and some constants $c_{r}$ and $C_{r}$ which
depend on $r$ only, we have uniformly in $\widetilde{v}\in\mathbf{V}$,
\begin{align*}
\mathbb{E}\left[\left(\zeta_{T,v}^{1/2r}\left(u_{2},\,\widetilde{v}\right)-\zeta_{T,v}^{1/2r}\left(u_{1},\,\widetilde{v}\right)\right)^{2r}\right] & \leq c_{r}\left|\left(\delta^{0}\right)'\left(\left|u_{2}-u_{1}\right|\Sigma_{i}\right)\delta^{0}\right|^{r}\leq C_{r}\left|u_{2}-u_{1}\right|^{r},
\end{align*}
 where $\Sigma_{i}$ is defined in Assumption \ref{Assumption A.9b Bai 97, LapBai97}
and $i=1$ if $u_{1}<0$ and $i=2$ if $u_{1}>0$.
\end{lem}
\begin{proof}
The proof is given for the case $u_{2}>u_{1}>0$. The other case is
similar and thus omitted. We follow closely the proof of Lemma III.5.2
in \citet{ibragimov/has:81}. Let $\mathcal{V}\left(u_{i}\right)=\exp\left(\mathscr{V}\left(u_{i}\right)\right)$,
$i=1,\,2$. We have $\mathbb{E}\left[\left(\mathcal{V}^{1/2r}\left(u_{2}\right)-\mathcal{V}^{1/2r}\left(u_{1}\right)\right)^{2r}\right]=\sum_{j=0}^{2r}\binom{2r}{j}\left(-1\right)^{j}\mathbb{E}_{u_{1}}\left[\mathcal{V}_{u_{1}}^{j/2r}\left(u_{2}\right)\right],$
where $\mathcal{V}_{u_{1}}\left(u_{2}\right)\triangleq\exp\left(\mathscr{V}\left(u_{2}\right)-\mathscr{V}\left(u_{1}\right)\right).$
Using the Gaussian property of $\mathscr{V}\left(u\right)$, for each
$u\in\mathbb{R},$ we have
\begin{align}
\mathbb{E}_{u_{1}}\left[\mathcal{V}^{j/2r}\left(u_{2}\right)\right] & =\exp\left(\frac{1}{2}\left(\frac{j}{2r}\right)^{2}4\left(\delta^{0}\right)'\left(\left|u_{2}-u_{1}\right|\Sigma_{2}\right)\delta^{0}-\frac{j}{2r}\left|\varLambda^{0}\left(u_{2}\right)-\varLambda^{0}\left(u_{1}\right)\right|\right).\label{eq. (5.9) I=000026H LapBai97}
\end{align}
 Then, $\mathbb{E}\left[\left(\mathcal{V}^{1/2r}\left(u_{2}\right)-\mathcal{V}^{1/2r}\left(u_{1}\right)\right)^{2r}\right]=\sum_{j=0}^{2r}\binom{2r}{j}\left(-1\right)^{j}d^{j/2r}$
with 
\begin{align*}
d & \triangleq\exp\left(\frac{j}{2r}2\left(\delta^{0}\right)'\left(\left|u_{2}-u_{1}\right|\Sigma_{2}\right)\delta^{0}-\left|\varLambda^{0}\left(u_{2}\right)-\varLambda^{0}\left(u_{1}\right)\right|\right).
\end{align*}
 Let $B\triangleq2\left(\delta^{0}\right)'\left(\left|u_{2}-u_{1}\right|\Sigma_{2}\right)\delta^{0}-\left|\varLambda^{0}\left(u_{2}\right)-\varLambda^{0}\left(u_{1}\right)\right|$.
There are different cases to be considered:\\
(1) $B<0$. Note that 
\begin{align*}
d & =\exp\left(\frac{j}{2r}2\left(\delta^{0}\right)'\left(\left|u_{2}-u_{1}\right|\Sigma_{2}\right)\delta^{0}-\left|\left(\delta^{0}\right)'\left(\left|u_{2}-u_{1}\right|\Sigma_{2}\right)\delta^{0}\right|+B\right)\\
 & =\exp\left(-\frac{2r-j}{r}\left(\delta^{0}\right)'\left(\left|u_{2}-u_{1}\right|\Sigma_{2}\right)\delta^{0}\right)e^{B},
\end{align*}
which then results in 
\begin{align}
\mathbb{E}\left[\left(\mathcal{V}^{1/2r}\left(u_{2}\right)-\mathcal{V}^{1/2r}\left(u_{1}\right)\right)^{2r}\right] & \leq p_{r}\left(a\right),\label{eq. 5.10 I=000026H LapBai97}
\end{align}
 where $p_{r}\left(a\right)\triangleq\sum_{j=0}^{2r}\binom{2r}{j}\left(-1\right)^{j}a^{\left(2r-j\right)}$
and $a=e^{B/2r}\exp\left(-r^{-1}\left(\delta^{0}\right)'\left(\left|u_{2}-u_{1}\right|\Sigma_{2}\right)\delta^{0}\right).$
\\
(2) $2\left(\delta^{0}\right)'\left(\left|u_{2}-u_{1}\right|\Sigma_{2}\right)\delta^{0}=\left|\varLambda^{0}\left(u_{2}\right)-\varLambda^{0}\left(u_{1}\right)\right|$.
This case is the same as the previous one but with $a=\exp\left(-r^{-1}\left(\delta^{0}\right)'\left(\left|u_{2}-u_{1}\right|\Sigma_{2}\right)\delta^{0}\right)$.
\\
(3) $B>0$. Upon simple manipulations, $\mathbb{E}\left[\left(\mathcal{V}^{1/2r}\left(u_{2}\right)-\mathcal{V}^{1/2r}\left(u_{1}\right)\right)^{2r}\right]\leq p_{r}\left(a\right),$
where 
\begin{align*}
p_{r}\left(a\right)=e^{-B/2r}\sum_{j=0}^{2r}\binom{2r}{j}\left(-1\right)^{j}a^{\left(2r-j\right)} & ,
\end{align*}
with $a=\exp\left(-r^{-1}\left(\delta^{0}\right)'\left(\left|u_{2}-u_{1}\right|\Sigma_{2}\right)\delta^{0}\right)$.
We can thus proceed with the same proof for all the above cases. Let
us consider the first case. We show that at the point $a=1,$ the
polynomial $p_{r}\left(a\right)$ admits a root of multiplicity $r$.
This can be established by verifying the equalities $p_{r}\left(1\right)=p_{r}^{\left(1\right)}\left(1\right)=\cdots=p_{r}^{\left(r-1\right)}\left(1\right)=0$.
One then recognizes that $p_{r}^{\left(i\right)}\left(a\right)$ is
a linear combination of summations $\mathcal{S}_{k}$ ($k=0,\,1,\ldots,\,2i$)
given by $\mathcal{S}_{k}=e^{B}\sum_{j=0}^{2r}\binom{2r}{j}j^{k}$.
Thus, one only needs to verify that $\mathcal{S}_{k}=0$ for $k=0,\,1,\ldots,\,2r-2$.
This follows because the expression for $\mathcal{S}_{k}$ is found
by applying the operator $e^{B}a\left(d/da\right)$ to the function
$\left(1-a^{2}\right)^{2r}$ and evaluating it at $a=1.$ Consequently,
$\mathcal{S}_{k}=0$ for $k=0,\,1,\ldots,\,2r-1$. Using this result
into \eqref{eq. 5.10 I=000026H LapBai97} we find, with $\widetilde{p}_{r}\left(a\right)$
being a polynomial of degree $r^{2}-r$, 
\begin{align}
\mathbb{E}\left[\left(\mathscr{V}^{1/2r}\left(u_{2}\right)-\mathscr{V}^{1/2r}\left(u_{1}\right)\right)^{2r}\right] & =\left(1-a\right)^{r}\widetilde{p}_{r}\left(a\right)\leq\left(r^{-1}\left(\delta^{0}\right)'\left(\left|u_{2}-u_{1}\right|\Sigma_{2}\right)\delta^{0}\right)^{r}\widetilde{p}_{r}\left(a\right),\label{eq. (5.10b) LapBai97}
\end{align}
where the last inequality follows from $1-e^{-c}\leq c$, for $c>0.$
Next, let $\overline{\zeta}_{T,v}^{1/2r}\left(u_{2},\,u_{1}\right)=\zeta_{T,v}^{1/2r}\left(u_{2}\right)-\zeta_{T,v}^{1/2r}\left(u_{1}\right)$.
By Lemma \ref{Lemma ZZ conv prob}  and \ref{Lemma B.2, LapBai97},
the continuous mapping theorem and \eqref{eq. (5.10b) LapBai97},
$\lim_{T\rightarrow\infty}\mathbb{E}\left[\overline{\zeta}_{T,v}^{1/2r}\left(u_{2},\,u_{1}\right)\right]\leq\left(1-a\right)^{r}\widetilde{p}_{r}\left(a\right),$
uniformly in $\widetilde{v}\in\mathbf{V}$. Noting that $j\leq2r,$
we can set $C_{r}=\max_{0\leq a\leq1}e^{B}\widetilde{p}_{r}\left(a\right)/r^{r}$
to prove the lemma. 
\end{proof}
\begin{lem}
\label{Lemma VII.Lemma 2.2 I=000026H LapBai97 }For $u_{1},\,u_{2}\in\mathbb{R}$
being of the same sign and satisfying $0<\left|u_{1}\right|<\left|u_{2}\right|<K<\infty$.
Then, for all $T$ sufficiently large, we have
\begin{align}
\mathbb{E}\left[\left(\zeta_{T,v}^{1/4}\left(u_{2},\,\widetilde{v}\right)-\zeta_{T,v}^{1/4}\left(u_{1},\,\widetilde{v}\right)\right)^{4}\right] & \leq C_{1}\left|u_{2}-u_{1}\right|^{2},\label{eq. (VII 2.7) I=000026H LapBai97}
\end{align}
 where $0<C_{1}<\infty$. Furthermore, for the constant $C_{1}$ from
Lemma \ref{Lemma III.5.2 I=000026H LapBai97}, we have
\begin{align}
\mathbb{P}\left[\zeta_{T,v}\left(u,\,\widetilde{v}\right)>\exp\left(-3C_{1}\left|u\right|/2\right)\right] & \leq\exp\left(-C_{1}\left|u\right|/4\right).\label{eq. VII 2.8  I=000026H LapBai97}
\end{align}
Both relationships are valid uniformly in $\widetilde{v}\in\mathbf{V}$. 
\end{lem}
\begin{proof}
Suppose $u>0.$ The relationship in \eqref{eq. (VII 2.7) I=000026H LapBai97}
follows from Lemma \ref{Lemma III.5.2 I=000026H LapBai97} with $r=2.$
By Markov's inequality and Lemma \ref{Lemma III.5.2 I=000026H LapBai97},
\begin{align*}
\mathbb{P}\left[\zeta_{T,v}\left(u,\,\widetilde{v}\right)>\exp\left(-3C_{1}\left|u\right|/2\right)\right] & \leq\exp\left(3C_{1}\left|u\right|/4\right)\mathbb{E}\left[\zeta_{T,v}^{1/2}\left(u,\,\widetilde{v}\right)\right]\\
 & \leq\exp\left(3C_{1}\left|u\right|/4-\left(\delta^{0}\right)'\left(\left|u\right|\Sigma_{2}\right)\delta^{0}\right)\leq\exp\left(-C_{1}\left|u\right|/4\right).
\end{align*}
\end{proof}
\begin{lem}
\label{Lemma Remark 5.1 I=000026H p. 44 LapBai97}Under the conditions
of Lemma \ref{Lemma VII.Lemma 2.2 I=000026H LapBai97 }, for any $\vartheta>0$
there exists a finite real number $c_{\vartheta}$ and a $\overline{T}$
such that for all $T>\overline{T}$, $\sup_{\widetilde{v}\in\mathbf{V}}\mathbb{P}\left[\sup_{\left|u\right|>M}\zeta_{T,v}\left(u,\,\widetilde{v}\right)>M^{-\vartheta}\right]\leq c_{\vartheta}M^{-\vartheta}.$
\end{lem}
\begin{proof}
It can be shown using Lemma \ref{Lemma III.5.2 I=000026H LapBai97}-\ref{Lemma VII.Lemma 2.2 I=000026H LapBai97 }.
\end{proof}

\begin{lem}
\label{Lemma I.5.1 I=000026H LapBai97}For every sufficiently small
$\epsilon\leq\overline{\epsilon}$, where $\overline{\epsilon}$ depends
on the smoothness of $\pi\left(\cdot\right),$ there exists $0<C<\infty$
such that
\begin{align}
\mathbb{P}\left[\int_{0}^{\epsilon}\zeta_{T,v}\left(u,\,\widetilde{v}\right)\pi\left(\lambda_{b}^{0}+u/\psi_{T}\right)du<\epsilon\pi\left(\lambda_{b}^{0}\right)\right] & <C\epsilon^{1/2}.\label{eq. (5.8, p.45) I=000026H LapBai97}
\end{align}
\end{lem}
\begin{proof}
Since $\mathbb{E}\left(\zeta_{T,v}\left(0,\,\widetilde{v}\right)\right)=1$
and $\mathbb{E}\left(\zeta_{T,v}\left(u,\,\widetilde{v}\right)\right)\leq1$
for sufficiently large $T$, we have
\begin{align}
\mathbb{E}\left|\zeta_{T,v}\left(u,\,\widetilde{v}\right)-\zeta_{T,v}\left(0,\,\widetilde{v}\right)\right| & \leq\left(\mathbb{E}\left|\zeta_{T,v}^{1/2}\left(u,\,\widetilde{v}\right)+\zeta_{T,v}^{1/2}\left(0,\,\widetilde{v}\right)\right|^{2}\mathbb{E}\left|\zeta_{T,v}^{1/2}\left(u,\,\widetilde{v}\right)-\zeta_{T,v}^{1/2}\left(0,\,\widetilde{v}\right)\right|^{2}\right)^{1/2}\leq C\left|u\right|^{1/2},\label{eq. (5.9, p.45) I=000026H LapBai97}
\end{align}
by Lemma \ref{Lemma III.5.2 I=000026H LapBai97} with $r=1$. By Assumption
\ref{Assumption Prior LapBai97}, $\left|\pi_{T,v}\left(u\right)-\pi^{0}\right|\leq\left|\pi\left(\lambda_{b,T}^{0}\left(v\right)\right)-\pi^{0}\right|+C\psi_{T}^{-1}\left|u\right|,$
with $C>0$. The first term on the right-hand side is $o\left(1\right)$
(and independent of $u$) while the second is asymptotically negligible
for small $u$. Thus, for a sufficiently small $\overline{\epsilon}>0,$
\begin{align*}
\int_{0}^{\epsilon}\zeta_{T,v}\left(u,\,\widetilde{v}\right)\pi_{T,v}\left(u\right)du>\frac{\pi^{0}}{2}\int_{0}^{\epsilon}\zeta_{T,v}\left(u,\,\widetilde{v}\right)du & .
\end{align*}
Next, using $\zeta_{T,v}\left(0,\,\widetilde{v}\right)=1$, 
\begin{align*}
\mathbb{P}\left[\int_{0}^{\epsilon}\zeta_{T,v}\left(u,\,\widetilde{v}\right)\pi_{T,v}\left(u\right)du<\epsilon/2\right] & \leq\mathbb{P}\left[\int_{0}^{\epsilon}\left(\zeta_{T,v}\left(u,\,\widetilde{v}\right)-\zeta_{T,v}\left(0,\,\widetilde{v}\right)\right)du<-\epsilon/2\right]\\
 & \leq\mathbb{P}\left[\int_{0}^{\epsilon}\left|\zeta_{T,v}\left(u,\,\widetilde{v}\right)-\zeta_{T,v}\left(0,\,\widetilde{v}\right)\right|du>\epsilon/2\right],
\end{align*}
 and by Markov's inequality together with \eqref{eq. (5.9, p.45) I=000026H LapBai97}
the last expression is less than or equal to 
\begin{align*}
\left(2/\epsilon\right)\int_{0}^{\epsilon}\mathbb{E}\left|\zeta_{T,v}\left(u,\,\widetilde{v}\right)-\zeta_{T,v}\left(0,\,\widetilde{v}\right)\right|du<2C\epsilon^{1/2} & .
\end{align*}
\end{proof}
\begin{lem}
\label{Lemma I.5.2.(i) I=000026H LapBai97}For $f_{T}\in\boldsymbol{F},$
and $M$ sufficiently large, there exist constants $c,\,C>0$ such
that
\begin{align}
\mathbb{P}\left[J_{1,M}>\exp\left(-cf_{T}\left(M\right)\right)\right] & \leq C\left(1+M^{C}\right)\exp\left(-cf_{T}\left(M\right)\right),\label{eq. (5.10) I=000026H LapBai97}
\end{align}
uniformly in $\widetilde{v}\in\mathbf{V}$.
\end{lem}
\begin{proof}
In view of the smotheness property of $\pi\left(\cdot\right)$, without
loss of generality we consider the case of the uniform prior (i.e.,
$\pi_{T,v}\left(u\right)=1$ for all $u$). We begin by dividing the
open interval $\left\{ u:\,M\leq\left|u\right|<M+1\right\} $ into
$I$ disjoint segments denoting the $i$-th one by $\Pi_{i}$. For
each segment $\Pi_{i}$ choose a point $u_{i}$ and define $J_{1,M}^{\Pi}\triangleq\sup_{\widetilde{v}\in\mathbf{V}}\sum_{i\in I}\zeta_{T,v}\left(u_{i},\,\widetilde{v}\right)\mathrm{Leb}\left(\Pi_{i}\right)=\sup_{\widetilde{v}\in\mathbf{V}}\sum_{i\in I}\int_{\Pi_{i}}\zeta_{T,v}\left(u_{i},\,\widetilde{v}\right)du.$
Then,
\begin{align}
\mathbb{P}\left[J_{1,M}^{\Pi}>\left(1/4\right)\exp\left(-cf_{T}\left(M\right)\right)\right] & \leq\mathbb{P}\left[\max_{i\in I}\sup_{\widetilde{v}\in\mathbf{V}}\zeta_{T,v}^{1/2}\left(u_{i},\,\widetilde{v}\right)\left(\mathrm{Leb}\left(\Gamma_{M}\right)\right)^{1/2}>\left(1/2\right)\exp\left(-f_{T}\left(M\right)/2\right)\right]\nonumber \\
 & \leq\sum_{i\in I}\mathbb{P}\left[\zeta_{T,v}^{1/2}\left(u_{i},\,\widetilde{v}\right)>\left(1/2\right)\left(\mathrm{Leb}\left(\Gamma_{M}\right)\right)^{-1/2}\exp\left(-f_{T}\left(M\right)/2\right)\right]\nonumber \\
 & \leq2I\left(\mathrm{Leb}\left(\Gamma_{M}\right)\right)^{1/2}\exp\left(-f_{T}\left(M\right)/12\right),\label{eq. (5.11, p.47) I=000026H}
\end{align}
where the last inequality follows from applying Lemma \ref{Lemma VII.Lemma 2.2 I=000026H LapBai97 }
to each summand. Upon using the inequality $\exp\left(-f_{T}\left(M\right)/2\right)<1/2$
(which is valid for sufficiently large $M$), we have
\begin{align*}
\mathbb{P}\left[J_{1,M}>\exp\left(-f_{T}\left(M\right)/2\right)\right] & \leq\mathbb{P}\left[\left|J_{1,M}-J_{1,M}^{\Pi}\right|>\left(1/2\right)\exp\left(-f_{T}\left(M\right)/2\right)\right]+\mathbb{P}\left[J_{1,M}^{\Pi}>\exp\left(-f_{T}\left(M\right)\right)\right].
\end{align*}
 Focusing on the first term, 
\begin{align*}
\mathbb{E}\left[J_{1,M}-J_{1,M}^{\Pi}\right] & \leq\sum_{i\in I}\int_{\Pi_{i}}\mathbb{E}\left|\zeta_{T,v}^{1/2}\left(u,\,\widetilde{v}\right)-\zeta_{T,v}^{1/2}\left(u_{i},\,\widetilde{v}\right)\right|du\\
 & \leq\sum_{i\in I}\int_{\Pi_{i}}\left(\mathbb{E}\left|\zeta_{T,v}^{1/2}\left(u,\,\widetilde{v}\right)+\zeta_{T,v}^{1/2}\left(u_{i},\,\widetilde{v}\right)\right|\mathbb{E}\left|\zeta_{T,v}^{1/2}\left(u,\,\widetilde{v}\right)-\zeta_{T,v}^{1/2}\left(u_{i},\,\widetilde{v}\right)\right|\right)^{1/2}du\\
 & \leq C\left(1+M\right)^{C}\sum_{i\in I}\int_{\Pi_{i}}\left|u_{i}-u\right|^{1/2}du,
\end{align*}
 where for the last inequality we have used Lemma \ref{Lemma VII.Lemma 2.2 I=000026H LapBai97 }
since we can always choose the partition of the segments such that
each $\Pi_{i}$ contains either positive or negative $u_{i}$. Since
each summand on the right-hand side above is less than $C\left(MI^{-1}\right)^{3/2}$
there exist numbers $C_{1}$ and $C_{2}$ such that
\begin{align}
\mathbb{E}\left[J_{1,M}-J_{1,M}^{\Pi}\right] & \leq C_{1}\left(1+M^{C_{2}}\right)I^{-1/2}.\label{eq. (5.12, p.47) I=000026H LapBai97}
\end{align}
 Using \eqref{eq. (5.11, p.47) I=000026H} and \eqref{eq. (5.12, p.47) I=000026H LapBai97}
we have 
\begin{align*}
\mathbb{P}\left[J_{1,M}>\exp\left(-f_{T}\left(M\right)/2\right)\right] & \leq C_{1}\left(1+M^{C_{2}}\right)I^{-1/2}+2I\left(\mathrm{Leb}\left(\Gamma_{M}\right)\right)^{1/2}\exp\left(-f_{T}\left(M\right)/12\right).
\end{align*}
The relationship in the last display leads to the claim of the lemma
if we choose $I$ satisfying $1\leq I^{3/2}\exp\left(-f_{T}\left(M\right)/4\right)\leq2.$ 
\end{proof}
\begin{lem}
\label{Lemma 7.5. I.5.2.(ii) }For $f_{T}\in\boldsymbol{F},$ and
$M$ sufficiently large, there exist constants $c,\,C>0$ such that
\begin{align}
\mathbb{E}\left[J_{1,M}/J_{2}\right] & \leq C\left(1+M^{C}\right)\exp\left(-cf_{T}\left(M\right)\right),\label{eq. (5.10, Lemma I.5.2) I=000026H LapBai97}
\end{align}
 uniformly in $\widetilde{v}\in\mathbf{V}.$
\end{lem}
\begin{proof}
Note that $J_{1,M}/J_{2}\leq1$. Thus, for any $\epsilon>0,$ 
\begin{align*}
\mathbb{E}\left[J_{1,M}/J_{2}\right] & \leq\mathbb{P}\left[J_{1,M}>\exp\left(-cf_{T}\left(M\right)/2\right)\right]+\left(4/\epsilon\right)\exp\left(-cf_{T}\left(M\right)\right)+\mathbb{P}\left[\int_{\Gamma_{T}}\zeta_{T,v}\left(u,\,\widetilde{v}\right)du<\epsilon/4\right].
\end{align*}
 By Lemma \ref{Lemma I.5.2.(i) I=000026H LapBai97}, the first term
is bounded by $C\left(1+M^{C}\right)\exp\left(-cf_{T}\left(M\right)/4\right)$
while for the last term we can use \eqref{eq. (5.8, p.45) I=000026H LapBai97}
to deduce 
\begin{align*}
\mathbb{E}\left[J_{1,M}/J_{2}\right] & \leq C\left(1+M^{C}\right)\exp\left(-cf_{T}\left(M\right)\right)+\left(4/\epsilon\right)\exp\left(-cf_{T}\left(M\right)\right)+C\epsilon^{1/2}.
\end{align*}
 Finally, choose $\epsilon=\exp\left(\left(-2c/3\right)f_{T}\left(M\right)\right)$
to complete the proof of the lemma. 
\end{proof}
\begin{lem}
\label{Theorem 5.2 I=000026H LapBai97}For $l\in\boldsymbol{L}$ and
and any $\vartheta>0,$ $\lim_{B\rightarrow\infty}\lim_{T\rightarrow\infty}B^{\vartheta}\mathbb{P}\left[\psi_{T}\left(\widehat{\lambda}_{b}^{\mathrm{GL}}-\lambda_{b}^{0}\right)>B\right]=0.$
\end{lem}
\begin{proof}
Let $p_{T}\left(u\right)\triangleq p_{1,T}\left(u\right)/\overline{p}_{T}$
where $p_{1,T}\left(u\right)=\exp\left(\widetilde{G}_{T,v}\left(u,\,\widetilde{v}\right)+Q_{T,v}\left(u\right)\right)$
and $\overline{p}_{T}\triangleq\int_{\mathbf{U}_{T}}p_{1,T}\left(w\right)dw$.
By definition, $\widehat{\lambda}_{b}^{\mathrm{GL}}$ is the minimum
of the function $\int_{\varGamma^{0}}l\left(T\left\Vert \delta_{T}\right\Vert ^{2}\left(s-u\right)\right)p_{1,T}\left(u\right)\pi_{T,v}\left(u\right)du$
with $s\in\varGamma^{0}$. Upon using a change in variables,
\begin{align*}
\int_{\varGamma^{0}} & l\left(T\left\Vert \delta_{T}\right\Vert ^{2}\left(s-u\right)\right)p_{1,T}\left(u\right)\pi_{T,v}\left(u\right)du\\
 & =\left(T\left\Vert \delta_{T}\right\Vert ^{2}\right)^{-1}\overline{p}_{T}\int_{\mathbf{U}_{T}}l\left(T\left\Vert \delta_{T}\right\Vert ^{2}\left(s-\lambda_{b}^{0}\right)-u\right)p_{T}\left(\lambda_{b,T}^{0}\left(v\right)+\left(T\left\Vert \delta_{T}\right\Vert ^{2}\right)^{-1}u\right)\\
 & \quad\times\pi_{T,v}\left(\lambda_{b,T}^{0}\left(v\right)+\left(T\left\Vert \delta_{T}\right\Vert ^{2}\right)^{-1}u\right)du.
\end{align*}
 Thus, $\lambda_{\delta,T}\triangleq T\left\Vert \delta_{T}\right\Vert ^{2}\left(\widehat{\lambda}_{b}^{\mathrm{GL}}-\lambda_{b}^{0}\right)$
is the\textbf{ }minimum of the function 
\begin{align*}
\mathcal{S}_{T}\left(s\right) & \triangleq\int_{\mathbf{U}_{T}}l\left(s-u\right)\frac{p_{T}\left(\lambda_{b}^{0}+\left(T\left\Vert \delta_{T}\right\Vert ^{2}\right)^{-1}u\right)\pi_{T,v}\left(\lambda_{b}^{0}+\left(T\left\Vert \delta_{T}\right\Vert ^{2}\right)^{-1}u\right)}{\int_{\mathbf{U}_{T}}p_{T}\left(\lambda_{b}^{0}+\left(T\left\Vert \delta_{T}\right\Vert ^{2}\right)^{-1}w\right)\pi_{T,v}\left(\lambda_{b}^{0}+\left(T\left\Vert \delta_{T}\right\Vert ^{2}\right)^{-1}w\right)dw}du,
\end{align*}
 where the optimization is over $\mathbf{U}_{T}$. The random function
$\mathcal{S}_{T}\left(\cdot\right)$ converges with probability one
in view of Lemma \ref{Lemma I.5.2.(i) I=000026H LapBai97}-\ref{Lemma 7.5. I.5.2.(ii) }
together with the properties of the loss function $l$ {[}cf. \eqref{eq. (10.9) I=000026H LapBai97}
and the discussion surrounding it{]}. Therefore, we shall show that
the random function $\mathcal{S}_{T}\left(s\right)$ is strictly larger
than $\mathcal{S}_{T}\left(0\right)$ on $\left\{ \left|s\right|>B\right\} $
with high probability as $T\rightarrow\infty$. This  reflects that
\begin{align}
\mathbb{P}\left[\left|T\left\Vert \delta_{T}\right\Vert ^{2}\left(\widehat{\lambda}_{b}^{\mathrm{GL}}-\lambda_{b}^{0}\right)\right|>B\right] & \leq\mathbb{P}\left[\inf_{\left|s\right|>B}\mathcal{S}_{T}\left(s\right)\leq\mathcal{S}_{T}\left(0\right)\right].\label{eq. AA5 LapBai97}
\end{align}
 We present the proof for the case $\pi_{T,v}\left(u\right)=1$ for
all $u$. The general case follows with no additional difficulties
due to the assumptions satisfied by the prior $\pi\left(\cdot\right)$.
By the properties of the family $\boldsymbol{L}$ of loss functions,
we can find $\overline{u}_{1},\,\overline{u}_{2}\in\mathbb{R},$ with
$0<\overline{u}_{1}<\overline{u}_{2}$ such that as $T$ increases,
\begin{align*}
\overline{l}_{1,T}\triangleq\sup\left\{ l\left(u\right):\,u\in\Gamma_{1,T}\right\}  & <\overline{l}_{2,T}\triangleq\inf\left\{ l\left(u\right):\,u\in\Gamma_{2,T}\right\} ,
\end{align*}
 where $\Gamma_{1,T}\triangleq\mathbf{U}_{T}\cap\left(\left|u\right|\leq\overline{u}_{1}\right)$
and $\Gamma_{2,T}\triangleq\mathbf{U}_{T}\cap\left(\left|u\right|>\overline{u}_{2}\right)$.
With this notation, 
\begin{align*}
\mathcal{S}_{T}\left(0\right) & \leq\overline{l}_{1,T}\int_{\Gamma_{1,T}}p_{T}\left(u\right)du+\int_{\mathbf{U}_{T}\cap\left(\left|u\right|>\overline{u}_{1}\right)}l\left(u\right)p_{T}\left(u\right)du.
\end{align*}
 Furthermore, if $l\in\boldsymbol{L}$, then for sufficiently large
$B$ the following relationships hold: (i) $l\left(u\right)-\inf_{\left|v\right|>B/2}l\left(v\right)\leq0$;
(ii) $\left|u\right|\leq\left(B/2\right)^{\vartheta},\,\vartheta>0.$
We shall assume that $B$ is chosen so that $B>2\overline{u}_{2}$
and $\left(B/2\right)^{\vartheta}>\overline{u}_{2}$ hold. Let $\Gamma_{T,B}\triangleq\left\{ u:\,\left(\left|u\right|>B/2\right)\cap\mathbf{U}_{T}\right\} $.
Then, whenever $\left|s\right|>B$ and $\left|u\right|\leq B/2$,
we have, 
\begin{align}
\left|u-s\right|>B/2>\overline{u}_{2} & \qquad\textrm{and}\qquad\inf_{u\in\Gamma_{T,B}}l\left(u\right)\geq\overline{l}_{2,T}.\label{eq. (5.13) I=000026H LapBai97}
\end{align}
 With this notation,
\begin{align*}
\inf_{\left|s\right|>B}\mathcal{S}_{T}\left(s\right) & \geq\inf_{u\in\Gamma_{T,B}}l_{T}\left(u\right)\int_{\left(\left|w\right|\leq B/2\right)\cap\mathbf{U}_{T}}p_{T}\left(w\right)dw\\
 & \geq\overline{l}_{2,T}\int_{\left(\left|w\right|\leq B/2\right)\cap\mathbf{U}_{T}}p_{T}\left(w\right)dw,
\end{align*}
 from which it follows that 
\begin{align*}
\mathcal{S}_{T}\left(0\right)-\inf_{\left|s\right|>B}\mathcal{S}_{T}\left(s\right) & \leq-\varpi\int_{\Gamma_{1,T}}p_{T}\left(u\right)du+\int_{\mathbf{U}_{T}\cap\left(\left(B/2\right)^{\vartheta}\geq\left|u\right|\geq\overline{u}_{1}\right)}\left(l\left(u\right)-\inf_{\left|s\right|>B/2}l_{T}\left(s\right)\right)p_{T}\left(u\right)du\\
 & \quad+\int_{\mathbf{U}_{T}\cap\left(\left|u\right|>\left(B/2\right)^{\vartheta}\right)}l\left(u\right)p_{T}\left(u\right)du,
\end{align*}
where $\varpi\triangleq\overline{l}_{2,T}-\overline{l}_{1,T}$. 
The last inequality can be manipulated further using \eqref{eq. (5.13) I=000026H LapBai97},
so that
\begin{align}
\mathcal{S}_{T}\left(0\right)-\inf_{\left|s\right|>B}\mathcal{S}_{T}\left(s\right) & \leq-\varpi\int_{\Gamma_{1,T}}p_{T}\left(u\right)du+\int_{\mathbf{U}_{T}\cap\left(\left|u\right|>\left(B/2\right)^{\vartheta}\right)}l_{T}\left(u\right)p_{T}\left(u\right)du.\label{eq. (5.14) I=000026H LapBai97}
\end{align}
 Let $B_{\vartheta}\triangleq\left(B/2\right)^{\vartheta}$ and fix
an arbitrary number $\overline{a}>0$. For the first term of \eqref{eq. (5.14) I=000026H LapBai97},
Lemma \ref{Lemma I.5.1 I=000026H LapBai97} implies that for sufficiently
large $T,$ we have
\begin{align}
\mathbb{P}\left[\int_{\Gamma_{1,T}}p_{T}\left(u\right)du<2\left(\varpi B^{\overline{a}}\right)^{-1}\right] & \leq c\left(\varpi B^{\overline{a}}\right)^{-1/2},\label{eq. AA.2 LapBai97}
\end{align}
 where $0<c<\infty$. Next, let us consider the second term of \eqref{eq. (5.14) I=000026H LapBai97}.
We show that for large enough $T$, an arbitrary number $\overline{a}>0,$
\begin{align}
\mathbb{P}\left[\int_{\mathbf{U}_{T}\cap\left\{ \left|u\right|>B_{\vartheta}\right\} }l\left(u\right)p_{T}\left(u\right)du>B^{-\overline{a}}\right] & \leq cB^{-\overline{a}}.\label{eq. LapBai97 AA.1}
\end{align}
Since $l\in\boldsymbol{L}$, we have $l\left(u\right)\leq\left|u\right|^{a},\,a>0$
when $u$ is large enough. Choosing $B$ large leads to
\begin{align*}
\mathbb{E}\left[\int_{\mathbf{U}_{T}\cap\left\{ \left|u\right|>B_{\vartheta}\right\} }l\left(u\right)p_{T}\left(u\right)du\right] & \leq\sum_{i=0}^{\infty}\left(B_{\vartheta}+i+1\right)^{a}\mathbb{E}\left(J_{1,B_{\vartheta}+i}/J_{2}\right),
\end{align*}
where $J_{1,B_{\vartheta}+i},\,J_{2}$ are defined as in \eqref{eq. J_1,M}.
By Lemma \ref{Lemma 7.5. I.5.2.(ii) },
\begin{alignat*}{1}
\mathbb{E}\left(J_{1,B_{\vartheta}+i}/J_{2}\right) & \leq c\left(1+\left(B_{\vartheta}+i\right)^{a}\right)\exp\left(-bf_{T}\left(B_{\vartheta}+i\right)\right),
\end{alignat*}
where $f_{T}\in\boldsymbol{F}$ and thus for some for some $b,\,0<c<\infty,$
\begin{align*}
\mathbb{E}\left[\int_{\mathbf{U}_{T}\cap\left\{ \left|u\right|>B_{\vartheta}\right\} }l\left(u\right)p_{T}\left(u\right)du\right] & \leq c\int_{B_{\vartheta}}^{\infty}\left(1+v^{a}\right)\exp\left(-bf_{T}\left(v\right)\right)dv\leq c\exp\left(-bf_{T}\left(B_{\vartheta}\right)\right).
\end{align*}
 By property (ii) of the function $f_{T}$ in the class $\boldsymbol{F},$
for any $d\in\mathbb{R},$ $\lim_{v\rightarrow\infty}\lim_{T\rightarrow\infty}v^{d}e^{-bf_{T}\left(v\right)}=0.$
Thus, we know that for $T$ large enough and some $0<c<\infty,$ 
\begin{align*}
\mathbb{E}\left[\int_{\mathbf{U}_{T}\cap\left\{ \left|u\right|>B_{\vartheta}\right\} }l\left(u\right)p_{T}\left(u\right)du\right] & \leq cB^{-2\overline{a}},
\end{align*}
 from which we deduce \eqref{eq. LapBai97 AA.1} after applying Markov's
inequality. Therefore, for sufficiently large $T$ and large $B,$
combining equation \eqref{eq. AA5 LapBai97}, and \eqref{eq. AA.2 LapBai97}-\eqref{eq. LapBai97 AA.1},
we have 
\begin{align*}
\mathbb{P} & \left[T\left\Vert \delta_{T}\right\Vert ^{2}\left(\widehat{\lambda}_{b}^{\mathrm{GL}}-\lambda_{b}^{0}\right)>B\right]\\
 & \leq\mathbb{P}\left[-\varpi\int_{\Gamma_{1,T}}p_{T}\left(u\right)du+\int_{\mathbf{U}_{T}\cap\left\{ \left|u\right|>B_{\vartheta}\right\} }l_{T}\left(u\right)p_{T}\left(u\right)du\leq0\right]\\
 & \leq\mathbb{P}\left[\int_{\Gamma_{1,T}}p_{T}\left(u\right)du<2\left(\varpi B^{\overline{a}}\right)^{-1}\right]+\mathbb{P}\left[\int_{\mathbf{U}_{T}\cap\left\{ \left|u\right|>B_{\vartheta}\right\} }l\left(u\right)p_{T}\left(u\right)du>B^{-\overline{a}}\right]\\
 & \leq c\left(B^{-\overline{a}/2}+B^{-\overline{a}}\right),
\end{align*}
which can be made arbitrarily small choosing $B$ large enough. 
\end{proof}
\begin{lem}
\label{Lemma: FD Convergence: Condition 2 in I=000026H p.107}As $T\rightarrow\infty,$
the marginal distributions of $\zeta_{T,v}\left(u,\,\widetilde{v}\right)$
converge to the marginal distributions of $\exp\left(\mathscr{V}\left(u\right)\right)$.
\end{lem}
\begin{proof}
The results follows from Lemma \ref{Lemma ZZ conv prob}, Lemma \ref{Lemma B.2, LapBai97}
and the continuous mapping theorem.
\end{proof}

\subsection{Proofs of Section \ref{Section Inference Methods}}

\subsubsection{Proof of Proposition \ref{Proposition Inference Limi Distr LapBai97}}

The preliminary lemmas below consider the Gaussian process $\mathscr{W}$
on the positive half-line with $s>0.$ The case $s\leq0$ is similar
and omitted. The generic constant $C>0$ used in the proofs of this
section may change from line to line.
\begin{lem}
\label{Lemma: H.1. Lap Bai97}For $\varpi>3/4$, we have $\lim_{T\rightarrow\infty}\limsup_{\left|s\right|\rightarrow\infty}\left|\widehat{\mathscr{W}}_{T}\left(s\right)\right|/\left|s\right|^{\varpi}=0$,
$\mathbb{P}$-a.s.
\end{lem}
\begin{proof}
For any $\epsilon>0,$ if we can show that
\begin{align}
\sum_{i=1}^{\infty}\mathbb{P}\left[\sup_{i-1\leq\left|s\right|<i}\left|\widehat{\mathscr{W}}_{T}\left(s\right)\right|/\left|s\right|^{\varpi}>\epsilon\right] & <\infty,\label{eq. (62)}
\end{align}
 then by the Borel-Cantelli lemma, $\mathbb{P}\left[\limsup_{\left|s\right|\rightarrow\infty}\left|\widehat{\mathscr{W}}_{T}\left(s\right)\right|/\left|s\right|^{\varpi}>\epsilon\right]=0$.
Proceeding as in the proof of Lemma \ref{Lemma B.5, LapBai97}, 
\begin{align*}
\mathbb{P}\left[\sup_{i-1\leq\left|s\right|<i}\left|\widehat{\mathscr{W}}_{T}\left(s\right)\right|/\left|s\right|^{\varpi}>\epsilon\right] & \leq\mathbb{P}\left[\sup_{\left|s\right|\leq1}\left|\widehat{\mathscr{W}}_{T}\left(s\right)\right|>\epsilon i^{\varpi-1/2}\right]\\
 & \leq\frac{1}{\epsilon^{4}}\mathbb{E}\left[\mathbb{E}\left(\sup_{\left|s\right|\leq1}\left(\widehat{\mathscr{W}}_{T}\left(s\right)\right)^{4}|\,\widehat{\varSigma}_{T}\right)\right]\frac{1}{i^{4\varpi-2}}.
\end{align*}
The series $\sum_{i=1}^{\infty}i^{-p}$ is a Riemann's zeta function
and satisfies $\sum_{i=1}^{\infty}i^{-p}<\infty$ if $p>1.$ Then,
\begin{align}
\sum_{i=1}^{\infty}\mathbb{P}\left[\sup_{i-1\leq\left|s\right|<i}\left|\widehat{\mathscr{W}}_{T}\left(s\right)\right|/\left|s\right|^{\varpi}>\epsilon\right] & \leq\left(C/\epsilon^{4}\right)\mathbb{E}\left[\mathbb{E}\left(\sup_{\left|s\right|\leq1}\left(\widehat{\mathscr{W}}_{T}\left(s\right)\right)^{4}|\,\widehat{\varSigma}_{T}\right)\right]\nonumber \\
 & \leq\left(C/\epsilon^{4}\right)\mathbb{E}\left[\mathbb{E}\left(\sup_{\left|s\right|\leq1}\widehat{\mathscr{W}}_{T}\left(s\right)|\,\widehat{\varSigma}_{T}\right)\right]^{4},\label{Eq. (63) Lap Bai97}
\end{align}
where $C>0$ and the last inequality follows from Proposition A.2.4
in \citet{vaart/wellner:96}. The process $\widehat{\mathscr{W}}_{T}$,
conditional on $\widehat{\varSigma}_{T}$, is sub-Gaussian with respect
to the semimetric $d_{VW}^{2}\left(t,\,s\right)=\widehat{\varSigma}_{T}\left(t,\,t\right)+\widehat{\varSigma}_{T}\left(s,\,s\right)$,
which by invoking Assumption \ref{Assumption Laplace Inference LapBai97}-(ii,iii)
is bounded by 
\begin{align*}
\widehat{\varSigma}_{T}\left(t-s,\,t-s\right)\leq\left|t-s\right|\sup_{\left|s\right|=1}\widehat{\varSigma}_{T}\left(s,\,s\right) & .
\end{align*}
Theorem 2.2.8 in \citet{vaart/wellner:96} then implies 
\begin{align*}
\mathbb{E}\left(\sup_{\left|s\right|\leq1}\widehat{\mathscr{W}}_{T}\left(s\right)|\,\widehat{\varSigma}_{T}\right) & \leq C\sup_{\left|s\right|=1}\widehat{\varSigma}_{T}^{1/2}\left(s,\,s\right).
\end{align*}
The above inequality can be used into the right-hand side of \eqref{Eq. (63) Lap Bai97}
to deduce that the latter is bounded by $C\mathbb{E}\left(\sup_{\left|s\right|=1}\widehat{\varSigma}_{T}^{2}\left(s,\,s\right)\right)$.
By Assumption \ref{Assumption Laplace Inference LapBai97}-(iv) $C\mathbb{E}\left(\sup_{\left|s\right|=1}\widehat{\varSigma}_{T}^{2}\left(s,\,s\right)\right)<\infty,$
and the proof is concluded. 
\end{proof}
\begin{lem}
\label{Lemma H.3 Lap Bai97}$\left\{ \widehat{\mathscr{W}}_{T}\right\} $
converges weakly toward $\mathscr{W}$ on compact subsets of $\mathbb{D}_{b}$. 
\end{lem}
\begin{proof}
By the definition of $\widehat{\mathscr{W}}_{T}\left(\cdot\right)$,
we have the finite-dimensional convergence in distribution of $\widehat{\mathscr{W}}_{T}$
toward $\mathscr{W}.$ Hence, it remains to show the (asymptotic)
stochastic equicontinuity of the sequence of processes $\left\{ \widehat{\mathscr{W}}_{T},\,T\geq1\right\} $.
Let $\mathbf{C}\subset\mathbb{R}_{+}$ be any compact set. Fix any
$\eta>0$ and $\epsilon>0$. We show that for any positive sequence
$\left\{ d_{T}\right\} $, with $d_{T}\downarrow0$, and for every
$t,\,s\in\mathbf{C}$,
\begin{align}
\limsup_{T\rightarrow\infty}\mathbb{P}\left(\sup_{\left|t-s\right|<d_{T}}\left|\widehat{\mathscr{W}}_{T}\left(t\right)-\widehat{\mathscr{W}}_{T}\left(s\right)\right|>\eta\right) & <\epsilon.\label{eq. (64) Lap Bai97}
\end{align}
By Markov's inequality, $\mathbb{P}\left(\sup_{\left|t-s\right|<d_{T}}\left|\widehat{\mathscr{W}}_{T}\left(t\right)-\widehat{\mathscr{W}}_{T}\left(s\right)\right|>\eta\right)\leq\mathbb{E}\left(\sup_{\left|t-s\right|<d_{T}}\left|\widehat{\mathscr{W}}_{T}\left(t\right)-\widehat{\mathscr{W}}_{T}\left(s\right)\right|\right)/\eta$.
Let $\widehat{\varUpsilon}_{T}\left(t,\,s\right)$ denote the covariance
matrix of $\left(\widehat{\mathscr{W}}_{T}\left(t\right),\,\widehat{\mathscr{W}}_{T}\left(s\right)\right)'$
and $\mathcal{N}$ be a two-dimensional standard normal vector. Letting
$\imath\triangleq\begin{bmatrix}1 & -1\end{bmatrix}',$ we have 
\begin{align*}
\left[\mathbb{E}\sup_{\left|t-s\right|<d_{T}}\left|\widehat{\mathscr{W}}_{T}\left(t\right)-\widehat{\mathscr{W}}_{T}\left(s\right)\right|\right]^{2} & =\left[\mathbb{E}\sup_{\left|t-s\right|<d_{T}}\left|\imath'\widehat{\varUpsilon}_{T}^{1/2}\left(t,\,s\right)\mathcal{N}\right|\right]^{2}\leq\mathbb{E}\left[\sup_{\left|t-s\right|<d_{T}}\iota'\widehat{\varUpsilon}_{T}\left(t,\,s\right)\iota\right]\\
 & =\mathbb{E}\left[\sup_{\left|t-s\right|<d_{T}}\widehat{\varSigma}_{T}\left(t-s,\,t-s\right)\right]\\
 & \leq d_{T}\mathbb{E}\left[\sup_{\left|s\right|=1}\widehat{\varSigma}_{T}\left(s,\,s\right)\right],
\end{align*}
and so $\mathbb{E}\left[\sup_{\left|t-s\right|<d_{T}}\widehat{\varSigma}_{T}\left(t-s,\,t-s\right)\right]\leq2d_{T}\mathbb{E}\left[\sup_{\left|s\right|=1}\widehat{\varSigma}_{T}\left(s,\,s\right)\right]$
where we have used Assumption \ref{Assumption Laplace Inference LapBai97}-(iii)
in the last step. As $d_{T}\downarrow0$ the right-hand side goes
to zero since $\mathbb{E}\left[\sup_{\left|s\right|=1}\widehat{\varSigma}_{T}\left(s,\,s\right)\right]=O\left(1\right)$
by Assumption \ref{Assumption Laplace Inference LapBai97}-(iv).
\end{proof}
\begin{lem}
\label{Lemma H.4 Lap Bai97}Fix $0<a<\infty$. For any $p\in\boldsymbol{P}$
and for any  positive sequence $\left\{ a_{T}\right\} $ satisfying
$a_{T}\overset{\mathbb{P}}{\rightarrow}a$, 
\begin{align*}
\int_{\mathbb{R}}\left|p\left(s\right)\right|\exp\left(\widehat{\mathscr{W}}_{T}\left(s\right)\right)\exp\left(-a_{T}\left|s\right|\right)ds & \overset{d}{\rightarrow}\int_{\mathbb{R}}\left|p\left(s\right)\right|\exp\left(\mathscr{W}\left(s\right)\right)\exp\left(-a\left|s\right|\right)ds.
\end{align*}
\end{lem}
\begin{proof}
Let $\mathbf{B}_{+}$ be a compact subset of $\mathbb{R}_{+}/\left\{ 0\right\} $.
Let 
\begin{align*}
\boldsymbol{G} & =\left\{ \left(W,\,a_{T}\right)\in\mathbb{D}_{b}\left(\mathbb{R},\,\mathscr{B},\,\mathbb{P}\right)\times\mathbf{B}_{+}:\,\limsup_{\left|s\right|\rightarrow\infty}\left|W\left(s\right)\right|/\left|s\right|^{\varpi}=0,\,\varpi>3/4,\,a_{T}=a+o_{\mathbb{P}}\left(1\right)\right\} ,
\end{align*}
 and denote by $f:\,\boldsymbol{G}\rightarrow\mathbb{R}$ the functional
given by $f\left(\boldsymbol{G}\right)=\int\left|p\left(s\right)\right|\exp\left(W\left(s\right)\right)\exp\left(-a_{T}\left|s\right|\right)ds$.
In view of the continuity of $f\left(\cdot\right)$ and $a_{T}\overset{\mathbb{P}}{\rightarrow}a$,
the claim of the lemma follows by Lemma \ref{Lemma: H.1. Lap Bai97}-\ref{Lemma H.3 Lap Bai97}
and the continuous mapping theorem.
\end{proof}
We are now in a position to conclude the proof of Proposition \ref{Proposition Inference Limi Distr LapBai97}.
Suppose $\gamma_{T}=CT\left\Vert \widehat{\delta}_{T}\right\Vert ^{2}$
for some $C>0$. Under mean-squared loss function, $\widehat{\xi}_{T}$
admits a closed form:
\begin{align*}
\widehat{\xi}_{T} & =\frac{\int u\exp\left(\mathscr{\widehat{W}}_{T}\left(u\right)-\widehat{\varLambda}_{T}\left(u\right)\right)du}{\int\exp\left(\mathscr{\widehat{W}}_{T}\left(u\right)-\widehat{\varLambda}_{T}\left(u\right)\right)du}.
\end{align*}
By Lemma \ref{Lemma H.4 Lap Bai97}, we deduce that $\widehat{\xi}_{T}$
converges in law to the distribution stated in \eqref{eq. Asy Dist Laplace Bai97 part (ii)}.
For general loss functions, a result corresponding to Lemma \ref{Lemma H.4 Lap Bai97}
can be shown to hold since $l\left(\cdot\right)$ is assumed to be
continuous.

\subsection{\label{subsec:Proofs-Multi}Proofs of Section \ref{Section Models with Multiple Breaks}}

Rewrite the GL estimator $\widehat{\boldsymbol{\lambda}}_{b}^{\mathrm{GL}}$
as the minimizer of
\begin{align}
\mathcal{R}_{l,T} & \triangleq\int_{\varGamma^{0}}l\left(s-\boldsymbol{\lambda}_{b}\right)\frac{\exp\left(-Q_{T}\left(\delta\left(\boldsymbol{\lambda}_{b}\right),\,\boldsymbol{\lambda}_{b}\right)\right)\pi\left(\boldsymbol{\lambda}_{b}\right)}{\int_{\varGamma^{0}}\exp\left(-Q_{T}\left(\delta\left(\boldsymbol{\lambda}_{b}\right),\,\boldsymbol{\lambda}_{b}\right)\right)\pi\left(\boldsymbol{\lambda}_{b}\right)d\boldsymbol{\lambda}_{b}}d\boldsymbol{\lambda}_{b}.\label{Eq. (14) - Definition of Psi(s,v,v) BP98}
\end{align}
We show with the following lemma that, for each $i$, $\widehat{\lambda}_{i}^{\mathrm{GL}}\overset{\mathbb{P}}{\rightarrow}\lambda_{i}^{0}$
no matter whether the magnitude of the shifts is fixed or not. Then,
the proof of Theorem \ref{Theorem Geneal Laplace Estimator LapBai97}
can be repeated for each $i=1,\ldots,\,m$ separately. We begin with
the proof for the case of fixed shifts.
\begin{lem}
\label{Theorem 5.2 I=000026H LapBai97-1}Under Assumptions \ref{Assumption A1-A6 BP98}-\ref{Ass A7 BP98},
except that $\Delta_{T,i}=\Delta_{i}^{0}$ for all $i$, for $l\in\boldsymbol{L}$
and any $B>0$ and $\varepsilon>0$, we have for all large $T$, $\mathbb{P}\left[\left|\widehat{\lambda}_{i}^{\mathrm{GL}}-\lambda_{i}^{0}\right|>B\right]<\varepsilon$
for each $i$. 
\end{lem}
\begin{proof}
Let $S_{T}\left(\delta\left(\boldsymbol{\lambda}_{b}\right),\,\boldsymbol{\lambda}_{b}\right)\triangleq Q_{T}\left(\delta\left(\boldsymbol{\lambda}_{b}\right),\,\boldsymbol{\lambda}_{b}\right)-Q_{T}\left(\delta\left(\boldsymbol{\lambda}_{b}^{0}\right),\,\boldsymbol{\lambda}_{b}^{0}\right)$.
Without loss of generality, we assume there are only three change-points
and provide a proof by contradiction for the consistency result. In
particular, we suppose that all but the second change-point are consistently
estimated. That is, consider the case $T_{2}<T_{2}^{0}$ and for some
finite $C>0$ assume that $\left|\lambda_{2}-\lambda_{2}^{0}\right|>C.$
 $Q_{T}\left(\delta\left(\boldsymbol{\lambda}_{b}\right),\,\boldsymbol{\lambda}_{b}\right)$
can be decomposed as,
\begin{align*}
Q_{T}\left(\delta\left(\boldsymbol{\lambda}_{b}\right),\,\boldsymbol{\lambda}_{b}\right) & =\sum_{t=1}^{T}e_{t}^{2}+\sum_{t=1}^{T}d_{t}^{2}-2\sum_{t=1}^{T}e_{t}d_{t},
\end{align*}
where $d_{t}=w'_{t}\left(\widehat{\phi}-\phi^{0}\right)+z'_{t}\left(\widehat{\delta}_{k}-\delta_{j}^{0}\right),$
for $t\in\left[\widehat{T}_{k-1}+1,\,\widehat{T}_{k}\right]\cap\left[T_{j-1}^{0}+1,\,T_{j}^{0}\right]$
$\left(k,\,j=1,\ldots,\,m+1\right)$ where $\widehat{\phi}$ and $\widehat{\delta}_{k}$
are asymptotically equivalent to the corresponding least-squares estimates.
\citet{bai/perron:98} showed that
\[
T^{-1}\sum_{t=1}^{T}d_{t}^{2}\overset{\mathbb{P}}{\rightarrow}K>0\qquad\mathrm{and}\qquad T^{-1}\sum_{t=1}^{T}e_{t}d_{t}=o_{\mathbb{P}}\left(1\right).
\]
Note that $Q_{T}\left(\delta\left(\boldsymbol{\lambda}_{b}^{0}\right),\,\boldsymbol{\lambda}_{b}^{0}\right)=S_{T}\left(T_{1}^{0},\,T_{2}^{0},\,T_{3}^{0}\right)$,
where $S_{T}\left(T_{1}^{0},\,T_{2}^{0},\,T_{3}^{0}\right)$ denotes
the sum of squared residuals evaluated at $\left(T_{1}^{0},\,T_{2}^{0},\,T_{3}^{0}\right).$
Since $T^{-1}S_{T}\left(T_{1}^{0},\,T_{2}^{0},\,T_{3}^{0}\right)$
is asymptotically equivalent to $T^{-1}\sum_{t=1}^{T}e_{t}^{2},$
this implies that $T^{-1}S_{T}\left(\delta\left(\boldsymbol{\lambda}_{b}\right),\,\boldsymbol{\lambda}_{b}\right)>0$
for all large $T$. For some finite $K>0$, this implies
\begin{align}
S_{T}\left(\delta\left(\boldsymbol{\lambda}_{b}\right),\,\boldsymbol{\lambda}_{b}\right) & \geq TK.\label{Eq: S_T}
\end{align}
Let $\mathbf{U}_{T}\triangleq\left\{ u\in\mathbb{R}:\,\boldsymbol{\lambda}_{b}^{0}+T^{-1}u\in\varGamma^{0}\right\} $.
Define $p_{T}\left(u\right)\triangleq p_{1,T}\left(u\right)/\overline{p}_{T}$
where $p_{1,T}\left(u\right)=\exp\left(-Q_{T}\left(\delta\left(u\right),\,u\right)\right)$
and $\overline{p}_{T}\triangleq\int_{\mathbf{U}_{T}}p_{1,T}\left(w\right)dw$.
By definition, $\widehat{\boldsymbol{\lambda}}_{b}^{\mathrm{GL}}$
is the minimum of the function $\int_{\varGamma^{0}}l\left(s-u\right)p_{1,T}\left(u\right)\pi\left(u\right)du$
with $s\in\varGamma^{0}$. Upon using a change in variables,
\begin{align*}
\int_{\varGamma^{0}} & l\left(s-u\right)p_{1,T}\left(u\right)\pi\left(u\right)du\\
 & =T^{-1}\overline{p}_{T}\int_{\mathbf{U}_{T}}l\left(T\left(s-\boldsymbol{\lambda}_{b}^{0}\right)-u\right)p_{T}\left(\boldsymbol{\lambda}_{b}^{0}+T^{-1}u\right)\pi\left(\boldsymbol{\lambda}_{b}^{0}+T^{-1}u\right)du.
\end{align*}
 Thus, $\boldsymbol{\lambda}_{\delta,T}\triangleq T\left(\widehat{\boldsymbol{\lambda}}_{b}^{\mathrm{GL}}-\boldsymbol{\lambda}_{b}^{0}\right)$
is the\textbf{ }minimum of the function, 
\begin{align*}
\mathcal{S}_{T}\left(s\right) & \triangleq\int_{\mathbf{U}_{T}}l\left(s-u\right)\frac{p_{T}\left(\boldsymbol{\lambda}_{b}^{0}+T^{-1}u\right)\pi\left(\boldsymbol{\lambda}_{b}^{0}+T^{-1}u\right)}{\int_{\mathbf{U}_{T}}p_{T}\left(\boldsymbol{\lambda}_{b}^{0}+T^{-1}w\right)\pi\left(\boldsymbol{\lambda}_{b}^{0}+T^{-1}w\right)dw}du,
\end{align*}
where the optimization is over $\mathbf{U}_{T}$. As in the proof
of Lemma \ref{Lemma Consistency Fixed delta}, we exploit the following
relationship,  
\begin{align}
\mathbb{P}\left[\left|\widehat{\boldsymbol{\lambda}}_{b}^{\mathrm{GL}}-\boldsymbol{\lambda}_{b}^{0}\right|>B\right] & \leq\mathbb{P}\left[\inf_{\left|s\right|>TB}\mathcal{S}_{T}\left(s\right)\leq\mathcal{S}_{T}\left(0\right)\right].\label{eq. AA5 LapBai97-1}
\end{align}
Thus, we need to show that the random function $\mathcal{S}_{T}\left(s\right)$
is strictly larger than $\mathcal{S}_{T}\left(0\right)$ on $\left\{ \left|s\right|>TB\right\} $
with high probability as $T\rightarrow\infty$. The same steps as
in Lemma \ref{Lemma Consistency Fixed delta} lead to,
\begin{align}
\mathcal{S}_{T}\left(0\right) & -\inf_{\left|s\right|>TB}\mathcal{S}_{T}\left(s\right)\label{eq. (5.14) I=000026H LapBai97-1}\\
 & \leq-\varpi\int_{\Gamma_{1,T}}p_{T}\left(u\right)du+\int_{\mathbf{U}_{T}\cap\left(\left|u\right|>\left(TB/2\right)^{\vartheta}\right)}l_{T}\left(u\right)p_{T}\left(u\right)du.\nonumber 
\end{align}
We can use the relationship \eqref{Eq: S_T} in place of \eqref{eq" S_bar goe to - inf}
in Lemma \ref{Lemma Consistency Fixed delta} to show that the second
term above converges to zero. The first term is negative using the
same argument as in Lemma \ref{Lemma Consistency Fixed delta}. Thus,
$\mathcal{S}_{T}\left(0\right)-\inf_{\left|s\right|>TB}\mathcal{S}_{T}\left(s\right)<0$.
This gives a contradiction to the fact that $\widehat{\boldsymbol{\lambda}}_{b}^{\mathrm{GL}}$
minimizes $\int_{\varGamma^{0}}l\left(s-u\right)p_{1,T}\left(u\right)\pi\left(u\right)du$.
Hence, each change-point is consistently estimated. 
\end{proof}
\begin{lem}
\label{Lemma mul Cons Small Break}Under Assumptions \ref{Assumption A1-A6 BP98}-\ref{Ass A7 BP98},
for $l\in\boldsymbol{L}$ and any $B>0$ and $\varepsilon>0$, we
have for all large $T$, $\mathbb{P}\left[\left|\widehat{\lambda}_{i}^{\mathrm{GL}}-\lambda_{i}^{0}\right|>B\right]<\varepsilon$
for each $i$. 
\end{lem}
\begin{proof}
The structure of the proof is similar to that of Lemma \ref{Theorem 5.2 I=000026H LapBai97-1}.
The difference consists on the fact that now $T^{-1}\sum_{t=1}^{T}d_{t}^{2}\overset{\mathbb{P}}{\rightarrow}0$
even when a break is not consistently estimated. However, \citet{bai/perron:98}
showed that $T^{-1}\sum_{t=1}^{T}d_{t}^{2}>2T^{-1}\sum_{t=1}^{T}e_{t}d_{t}$
and thus one can proceed as in the aforementioned proof to complete
the proof.
\end{proof}
\begin{lem}
\label{Lemma mul Cons Small Break-1}Under Assumptions \ref{Assumption A1-A6 BP98}-\ref{Ass A7 BP98},
for $l\in\boldsymbol{L}$ and for every $\varepsilon>0$ there exists
a $B<\infty$ such that for all large $T$, $\mathbb{P}\left[Tv_{T}^{2}\left|\widehat{\lambda}_{i}^{\mathrm{GL}}-\lambda_{i}^{0}\right|>B\right]<\varepsilon$
for each $i$.
\end{lem}
\begin{proof}
Let $S_{T}\left(\delta\left(\boldsymbol{\lambda}_{b}\right),\,\boldsymbol{\lambda}_{b}\right)\triangleq Q_{T}\left(\delta\left(\boldsymbol{\lambda}_{b}\right),\,\boldsymbol{\lambda}_{b}\right)-Q_{T}\left(\delta\left(\boldsymbol{\lambda}_{b}^{0}\right),\,\boldsymbol{\lambda}_{b}^{0}\right)$.
Without loss of generality, we assume there are only three change-points
and provide an explicit proof only for $\lambda_{2}^{0}$. We use
the same notation as in \citet{bai/perron:98}, pp. 69-70. Note that
their results concerning the estimates of the regression parameters
can be used in our context because once we have the consistency of
the fractional change-points the estimates of the regression parameters
are asymptotically equivalent to the corresponding least-squares estimates.
For each $\epsilon>0,$ let $V_{\epsilon}=\left\{ \left(T_{1},\,T_{2},\,T_{3}\right);\,\left|\widehat{T}_{i}^{\mathrm{}}-T_{i}^{0}\right|\leq\epsilon T,\,i=1\leq i\leq3\right\} $.
By the consistency result, for each $\epsilon>0$ and $T$ large,
we have $\left|\widehat{T}_{i}^{\mathrm{}}-T_{i}^{0}\right|\leq\epsilon T$,
where $\widehat{T}_{i}=\widehat{T}_{i}^{\mathrm{GL}}=T\widehat{\lambda}_{i}^{\mathrm{GL}}$.
Hence, $\mathbb{P}\left(\left\{ \widehat{T}_{1},\,\widehat{T}_{2},\,\widehat{T}_{3}\right\} \in V_{\epsilon}\right)\rightarrow1$
with high probability. Therefore we only need to examine the behavior
of $S_{T}\left(\delta\left(\boldsymbol{\lambda}_{b}\right),\,\boldsymbol{\lambda}_{b}\right)$
for those $T_{i}$ that are close to the true break dates such that
$\left|T_{i}-T_{i}^{0}\right|<\epsilon T$ for all $i$. By symmetry,
we can, without loss of generality, consider the case $T_{2}<T_{2}^{0}$.
For $C>0$, define 
\begin{align*}
V_{\epsilon}^{*}\left(C\right) & =\left\{ \left(T_{1},\,T_{2},\,T_{3}\right);\,\left|\widehat{T}_{i}-T_{i}^{0}\right|<\epsilon T,\,1\leq i\leq3,\,T_{2}-T_{2}^{0}<-C/v_{T}^{2}\right\} .
\end{align*}
Define the sum of squared residuals evaluated at $\left(T_{1},\,T_{2},\,T_{3}\right)$
by $S_{T}\left(T_{1},\,T_{2},\,T_{3}\right)$. Let $SSR_{1}=S_{T}\left(T_{1},\,T_{2},\,T_{3}\right),$
$SSR_{2}=S_{T}\left(T_{1},\,T_{2}^{0},\,T_{3}\right)$ and $SSR_{3}=S_{T}\left(T_{1},\,T_{2},\,T_{2}^{0},\,T_{3}\right)$.
We have omitted the dependence on $\delta.$ With this notation, we
have $S_{T}\left(\delta\left(\boldsymbol{\lambda}_{b}\right),\,\boldsymbol{\lambda}_{b}\right)=S_{T}\left(T_{1},\,T_{2},\,T_{3}\right)-S_{T}\left(T_{1}^{0},\,T_{2}^{0},\,T_{3}^{0}\right)$
which can be decomposed as 
\begin{align}
S_{T} & \left(\delta\left(\boldsymbol{\lambda}_{b}\right),\,\boldsymbol{\lambda}_{b}\right)\label{Eq. (29) in BP98}\\
 & =\left[\left(SSR_{1}-SSR_{3}\right)-\left(SSR_{2}-SSR_{3}\right)\right]+\left(SSR_{2}-S_{T}\left(T_{1}^{0},\,T_{2}^{0},\,T_{3}^{0}\right)\right).\nonumber 
\end{align}
 In their Proposition 4-(ii), \citet{bai/perron:98} showed that the
first term on the right-hand side above satisfies the following: for
every $\varepsilon>0$, there exists $B>0$ and $\epsilon>0$ such
that for large $T$,
\begin{align*}
\mathbb{P}\left[\min\left\{ \left[S_{T}\left(T_{1},\,T_{2},\,T_{3}\right)-S_{T}\left(T_{1},\,T_{2}^{0},\,T_{3}\right)\right]/\left(T_{2}^{0}-T_{2}\right)\right\} \leq0\right] & <\varepsilon,
\end{align*}
 where the minimum is taken over $V_{\epsilon}^{*}\left(C\right)$.
The second term of \eqref{Eq. (29) in BP98} divided by $T_{2}^{0}-T_{2}$
can be shown to be negligible for $\left\{ T_{1},\,T_{2},\,T_{3}\right\} \in V_{\epsilon}^{*}\left(C\right)$
and $C$ large enough because on $V_{\epsilon}^{*}\left(C\right)$
the consistency result guarantees that $\widehat{\lambda}_{i}$ can
be made arbitrary close to $\lambda_{i}^{0}$. This leads to a result
similar to \eqref{Eq: S_T} where $T$ is replaced by $v_{T}^{-2}$.
Then one can continue with the same argument used in the second part
of the proof of Lemma \ref{Theorem 5.2 I=000026H LapBai97-1}. 
\end{proof}

\subsection{Proofs of Section \ref{Section Theoretical-Properties-of GL Inference}}

\subsubsection{Proof of Proposition \ref{Proposition Bet Proof}}

Let 
\begin{align*}
p_{1,T}\left(y|\,\lambda_{b}^{0}+\psi_{T}^{-1}u\right)\triangleq\exp\left(\left(\widetilde{G}_{T,0}\left(u,\,0\right)+Q_{T,0}\left(u\right)\right)/2\right) & ,
\end{align*}
 where $\widetilde{G}_{T,0}\left(u,\,0\right)$ and $Q_{T,0}\left(u\right)$
were defined in equation \eqref{eq. (17)-1-1}. Let $p_{1}\left(y|\,\lambda_{b}\right)\triangleq\exp\left(\left(L^{2}\left(\lambda_{b}\right)-L^{2}\left(\lambda_{0}\right)\right)/2\right)$
where $L\left(\lambda_{b}\right)=\left(T_{b}\left(T-T_{b}\right)\right)^{1/2}\left(\overline{Y}_{T_{b}}^{*}-\overline{Y}_{T_{b}}\right)$
with $\overline{Y}_{T_{b}}=T_{b}^{-1}\sum_{t=1}^{T_{b}}y_{t}$ and
$\overline{Y}_{T_{b}}^{*}=\left(T-T_{b}\right)^{-1}\sum_{t=T_{b}+1}^{T}y_{t}$.
Following \citet{bhattacharya:94} we use a prior $\check{\pi}\left(\cdot\right)$
on the random variable $\overline{\lambda}_{b}$.  The posterior
distribution of $\overline{\lambda}_{b}=\lambda_{b}$ is given by
$p\left(\lambda_{b}|\,y\right)=h\left(\lambda_{b}\right)/\int_{0}^{1}h\left(s\right)ds$
where  $h\left(\lambda_{b}\right)=p_{1}\left(y|\,\lambda_{b}\right)\check{\pi}\left(\lambda_{b}\right)$.
The total variation distance between two probability measures $\nu_{1}$
and $\nu_{2}$ defined on some probability space $S\in\mathbb{R}$
is denoted as $\left|\nu_{1}-\nu_{2}\right|_{\mathrm{TV}}\triangleq\int_{S}\left|\nu_{1}\left(u\right)-\nu_{2}\left(u\right)\right|du$.
Given the local parameter $\lambda_{b}=\lambda_{b}^{0}+\left(Tv_{T}^{2}\right)^{-1}u$
with $u\in\left[-M,\,M\right]$ for a given $M>0$, the posterior
for $u$ is equal to $p^{*}\left(u|\,y\right)=\left(Tv_{T}^{2}\right)^{-1}p\left(\left(Tv_{T}^{2}\right)^{-1}u+\lambda_{b}^{0}|\,y\right)$
while the quasi-posterior is given by $p_{T}^{*}\left(u|\,y\right)=\left(Tv_{T}^{2}\right)^{-1}p_{T}\left(\left(Tv_{T}^{2}\right)^{-1}u+\lambda_{b}^{0}|\,y\right)$. 
\begin{lem}
\label{Lemma Quasi Posterior =00003D Posterior} Let Assumptions \ref{Assumption Prior LapBai97}-\ref{Assumption Small Shift BP}
and \ref{Assumption Gaussian Process for Lap LapBai97}-(i) hold and
$\check{\pi}\left(\cdot\right)$ satisfy Assumption \ref{Assumption Prior LapBai97}.
Then, 
\[
\left|p_{T}^{*}\left(Tv_{T}^{2}\left(\overline{\lambda}_{b}-\lambda_{b}^{0}\right)|\,y\right)-p^{*}\left(Tv_{T}^{2}\left(\overline{\lambda}_{b}-\lambda_{b}^{0}\right)|\,y\right)\right|_{\mathrm{TV}}\overset{\mathbb{P}}{\rightarrow}0.
\]
\end{lem}
\begin{proof}
 By assumption \ref{Assumption Prior LapBai97}, $\pi\left(\cdot\right)$
and $\check{\pi}\left(\cdot\right)$ are bounded, and 
\begin{align*}
\sup_{\left|u\right|\leq M}\left|\pi\left(\left(Tv_{T}^{2}\right)^{-1}u+\lambda_{b}^{0}\right)-\pi\left(\lambda_{b}^{0}\right)\right| & \overset{\mathbb{P}}{\rightarrow}0,\\
\sup_{\left|u\right|\leq M}\left|\check{\pi}\left(\left(Tv_{T}^{2}\right)^{-1}u+\lambda_{b}^{0}\right)-\check{\pi}\left(\lambda_{b}^{0}\right)\right| & \overset{\mathbb{P}}{\rightarrow}0.
\end{align*}
Since $\pi\left(\cdot\right)$ {[}$\check{\pi}\left(\cdot\right)${]}
appears in both the numerator and denominator of $p_{T}^{*}\left(\cdot|\,y\right)$
{[}$p^{*}\left(\cdot|\,y\right)${]}, it cancels from that expression
asymptotically. Turning to the Laplace estimator, the results of Section
\ref{Section Asymptotic Results LapBai97} (see Lemma \ref{Lemma, Area (ia), Bai A.5}
and \ref{Lemma, Area (ib) Q}) imply that for $u\leq0$, using $Q\left(\delta\left(\lambda_{b}\right),\,\lambda_{b}\right)/2$
in place of $Q\left(\delta\left(\lambda_{b}\right),\,\lambda_{b}\right)$,
\begin{align}
\exp & \left(\left(\widetilde{G}_{T,0}\left(u,\,0\right)+Q_{T,0}\left(u\right)\right)/2\right)\label{eq (3)}\\
 & =\exp\left(\delta_{T}\sum_{t=0}^{v_{T}^{-2}\left|u\right|}e_{T_{b}^{0}-t}-\left|u\right|\delta_{0}^{2}/2\right)\left(1+A_{T}\right),\nonumber 
\end{align}
 where $A_{T}=o_{\mathbb{P}}\left(1\right)$ is uniform in the region
$u\leq\eta Tv_{T}^{2}$ for small $\eta>0$. By symmetry, the case
$u>0$ results in the same relationship as \eqref{eq (3)} with $e_{T_{b}^{0}-t}$
replaced by $e_{T_{b}^{0}+t}$. The results in the proof of Theorem
1 in \citet{bai:94a} combined with the arguments referenced for the
derivation of \eqref{eq (3)} suggest that for $u\leq0,$
\begin{align}
\exp & \left(\left(L^{2}\left(\left(Tv_{T}^{2}\right)^{-1}u+\lambda_{b}^{0}\right)-L^{2}\left(\lambda_{b}^{0}\right)\right)/2\right)\label{eq (1)}\\
 & =\exp\left(\delta_{T}\sum_{t=0}^{v_{T}^{-2}\left|u\right|}e_{T_{b}^{0}-t}-\left|u\right|\delta_{0}^{2}/2\right)\left(1+B_{T}\right),\nonumber 
\end{align}
where $B_{T}=o_{\mathbb{P}}\left(1\right)$ is uniform in the region
$u\leq\eta Tv_{T}^{2}$ for small $\eta>0$. By symmetry, the case
$u>0$ results in the same relationship as \eqref{eq (1)} with $e_{T_{b}^{0}-t}$
replaced by $e_{T_{b}^{0}+t}$. By Lemma \ref{Lemma, Area (ii), Q LapBai97}
and the results in \citet{bai:94a}, $p_{T}\left(u|\,y\right)$ and
$p\left(u|\,y\right)$ are negligible uniformly in $u$ for $u>\eta Tv_{T}^{2}$
for every $\eta$. Thus, \eqref{eq (3)}-\eqref{eq (1)} yield,
\begin{align*}
\left|p_{T}^{*}\left(Tv_{T}^{2}\left(\overline{\lambda}_{b}-\lambda_{b}^{0}\right),\,y\right)-p^{*}\left(Tv_{T}^{2}\left(\overline{\lambda}_{b}-\lambda_{b}^{0}\right),\,y\right)\right| & _{\mathrm{TV}}\leq\left|A_{T}\right|+\left|B_{T}\right|\overset{\mathbb{P}}{\rightarrow}0.
\end{align*}
\end{proof}
Continuing with the proof of Proposition \ref{Proposition Bet Proof},
we begin with part (i). Note that $\varphi\left(\lambda_{b},\,y\right)$
is defined by 
\[
\int\left(1-\varphi\left(\lambda_{b},\,y\right)\right)p_{T}\left(y|\,\lambda_{b}\right)d\varPi\left(\lambda_{b}\right)\geq1-\alpha
\]
 for all $y,$ where $\varPi\left(\cdot\right)$ is a probability
measure on $\varGamma^{0}$ such that $\varPi\left(\lambda_{b}\right)=\pi\left(\lambda_{b}\right)d\lambda_{b}$.
The fact that $\left|1-\varphi\left(\lambda_{b},\,y\right)\right|\leq1$
and Lemma \ref{Lemma Quasi Posterior =00003D Posterior} lead to,

\begin{align}
\int & \left(1-\varphi\left(\lambda_{b},\,y\right)\right)p_{T}\left(y|\,\lambda_{b}\right)d\varPi\left(\lambda_{b}\right)\label{eq. PT =00003D P}\\
 & =\int\left(1-\varphi\left(\lambda_{b},\,y\right)\right)p\left(y|\,\lambda_{b}\right)d\varPi\left(\lambda_{b}\right)+o_{\mathbb{P}}\left(1\right).\nonumber 
\end{align}
 Given that Definition \ref{def:Highest-Density-Region} of the GL
confidence interval involves an inequality that explicitly allows
for conservativeness, \eqref{eq. PT =00003D P} implies the following
relationship, 
\begin{align*}
\int\varphi\left(\lambda_{b},\,y\right)p_{T}\left(y|\,\lambda_{b}\right)d\varPi\left(\lambda_{b}\right) & =\int\varphi\left(\lambda_{b},\,y\right)p\left(y|\,\lambda_{b}\right)d\varPi\left(\lambda_{b}\right)+\varepsilon_{T}\leq\alpha\int p\left(y|\,\lambda_{b}\right)d\varPi\left(\lambda_{b}\right),
\end{align*}
where $\varepsilon_{T}=\int\varphi\left(\lambda_{b},\,y\right)\left(p_{T}\left(y|\,\lambda_{b}\right)-p\left(y|\,\lambda_{b}\right)\right)d\varPi\left(\lambda_{b}\right)$.
Rearranging, we have, 
\begin{align*}
\int\left(\alpha-\varphi\left(\lambda_{b},\,y\right)\right)p\left(y|\,\lambda_{b}\right)d\varPi\left(\lambda_{b}\right)-\varepsilon_{T} & \geq0,
\end{align*}
 for all $y.$ Now multiply both sides by $\widetilde{b}\left(y\right)\geq0$
and integrating with respect to $\zeta\left(y\right)$ yields,
\begin{align*}
\int\int\left(\alpha-\varphi\left(\lambda_{b},\,y\right)\right)\widetilde{b}\left(y\right)p\left(y|\,\lambda_{b}\right)d\zeta\left(y\right)d\varPi\left(\lambda_{b}\right)-\varepsilon_{T}\int\widetilde{b}\left(y\right)d\zeta\left(y\right) & \geq0,
\end{align*}
 or
\begin{align*}
\left(1-\alpha\right)\int L_{\alpha}\left(\varphi,\,\widetilde{b},\,\lambda_{b}\right)d\varPi\left(\lambda_{b}\right)-\varepsilon_{T}\int\widetilde{b}\left(y\right)d\zeta\left(y\right) & \geq0.
\end{align*}
Taking the limit as $T\rightarrow\infty$,
\begin{align*}
\left(1-\alpha\right)\int L_{\alpha}\left(\varphi,\,\widetilde{b},\,\lambda_{b}\right)d\varPi\left(\lambda_{b}\right) & \geq0.
\end{align*}
The latter implies that $L_{\alpha}\left(\varphi,\,\widetilde{b},\,\lambda_{b}\right)\geq0$
for some $\lambda_{b}$. Thus, $\varphi$ is bet-proof at level $1-\alpha$.

We now prove part (ii). We use a proof by contradiction. If $\int\varphi'\left(\lambda_{b},\,y\right)d\lambda_{b}\geq\int\varphi\left(\lambda_{b},\,y\right)d\lambda_{b}$
for all $y\in\mathcal{Y}$ and $\int\varphi'\left(\lambda_{b},\,y\right)d\lambda_{b}>\int\varphi\left(\lambda_{b},\,y\right)d\lambda_{b}$
for all $y\in\mathcal{Y}_{0}$ with $\zeta\left(\mathcal{Y}_{0}\right)>0$,
then we show that $\int\varphi'\left(\lambda_{b},\,y\right)p\left(y|\,\lambda_{b}\right)d\zeta\left(y\right)>\alpha$
for some $\lambda_{b}\in\varGamma^{0}.$ By Lemma \ref{Lemma Quasi Posterior =00003D Posterior}
and \eqref{eq (def Quasi-credible set)} holding with equality, 
\begin{align*}
\int\varphi\left(\lambda_{b},\,y\right)p_{T}\left(y|\,\lambda_{b}\right)d\varPi\left(\lambda_{b}\right) & =\alpha\int p_{T}\left(y|\,\lambda_{b}\right)d\varPi\left(\lambda_{b}\right)\\
 & =\alpha\int p\left(y|\,\lambda_{b}\right)d\varPi\left(\lambda_{b}\right)+o_{\mathbb{P}}\left(1\right).
\end{align*}
Integrating both sides with respect to $\zeta\left(y\right)$ yields,
\begin{align}
\int\left(\int\varphi\left(\lambda_{b},\,y\right)p\left(y|\,\lambda_{b}\right)d\zeta\left(y\right)\right)d\varPi\left(\lambda_{b}\right) & =\alpha+o_{\mathbb{P}}\left(1\right).\label{eq: (12, M=000026N)}
\end{align}
 By Assumption \ref{Assumption Prior LapBai97}, $\pi\left(\lambda_{b}\right)>0$
for all $\lambda_{b}\in\varGamma^{0}$. Taking the limit as $T\rightarrow\infty$
of both sides of \eqref{eq: (12, M=000026N)} yields $\int\left(\int\varphi\left(\lambda_{b},\,y\right)p\left(y|\,\lambda_{b}\right)d\zeta\left(y\right)\right)d\varPi\left(\lambda_{b}\right)=\alpha.$
The latter holds only if $\int\varphi\left(\lambda_{b},\,y\right)p\left(y|\,\lambda_{b}\right)d\zeta\left(y\right)=\alpha$
for all $\lambda_{b}\in\varGamma^{0}$. This means that $\varphi$
is similar. The definition of HPD confidence set $\varphi\left(\lambda_{b},\,y\right)$
implies that for $\zeta$-almost all $y$, if $\int\varphi\left(\lambda_{b},\,y\right)d\lambda_{b}=\int\varphi'\left(\lambda_{b},\,y\right)d\lambda_{b}$
then $\int\varphi\left(\lambda_{b},\,y\right)p_{T}\left(\lambda_{b}|\,y\right)d\lambda_{b}\leq\int\varphi'\left(\lambda_{b},\,y\right)p_{T}\left(\lambda_{b}|\,y\right)d\lambda_{b}$.
The latter relationship and Lemma \ref{Lemma Quasi Posterior =00003D Posterior}
imply that,
\begin{align*}
\int\varphi\left(\lambda_{b},\,y\right)p\left(y|\,\lambda_{b}\right)d\varPi\left(\lambda_{b}\right) & \leq\int\varphi'\left(\lambda_{b},\,y\right)p\left(y|\,\lambda_{b}\right)d\varPi\left(\lambda_{b}\right),
\end{align*}
for all $y\in\mathcal{Y}$ and 
\begin{align*}
\int\varphi\left(\lambda_{b},\,y\right)p\left(y|\,\lambda_{b}\right)d\varPi\left(\lambda_{b}\right) & <\int\varphi'\left(\lambda_{b},\,y\right)p\left(y|\,\lambda_{b}\right)d\varPi\left(\lambda_{b}\right),
\end{align*}
 for all $y\in\mathcal{Y}_{0}$. Integrating both sides with respect
to $\zeta$ yields 
\begin{align*}
\int & \left(\int\varphi\left(\lambda_{b},\,y\right)p\left(y|\,\lambda_{b}\right)d\zeta\left(y\right)\right)d\varPi\left(\lambda_{b}\right)\\
 & <\int\left(\int\varphi'\left(\lambda_{b},\,y\right)p\left(y|\,\lambda_{b}\right)d\zeta\left(y\right)\right)d\varPi\left(\lambda_{b}\right),
\end{align*}
 or 
\begin{align*}
\int\left(\int\left(\varphi\left(\lambda_{b},\,y\right)-\varphi'\left(\lambda_{b},\,y\right)\right)p\left(y|\,\lambda_{b}\right)d\zeta\left(y\right)\right)d\varPi\left(\lambda_{b}\right) & <0.
\end{align*}
 Since $\varphi\left(\lambda_{b},\,y\right)$ is similar, there exists
a $\lambda_{b}$ such that $\int\varphi'\left(\lambda_{b},\,y\right)p\left(y|\,\lambda_{b}\right)d\zeta\left(y\right)>\alpha.$
Thus, $\varphi'$ is not of level $1-\alpha.$ $\square$

\vfill{}

\pagebreak{}

\section{\label{Section Comparison-to}Comparison to \citet{casini/perron_Lap_CR_Single_Inf}}

In this section we compare the GL-LN method to the GL estimators/confidence
intervals proposed in \citet{casini/perron_Lap_CR_Single_Inf}. Table
\ref{Table M2 Bias-1}-\ref{Table M3b-1} report the results. We have
considered a data-generating mechanism with higher serial dependence
in the errors. In terms of the empirical performance of the estimators,
Table \ref{Table M2 Bias-1} shows that overall the estimator that
does better is $\widehat{\lambda}_{b}^{\mathrm{GL-LN}}$. $\widehat{\lambda}_{b}^{\mathrm{GL-CR-Iter}}$
is the one that does best when $\lambda_{b}^{0}=0.5$ but it does
worse in relative terms when the break is in the tails. The performance
of $\widehat{\lambda}_{b}^{\mathrm{GL-LN}}$ is in general superior
to $\widehat{\lambda}_{b}^{\mathrm{GL-CR}}$ especially for medium
to large breaks both in terms of MAE and RMSE. From other simulations
(not reported), we conclude that GL-LN does in general better for
moderate to large breaks. $\widehat{\lambda}_{b}^{\mathrm{GL-CR-Iter}}$
is the one that does best when the break is in the middle but its
precision deteriorates as the break moves to the tails. In addition,
$\widehat{\lambda}_{b}^{\mathrm{GL-LN}}$ is valid for models with
multiple breaks and models with trending regressors that are not covered
in \citet{casini/perron_Lap_CR_Single_Inf}. So overall we believe
that the estimators $\widehat{\lambda}_{b}^{\mathrm{GL-LN}}$, $\widehat{\lambda}_{b}^{\mathrm{GL-CR}}$
and $\widehat{\lambda}_{b}^{\mathrm{GL-CR-Iter}}$ can be seen as
complementary. 

Turning to the finite-sample performance of the confidence intervals,
Table \ref{Table M3b-1} clearly shows that when there is higher serial
dependence in the errors, the method that dominates is GL-LN. The
gain in terms of coverage accuracy and lengths can be substantial
relative to the GL-CR and GL-CR-Iter. When the serial dependence in
the errors is low (not reported), the difference in performance of
the three confidence intervals becomes smaller. 

Overall, we find that both estimation and confidence intervals based
on GL-LN perform well relative to the continuous record counterparts,
where major gains appear to occur when there is high serial correlation
in the errors. 

\vfill{}

\pagebreak{}

\setcounter{table}{0} \renewcommand{\thetable}{S-\arabic{table}}

\begin{table}[H]
\caption{\label{Table M2 Bias-1}Small-sample accuracy of the estimates of
the break point $T_{b}^{0}$}

\begin{singlespace}
\begin{centering}
{\footnotesize{}}%
\begin{tabular}{ccccccc|ccccc}
\hline 
 &  & {\footnotesize{}MAE} & {\footnotesize{}Std} & {\footnotesize{}$\textrm{RMSE}$} & {\footnotesize{}$Q_{0.25}$} & \multicolumn{1}{c}{{\footnotesize{}$Q_{0.75}$}} & {\footnotesize{}MAE} & {\footnotesize{}Std} & {\footnotesize{}$\textrm{RMSE}$} & {\footnotesize{}$Q_{0.25}$} & {\footnotesize{}$Q_{0.75}$}\tabularnewline
\cline{3-12} \cline{4-12} \cline{5-12} \cline{6-12} \cline{7-12} \cline{8-12} \cline{9-12} \cline{10-12} \cline{11-12} \cline{12-12} 
 &  & \multicolumn{5}{c|}{{\footnotesize{}$\lambda_{0}=0.3$}} & \multicolumn{5}{c}{{\footnotesize{}$\lambda_{0}=0.5$}}\tabularnewline
{\footnotesize{}$\delta^{0}=0.3$} & {\footnotesize{}OLS} & {\footnotesize{}26.84} & {\footnotesize{}28.12} & {\footnotesize{}33.00} & {\footnotesize{}21} & {\footnotesize{}76} & {\footnotesize{}23.02} & {\footnotesize{}26.86} & {\footnotesize{}26.76} & {\footnotesize{}25} & {\footnotesize{}75}\tabularnewline
 & {\footnotesize{}GL-LN} & {\footnotesize{}13.63} & {\footnotesize{}14.07} & {\footnotesize{}17.25} & {\footnotesize{}27} & {\footnotesize{}56} & {\footnotesize{}10.84} & {\footnotesize{}13.03} & {\footnotesize{}14.40} & {\footnotesize{}35} & {\footnotesize{}65}\tabularnewline
 & {\footnotesize{}GL-CR} & {\footnotesize{}12.79} & {\footnotesize{}13.13} & {\footnotesize{}18.46} & {\footnotesize{}29} & {\footnotesize{}57} & {\footnotesize{}11.84} & {\footnotesize{}13.17} & {\footnotesize{}13.12} & {\footnotesize{}35} & {\footnotesize{}65}\tabularnewline
 & {\footnotesize{}GL-CR-Iter} & {\footnotesize{}14.47} & {\footnotesize{}10.29} & {\footnotesize{}20.21} & {\footnotesize{}28} & {\footnotesize{}58} & {\footnotesize{}8.76} & {\footnotesize{}10.01} & {\footnotesize{}10.24} & {\footnotesize{}41} & {\footnotesize{}59}\tabularnewline
 & {\footnotesize{}GL-Uni} & {\footnotesize{}21.78} & {\footnotesize{}21.73} & {\footnotesize{}27.71} & {\footnotesize{}28} & {\footnotesize{}66} & {\footnotesize{}17.84} & {\footnotesize{}20.90} & {\footnotesize{}20.98} & {\footnotesize{}32} & {\footnotesize{}68}\tabularnewline
{\footnotesize{}$\delta^{0}=0.4$} & {\footnotesize{}OLS} & {\footnotesize{}23.62} & {\footnotesize{}26.99} & {\footnotesize{}30.23} & {\footnotesize{}21} & {\footnotesize{}70} & {\footnotesize{}21.23} & {\footnotesize{}25.43} & {\footnotesize{}25.44} & {\footnotesize{}25} & {\footnotesize{}75}\tabularnewline
 & {\footnotesize{}GL-LN} & {\footnotesize{}11.53} & {\footnotesize{}13.66} & {\footnotesize{}15.44} & {\footnotesize{}27} & {\footnotesize{}51} & {\footnotesize{}10.11} & {\footnotesize{}12.15} & {\footnotesize{}13.37} & {\footnotesize{}37} & {\footnotesize{}63}\tabularnewline
 & {\footnotesize{}GL-CR} & {\footnotesize{}16.36} & {\footnotesize{}13.86} & {\footnotesize{}21.49} & {\footnotesize{}29} & {\footnotesize{}61} & {\footnotesize{}11.56} & {\footnotesize{}11.97} & {\footnotesize{}12.25} & {\footnotesize{}36} & {\footnotesize{}64}\tabularnewline
 & {\footnotesize{}GL-CR-Iter} & {\footnotesize{}17.19} & {\footnotesize{}10.81} & {\footnotesize{}20.35} & {\footnotesize{}28} & {\footnotesize{}57} & {\footnotesize{}8.30} & {\footnotesize{}9.95} & {\footnotesize{}10.01} & {\footnotesize{}43} & {\footnotesize{}57}\tabularnewline
 & {\footnotesize{}GL-Uni} & {\footnotesize{}20.18} & {\footnotesize{}21.25} & {\footnotesize{}26.30} & {\footnotesize{}28} & {\footnotesize{}64} & {\footnotesize{}16.53} & {\footnotesize{}19.97} & {\footnotesize{}19.98} & {\footnotesize{}34} & {\footnotesize{}64}\tabularnewline
{\footnotesize{}$\delta^{0}=0.6$} & {\footnotesize{}OLS} & {\footnotesize{}19.80} & {\footnotesize{}24.62} & {\footnotesize{}26.25} & {\footnotesize{}21} & {\footnotesize{}57} & {\footnotesize{}17.34} & {\footnotesize{}22.39} & {\footnotesize{}22.34} & {\footnotesize{}37} & {\footnotesize{}65}\tabularnewline
 & {\footnotesize{}GL-LN} & {\footnotesize{}8.86} & {\footnotesize{}11.63} & {\footnotesize{}12.77} & {\footnotesize{}29} & {\footnotesize{}42} & {\footnotesize{}8.05} & {\footnotesize{}10.29} & {\footnotesize{}11.18} & {\footnotesize{}41} & {\footnotesize{}59}\tabularnewline
 & {\footnotesize{}GL-CR} & {\footnotesize{}12.84} & {\footnotesize{}13.66} & {\footnotesize{}18.23} & {\footnotesize{}30} & {\footnotesize{}56} & {\footnotesize{}9.96} & {\footnotesize{}11.93} & {\footnotesize{}11.99} & {\footnotesize{}38} & {\footnotesize{}58}\tabularnewline
 & {\footnotesize{}GL-CR-Iter} & {\footnotesize{}14.85} & {\footnotesize{}11.52} & {\footnotesize{}17.56} & {\footnotesize{}29} & {\footnotesize{}52} & {\footnotesize{}7.26} & {\footnotesize{}9.20} & {\footnotesize{}9.22} & {\footnotesize{}44} & {\footnotesize{}55}\tabularnewline
 & {\footnotesize{}GL-Uni} & {\footnotesize{}16.04} & {\footnotesize{}20.05} & {\footnotesize{}22.77} & {\footnotesize{}26} & {\footnotesize{}56} & {\footnotesize{}13.85} & {\footnotesize{}17.81} & {\footnotesize{}17.94} & {\footnotesize{}38} & {\footnotesize{}60}\tabularnewline
{\footnotesize{}$\delta^{0}=1$} & {\footnotesize{}OLS} & {\footnotesize{}11.69} & {\footnotesize{}18.43} & {\footnotesize{}19.26} & {\footnotesize{}27} & {\footnotesize{}40} & {\footnotesize{}9.38} & {\footnotesize{}14.40} & {\footnotesize{}14.40} & {\footnotesize{}46} & {\footnotesize{}54}\tabularnewline
 & {\footnotesize{}GL-LN} & {\footnotesize{}5.63} & {\footnotesize{}9.56} & {\footnotesize{}9.57} & {\footnotesize{}27} & {\footnotesize{}31} & {\footnotesize{}5.40} & {\footnotesize{}8.21} & {\footnotesize{}8.59} & {\footnotesize{}49} & {\footnotesize{}51}\tabularnewline
 & {\footnotesize{}GL-CR} & {\footnotesize{}6.82} & {\footnotesize{}10.85} & {\footnotesize{}12.81} & {\footnotesize{}27} & {\footnotesize{}38} & {\footnotesize{}6.96} & {\footnotesize{}9.43} & {\footnotesize{}9.52} & {\footnotesize{}44} & {\footnotesize{}53}\tabularnewline
 & {\footnotesize{}GL-CR-Iter} & {\footnotesize{}10.67} & {\footnotesize{}7.54} & {\footnotesize{}13.02} & {\footnotesize{}30} & {\footnotesize{}39} & {\footnotesize{}4.44} & {\footnotesize{}6.71} & {\footnotesize{}6.85} & {\footnotesize{}47} & {\footnotesize{}53}\tabularnewline
 & {\footnotesize{}GL-Uni} & {\footnotesize{}9.44} & {\footnotesize{}14.60} & {\footnotesize{}15.15} & {\footnotesize{}27} & {\footnotesize{}37} & {\footnotesize{}8.17} & {\footnotesize{}12.34} & {\footnotesize{}12.34} & {\footnotesize{}45} & {\footnotesize{}54}\tabularnewline
\hline 
\end{tabular}{\footnotesize\par}
\par\end{centering}
\end{singlespace}
\noindent\begin{minipage}[t]{1\columnwidth}%
{\scriptsize{}The model is $y_{t}=\delta_{1}^{0}+\delta^{0}\mathbf{1}_{\left\{ t>\left\lfloor T\lambda_{0}\right\rfloor \right\} }+e_{t},\,e_{t}=0.6e_{t-1}+u_{t},\,u_{t}\sim i.i.d.\,\mathscr{N}\left(0,\,0.49\right),\,T=100$.}%
\end{minipage}
\end{table}

\begin{table}[H]
\caption{\label{Table M3b-1}Small-sample coverage rates and lengths of the
confidence sets}

\begin{centering}
{\footnotesize{}}%
\begin{tabular}{cccccccc}
\hline 
 &  & \multicolumn{2}{c}{{\footnotesize{}$\delta^{0}=0.4$}} & \multicolumn{2}{c}{{\footnotesize{}$\delta^{0}=0.8$}} & \multicolumn{2}{c}{{\footnotesize{}$\delta^{0}=1.6$}}\tabularnewline
 &  & {\footnotesize{}$\textrm{Cov.}$} & {\footnotesize{}$\textrm{Lgth.}$} & {\footnotesize{}$\textrm{Cov.}$} & {\footnotesize{}$\textrm{Lgth.}$} & {\footnotesize{}$\textrm{Cov.}$} & {\footnotesize{}$\textrm{Lgth.}$}\tabularnewline
\cline{3-8} \cline{4-8} \cline{5-8} \cline{6-8} \cline{7-8} \cline{8-8} 
{\footnotesize{}$\lambda_{0}=0.5$} & {\footnotesize{}OLS-CR} & {\footnotesize{}0.910} & {\footnotesize{}67.57} & {\footnotesize{}0.911} & {\footnotesize{}68.87} & {\footnotesize{}0.945} & {\footnotesize{}42.30}\tabularnewline
 & {\footnotesize{}Bai (1997)} & {\footnotesize{}0.808} & {\footnotesize{}67.57} & {\footnotesize{}0.811} & {\footnotesize{}50.22} & {\footnotesize{}0.894} & {\footnotesize{}20.74}\tabularnewline
 & {\footnotesize{}GL-LN} & {\footnotesize{}0.925} & {\footnotesize{}57.43} & {\footnotesize{}0.965} & {\footnotesize{}37.35} & {\footnotesize{}0.985} & {\footnotesize{}9.30}\tabularnewline
 & {\footnotesize{}GL-CR} & {\footnotesize{}0.885} & {\footnotesize{}60.05} & {\footnotesize{}0.884} & {\footnotesize{}52.63} & {\footnotesize{}0.926} & {\footnotesize{}32.61}\tabularnewline
 & {\footnotesize{}GL-CR-Iter} & {\footnotesize{}0.911} & {\footnotesize{}76.72} & {\footnotesize{}0.911} & {\footnotesize{}69.06} & {\footnotesize{}0.944} & {\footnotesize{}42.20}\tabularnewline
\cline{3-8} \cline{4-8} \cline{5-8} \cline{6-8} \cline{7-8} \cline{8-8} 
{\footnotesize{}$\lambda_{0}=0.35$} & {\footnotesize{}OLS-CR} & {\footnotesize{}0.927} & {\footnotesize{}75.58} & {\footnotesize{}0.910} & {\footnotesize{}66.20} & {\footnotesize{}0.944} & {\footnotesize{}39.15}\tabularnewline
 & {\footnotesize{}Bai (1997)} & {\footnotesize{}0.838} & {\footnotesize{}66.86} & {\footnotesize{}0.821} & {\footnotesize{}49.34} & {\footnotesize{}0.893} & {\footnotesize{}20.77}\tabularnewline
 & {\footnotesize{}GL-LN} & {\footnotesize{}0.965} & {\footnotesize{}54.57} & {\footnotesize{}0.974} & {\footnotesize{}32.88} & {\footnotesize{}0.984} & {\footnotesize{}9.39}\tabularnewline
 & {\footnotesize{}GL-CR} & {\footnotesize{}0.898} & {\footnotesize{}57.32} & {\footnotesize{}0.888} & {\footnotesize{}50.29} & {\footnotesize{}0.924} & {\footnotesize{}29.06}\tabularnewline
 & {\footnotesize{}GL-CR-Iter} & {\footnotesize{}0.930} & {\footnotesize{}75.87} & {\footnotesize{}0.913} & {\footnotesize{}66.13} & {\footnotesize{}0.944} & {\footnotesize{}38.71}\tabularnewline
\cline{3-8} \cline{4-8} \cline{5-8} \cline{6-8} \cline{7-8} \cline{8-8} 
{\footnotesize{}$\lambda_{0}=0.2$} & {\footnotesize{}OLS-CR} & {\footnotesize{}0.910} & {\footnotesize{}75.24} & {\footnotesize{}0.917} & {\footnotesize{}64.17} & {\footnotesize{}0.953} & {\footnotesize{}34.26}\tabularnewline
 & {\footnotesize{}Bai (1997)} & {\footnotesize{}0.808} & {\footnotesize{}67.03} & {\footnotesize{}0.852} & {\footnotesize{}50.40} & {\footnotesize{}0.937} & {\footnotesize{}21.76}\tabularnewline
 & {\footnotesize{}GL-LN} & {\footnotesize{}0.921} & {\footnotesize{}57.96} & {\footnotesize{}0.962} & {\footnotesize{}39.63} & {\footnotesize{}0.969} & {\footnotesize{}10.86}\tabularnewline
 & {\footnotesize{}GL-CR} & {\footnotesize{}0.912} & {\footnotesize{}56.87} & {\footnotesize{}0.909} & {\footnotesize{}48.68} & {\footnotesize{}0.932} & {\footnotesize{}23.91}\tabularnewline
 & {\footnotesize{}GL-CR-Iter} & {\footnotesize{}0.894} & {\footnotesize{}75.15} & {\footnotesize{}0.923} & {\footnotesize{}64.14} & {\footnotesize{}0.953} & {\footnotesize{}34.06}\tabularnewline
\hline 
\end{tabular}{\footnotesize\par}
\par\end{centering}
~~~~~~~~~~~~~~~~~~~~~~~~~~~~~~~%
\noindent\begin{minipage}[t]{1\columnwidth}%
{\scriptsize{}The model is $y_{t}=\delta_{1}^{0}+\delta^{0}\mathbf{1}_{\left\{ t>\left\lfloor T\lambda_{0}\right\rfloor \right\} }+e_{t},\,e_{t}=0.6e_{t-1}+u_{t},\,u_{t}\sim i.i.d.\,\mathscr{N}\left(0,\,0.49\right),\,T=100$.}%
\end{minipage}
\end{table}

\end{singlespace}

\setcounter{page}{1} \renewcommand{\thepage}{S-T-\arabic{page}}

\end{document}